\documentclass[11pt]{amsart}
\RequirePackage{amsmath,amssymb,amsthm, amscd, comment,mathtools}

\usepackage{kantlipsum}
\usepackage[all]{xy}
\usepackage{mathabx}
\usepackage{graphicx}
\usepackage{enumitem}
\usepackage{times}
\usepackage{subfig}
\usepackage{mnsymbol}
\usepackage[pagebackref,breaklinks,colorlinks,linkcolor=red,anchorcolor=red,citecolor=blue]{hyperref}

\usepackage{enumitem}

\calclayout

\newtheorem{theorem}[subsubsection]{Theorem}
\newtheorem{proposition}[subsubsection]{Proposition}
\newtheorem{lemma}[subsubsection]{Lemma}
\newtheorem{corollary}[subsubsection]{Corollary}
\newtheorem{conjecture}[subsubsection]{Conjecture}
\newtheorem{notation}[subsubsection]{Notation}
\theoremstyle{definition}
\newtheorem{definition}[subsubsection]{Definition}
\theoremstyle{remark}
\newtheorem{remark}[subsubsection]{Remark}
\newtheorem{example}[subsubsection]{Example}
\newtheorem{recall}[subsubsection]{Recall}

\numberwithin{equation}{subsubsection}

\DeclareMathAlphabet{\mathbbold}{U}{bbold}{m}{n}

\begin{document}
\title[K\"unneth formulas and additivity of traces]{K\"unneth formulas for  motives and additivity of traces}
\author{Fangzhou Jin}
\address{Fakult\"at f\"ur Mathematik,
Universit\"at Duisburg-Essen,
Thea-Leymann-Strasse 9,
45127 Essen,
Germany}
\email{fangzhou.jin@uni-due.de}

\author{Enlin Yang}
\address{School of Mathematical Sciences, Peking University, No.5 Yiheyuan Road Haidian District, Beijing, 100871, P. R. China}
\email{yangenlin@math.pku.edu.cn}

\date{\today}


\maketitle

\begin{abstract}

We prove several K\"unneth formulas in motivic homotopy categories and deduce a Verdier pairing in these categories following SGA5, which leads to the characteristic class of a constructible motive, an invariant closely related to the Euler-Poincar\'e characteristic. We prove an additivity property of the Verdier pairing using the language of derivators, following the approach of May and Groth-Ponto-Shulman; using such a result we show that in the presence of a Chow weight structure, the characteristic class for all constructible motives is uniquely characterized by proper covariance, additivity along distinguished triangles, refined Gysin morphisms and Euler classes. In the relative setting, we prove the relative K\"unneth formulas under some transversality conditions, and define the relative characteristic class.

\end{abstract}

\tableofcontents

\noindent

\section{Introduction}
\subsection{The Euler-Poincar\'e characteristic}
\subsubsection{}
The \emph{Euler-Poincar\'e characteristic} (or \emph{Euler characteristic}) is an important invariant of topological spaces in algebraic topology which gives rise to various generalizations in geometry, homological algebra and category theory. In topology, the Euler characteristic of a finite CW-complex is the alternating sum of the dimensions of its singular homology groups. In algebraic geometry, this notion is generalized for \'etale sheaves as follows: let $X$ be a separated scheme of finite type over a perfect field $k$ of characteristic $p$; if $\ell$ is a prime different from $p$ and $\mathcal{F}$ is a constructible complex of $\ell$-adic \'etale sheaves over $X$, then the Euler characteristic (with compact support)
\begin{equation}
\label{eq:EP}
\chi_c(X_{\bar{k}},\mathcal{F})=\sum_{i\geqslant0}(-1)^i\cdot\operatorname{dim} H^i_c(X_{\bar{k}},\mathcal{F})
\end{equation}
is a well-defined integer, by the finiteness theorems in \cite{SGA4.5}. 

\subsubsection{}
Morel and Voevodsky introduced motivic homotopy theory (\cite{MV}) where one study cohomology theories over algebraic varieties by means of the homotopy theory relative to the affine line $\mathbb{A}^1$, leading to several triangulated categories of motives: the stable motivic homotopy category $\mathbf{SH}$ classifies cohomology theories which satisfy $\mathbb{A}^1$-homotopy invariance, and Voevodsky's category of motivic complexes $\mathbf{DM}$ (\cite{VSF}) computes motivic cohomology. These categories are built in a style very close to the derived category of $\ell$-adic \'etale sheaves: the work of Ayoub (\cite{Ayo}) and Cisinski-D\'eglise (\cite{CD1}) establish a \emph{six functors formalism} similar to the powerful machinery in \cite{SGA4}, and the \'etale realization functor (\cite{Ayo2}, \cite{CD3}) gives a map from motives to the derived category \'etale sheaves which preserves the six functors, generalizing the cycle class map in \'etale cohomology \cite[Cycle]{SGA4.5}.

\subsubsection{}
A natural question arises to define the Euler characteristic of a motive. However, for constructible objects in the categories of motives mentioned above, the very definition with~\eqref{eq:EP} apparently does not work, since motivic cohomology groups, or equivalently Bloch's higher Chow groups (\cite{Blo}), are in general infinite-dimensional as vector spaces. Instead, there is a more categorical approach using the trace of a morphism: recall that if $\mathcal{C}$ is a symmetric monoidal category with unit $\mathbbold{1}$, $M$ is a \emph{(strongly) dualizable} object in $\mathcal{C}$ (which corresponds to \emph{locally constant} or \emph{smooth} sheaves in the \'etale setting) with dual $M^\vee$ and $u:M\xrightarrow{}M$ is an endomorphism of $M$, then the \emph{trace} of $u$ is the map
\begin{equation}
\label{eq:trace_def}
Tr(u):
\mathbbold{1}
\xrightarrow{\eta}
M^\vee\otimes M
\xrightarrow{id\otimes u}
M^\vee\otimes M
\simeq
M\otimes M^\vee
\xrightarrow{\epsilon}
\mathbbold{1}
\end{equation}
considered as an endomorphism of the unit $\mathbbold{1}$, where $\eta$ and $\epsilon$ are unit and counits of the duality. The \emph{Euler characteristic} of $M$ is defined as the trace of the identity map of $M$. If $k$ is a field, in the stable motivic homotopy category $\mathbf{SH}(k)$ the endomorphism ring of the unit $\mathbbold{1}_k$ is identified as 
\begin{equation}
End_{\mathbf{SH}(k)}(\mathbbold{1}_k)
\simeq
GW(k)
\end{equation}
where $GW(k)$ is the Grothendieck-Witt ring of $k$, that is, the Grothendieck group of non-degenerate quadratic forms over $k$. Therefore the Euler characteristic of motives in this case is an invariant in terms of quadratic forms, refining the usual integer-valued Euler characteristic.

\subsubsection{}
For example, if $f:X\to Y$ is a smooth and proper morphism, then the motive $M_Y(X)=f_\#\mathbbold{1}_X$ is dualizable (\cite{Hoy}, \cite{Lev}); the \emph{motivic Gauss-Bonnet formula} states that the Euler characteristic of $M_Y(X)$ can be computed as the degree of the (motivic) \emph{Euler class} of the tangent bundle of $f$ (\cite[Theorem 1]{Lev}, \cite[Theorem 4.6.1]{DJK}), and is a refinement of the classical Gauss-Bonnet formula (\cite[VII 4.9]{SGA5}). There are other examples of dualizable motives (\cite{Lev2}), and more generally if $k$ is a perfect field which has resolution of singularities, then every constructible object in $\mathbf{SH}(k)$ is dualizable.

\subsubsection{}
The \emph{Lefschetz trace formula} (\cite[Cycle]{SGA4.5}) plays an important role in Grothendieck's approach to the Weil conjectures via a cohomological interpretation of the $L$-functions (\cite[Rapport]{SGA4.5}); in \cite[III]{SGA5}, this formula for constant coefficients is generalized to a more general form, called the \emph{Lefschetz-Verdier formula}. In order to express this last formula, a very general cohomological pairing, called the \emph{Verdier pairing}, is constructed from several K\"unneth type formulas for \'etale sheaves and a delicate analysis on the six functors. The idea behind this construction is the formalism of \emph{Grothendieck-Verdier local duality} (\cite[4.4.23]{CD1}) which, in the setting of the six functors, gives rise to a generalized trace map (see~\eqref{eq:cc_trace_intro} below), in the way that the usual formalism of (strong) duality produces the trace map~\eqref{eq:trace_def}; the construction works not only for dualizable objects, but also for all constructible ones, which can be considered as \emph{weakly dualizable}. If $X$ is a scheme, $\mathcal{F}$ is a constructible complex of $\ell$-adic \'etale sheaves over $X$ and $u$ is an endomorphism of $\mathcal{F}$, this generalized trace for $u$ is an element in the group of global sections of the \emph{dualizing complex} over $X$, called the \emph{characteristic class} of $u$, or the characteristic class of $\mathcal{F}$ when $u$ is the identity map (\cite[Definition 2.1.1]{AS}).

\subsubsection{}
\label{num:proper_ec}
The characteristic class is closely related to the Euler characteristic: when $X$ is the spectrum of the base field $k$, the characteristic class agrees with the Euler characteristic; more generally, if $f:X\to\operatorname{Spec}(k)$ is a proper morphism, the Lefschetz-Verdier formula implies that the degree of the characteristic class of $\mathcal{F}$ agrees with the Euler characteristic of $Rf_*\mathcal{F}$.

\subsubsection{}
The main goal of this paper is to define the characteristic class for constructible motives and study its properties. Given the six functors formalism in the motivic context, analogous to the classical one, we would like to use the very definition of \cite[III]{SGA5} to define the Verdier pairing. For this, a major input is the proof of some K\"unneth formulas for motives.

\subsection{K\"unneth formulas for motives}
\subsubsection{}
In Section~\ref{chap_kun}, we prove several general K\"unneth formulas for motives that would lead to the Verdier pairing, summarized as follows:
\begin{theorem}[see Theorem~\ref{th:kun}]\label{introThmKun}
Let $f_1:X_1\to Y_1$ and $f_2:X_2\to Y_2$ be two morphisms between separated schemes of finite type over a field $k$, with the following commutative diagram
\begin{align}
\begin{gathered}
  \xymatrix{
   X_1 \ar[d]_-{f_1} & X_1\times_k X_2 \ar[l]_-{p_1} \ar[r]^-{p_2} \ar[d]_-{f} & X_2 \ar[d]^-{f_2}\\
   Y_1 & Y_1\times_k Y_2 \ar[l]^-{p_1'} \ar[r]_-{p_2'} & Y_2.
  }
  \end{gathered}
\end{align}
Let $\mathbf{T}_c$ be the category of constructible motivic spectra $\mathbf{SH}_c$ or the category of constructible $cdh$-motives $\mathbf{DM}_{cdh,c}$ (more generally, $\mathbf{T}_c$ can be the subcategory of constructible objects in a \emph{motivic triangulated category}, see Definition~\ref{def:mot_tri_cat}).

We assume resolution of singularities (by blowups or by alterations, see the condition~\ref{resol} in \ref{par:devissage} below). For $i=1,2$, consider objects $L_i\in\mathbf{T}_c(X_i)$ and $M_i,N_i\in\mathbf{T}_c(Y_i)$. Then there are canonical isomorphisms
\begin{align}
p'^*_1f_{1*}L_1\otimes p'^*_2f_{2*}L_2
&\to
f_*(p^*_1L_1\otimes p^*_2L_2);\\
\label{intro:eq:f2}p'^*_1f_{1!}L_1\otimes p'^*_2f_{2!}L_2
&\to
f_!(p^*_1L_1\otimes p^*_2L_2);\\
\label{intro:KuHom00}
p^*_1f^!_1M_1\otimes p^*_2f^!_2M_2
&\to
f^!(p'^*_1M_1\otimes p'^*_2M_2);\\
\label{intro:KuHom11}p_1'^*\underline{Hom}(M_1, N_1)\otimes p_2'^*\underline{Hom}(M_2, N_2)
&\to
\underline{Hom}(p_1'^*M_1\otimes p_2'^*M_2, p_1'^*N_1\otimes p_2'^*N_2).
\end{align}
\end{theorem}

\subsubsection{}
The proof of Theorem~\ref{th:kun} is quite different from the classical case: the main ingredient of the proof is the strong devissage property (Definition~\ref{def:dev}), which says that under resolution of singularities, the category of constructible motives is generated by (relative) Chow motives as a thick subcategory; we therefore reduce to the case of Chow motives, in which case a careful manipulation of the functors gives the desired isomorphisms.
The isomorphism \eqref{intro:eq:f2} involving $f_!$ is quite formal, and holds more generally when we replace
the base field by any base scheme, while the other ones fail to hold in general. We will see later in Section \ref{sec:KFOGB} that under some assumptions they also hold in the relative case.

\subsection{The Verdier pairing and the characteristic class}

\subsubsection{}\label{subsubsec:introKu}
In Section~\ref{section:pairing}, 
we use the K\"unneth formulas to define the Verdier pairing, following \cite[III]{SGA5}. 
Let $X_1$ and $X_2$ be two separated schemes of finite type over $k$. 
Let $c:C\to X_1\times_kX_2$ and 
$d:D\to X_1\times_kX_2$ be two morphisms. 
We denote by $E=C\times_{X_1\times_kX_2}D$. 
For $i=1,2$, we denote by $p_i:X_1\times_k X_2\to X_i$ the projections, 
 $c_i=p_i\circ c:C\to X_i$, $d_i=p_i\circ d:D\to X_i$ and let $L_i\in\mathbf{T}_c(X_i)$. 
Then given two maps $u:c_1^*L_1\to c_2^!L_2$ and $v:d_2^*L_2\to d_1^!L_1$, 
the Verdier pairing $\langle u,v\rangle$ is an element of the bivariant group (or Borel-Moore theory group) $H_0(E/k)$ (see Definition~\ref{def:cc}) seen as a map
\begin{align}
\langle u,v\rangle:\mathbbold{1}_E\to\mathcal{K}_E
\end{align}
where $\mathcal{K}_E=\mathbb{D}(\mathbbold{1}_E)$ is the \emph{dualizing object}
(Definition~\ref{def:verdier_pairing}). The Lefschetz-Verdier formula (Proposition~\ref{proper_pf}) states that this pairing is compatible with proper direct images. 
In Proposition~\ref{prop:bil_tr} we show that this pairing can always be reduced to a generalized trace map, which we explicitly identify in Proposition~\ref{prop:pair_trace}. 

\subsubsection{}
Let $X$ be a scheme, $M$ be a motive over $X$ and $u:M\to M$ be an endomorphism of $M$.
We define the \emph{characteristic class}
 $C_X(M,u):=\langle u,1_M\rangle$ as a particular case of the Verdier pairing (Definition~\ref{def:cc}).  Explicitly, the characteristic class is the composition
 \begin{align}
 \label{eq:cc_trace_intro}
\mathbbold{1}_X
\xrightarrow{u}
\underline{Hom}(M,M)
\xrightarrow{}
\mathbb{D}(M)\otimes M
\simeq
M\otimes\mathbb{D}(M)
\xrightarrow{\epsilon_M}
\mathcal{K}_X
\end{align}
where the second map is deduced from the K\"unneth formulas.
We denote $C_X(M)=C_X(M,1_M)$. The bivariant group $H_0(X/k)$ in which it lives can be computed in many cases:
\begin{itemize}
\item If $M$ is a constructible $cdh$-motive in $\mathbf{DM}_{cdh,c}(X,\mathbb{Z}[1/p])$, the characteristic class $C_X(M)$ is a $0$-cycle in the Chow group $CH_0(X)[1/p]=CH_0(X)\otimes_{\mathbb{Z}}\mathbb{Z}[1/p]$ of $X$ up to $p$-torsion.

\item If $M$ is a constructible element in the homotopy category of $\mathbf{KGL}$-modules over $X$, the characteristic class $C_X(M)$ is an element in the $0$-th algebraic $G$-theory group $G_0(X)$.

\item If $M$ is a constructible motivic spectrum in $\mathbf{SH}_c(X)$ and when we apply the $\mathbb{A}^1$-regulator map with values in the Milnor-Witt spectrum (\cite[Example 4.4.6]{DJK}), then the Milnor-Witt-valued characteristic class $C_X^{MW}(M)$ is an element in the Chow-Witt group $\widetilde{CH}_0(X)$.
\end{itemize}
In other words, the characteristic class associates to every constructible motive a concrete object, which is realized as either a $0$-cycle, a formal sum of coherent sheaves or a Milnor-Witt $0$-cycle. It lifts the $\ell$-adic characteristic class to the cycle-theoretic level, therefore giving an illustration of the general philosophy of mixed motives.

\subsubsection{} 
The results in Section~\ref{section:pairing} are rather transcriptions of classical results in our context, but will be important for the next sections.
A particular case of the construction is already given in \cite{Ols}. For $h$-motives in $\mathbf{DM}_{h}$, the characteristic class is also defined independently in \cite{Cis}.
Note that since the \'etale realization functor is compatible with the six functors, our constructions are compatible with the ones in \cite{AS} and \cite[III]{SGA5}. 

\subsection{Additivity of traces}

\subsubsection{}
Given a distinguished triangle $L\to M\to N\to L[1]$ of motives, one can naturally ask about the relations between the characteristic classes of $L$, $M$ and $N$.
It is well known that for monoidal triangulated categories, the trace map fails to be additive along distinguished triangles in general (\cite{Fer}). This failure may be explained by the defect of the axioms of a triangulated category, where the cone of a morphism exists uniquely only up to isomorphism but not up to \emph{unique} isomorphism; more concretely, a commutative diagram between distinguished triangles in the derived category does not always reflect a commutative diagram in the category of complexes, which would mean to be ``truly commutative''. Behind such a phenomena lies the idea of higher category theory as illustrated by the vast theory of $(\infty,1)$-categories (\cite{HTT}).

\subsubsection{}
A landmarking breakthrough in this direction is made by \cite{May}, where it is shown that the trace map is additive along distinguished triangles for triangulated categories satisfying some extra axioms that arise naturally from topology; with the same spirit, the additivity of traces is generalized to stable derivators in \cite{GPS}, where it is shown that the additivity holds for an endomorphism of distinguished triangles if the diagram commutes in the strong sense of derivators. A derivator, in some sense, lies between a $1$-category and an $(\infty,1)$-category, in which one can define left and right homotopy Kan extensions using only the $1$-categorical language, and which carries enough information to characterize homotopy limits and colimits by $1$-categorical universal properties. In particular the axioms of a stable derivator produce functorial cone objects, fixing the problem above for triangulated categories.

\subsubsection{}
We use a very similar approach for the generalized trace map: in Section~\ref{section:additivity}, we prove the additivity of the characteristic class using the language of derivators in the motivic setting (\cite[Section 2.4.5]{Ayo}). 
Using the same notations as \ref{subsubsec:introKu}, the main result is the following additivity for the Verdier pairing:
\begin{theorem}[see Theorem~\ref{th:add_trace}]
Let $\mathcal{T}_c$ be a constructible motivic derivator (see Definition~\ref{def:cons_mot_der}) whose underlying motivic triangulated category $\mathbf{T}$ satisfies resolution of singularities (see condition~\ref{resol} in \ref{par:devissage}). For $i\in\{1,2\}$, let
\begin{align}
\begin{gathered}
  \xymatrix{
    L_i \ar^-{}[r] \ar_-{}[d] \ar@{}[rd]|{\Gamma_i} & M_i \ar^-{}[d]\\
    \ast \ar_-{}[r] & N_i
  }
\end{gathered}
\end{align}
be a coherent biCartesian square in $\mathcal{T}_c(X_i,\Box)$. Let $f:c_1^*\Gamma_1\to c_2^!\Gamma_2$ and $g:d_2^*\Gamma_2\to d_1^!\Gamma_1$ be morphisms of coherent squares in $\mathcal{T}_c(C,\Box)$ and $\mathcal{T}_c(D,\Box)$. 
Then the Verdier pairing satisfies
\begin{align}
\langle f_M ,g_M \rangle=\langle f_L ,g_L \rangle+\langle f_N ,g_N \rangle
\end{align}
where $f_M:c_1^*M_1\to c_2^!M_2$ is the restriction of $f$, and similarly for the other maps.
\end{theorem}
\subsubsection{}The above result corresponds to \cite[III (4.13.1)]{SGA5} which claims the additivity of the Verdier pairing in the filtered derived category. The strategy of the proof is to first use Proposition~\ref{prop:bil_tr} to reduce the pairing to a generalized trace map with one single entry. Then we follow closely the same steps of proof as in \cite{GPS}, where we need to check the same axioms for the local duality functor instead of the usual duality functor. Note that all the usual examples such as $\mathbf{SH}_c$ or $\mathbf{DM}_{cdh,c}$ arise from constructible motivic derivators, so working with derivators is not a restriction in practice.

\subsection{A characterization of the characteristic class of a motive}

\subsubsection{}
There has been an extensive study in the literature around the Euler characteristic of \'etale sheaves via ramification theory, see for example \cite{AS}, \cite{KS} and \cite{Sai}. In this paper, we use a different approach to give a description of the characteristic class for $cdh$-motives in $\mathbf{DM}_{cdh,c}$. 
In Section~\ref{section:CC} we start with the study of some elementary properties of the characteristic class, using the (Fulton-style) intersection theory developed in \cite{DJK}. The main result is the following characterization of the characteristic class for $cdh$-motives over a perfect field:
\begin{theorem}[see Theorem~\ref{th:uniqueness_cc}]
Assume that the base field $k$ is perfect, and let $X$ be a scheme. Then the map
\begin{align}
\begin{split}
\mathbf{DM}_{cdh,c}(X,\mathbb{Z}[1/p])&\to CH_0(X)[1/p]\\
M&\mapsto C_X(M)
\end{split}
\end{align}
is the unique map satisfying the following properties:
\begin{enumerate}
\item For any distinguished triangle $L\to M\to N\to L[1]$ in $\mathbf{DM}_{cdh,c}(X)$, $C_X(M)=C_X(L)+C_X(N)$.
\item Let $f:Y\to X$ be a proper morphism with $Y$ smooth of dimension $d$ over $k$ and let $M$ be the direct summand of the Chow motive $f_*\mathbbold{1}_Y(n)$ defined by an endomorphism $u$. Then $u$ is identified as a cycle $u'\in CH_d(Y\times_XY)[1/p]$, and we have 
\begin{align}
C_X(M)=C_X(f_*\mathbbold{1}_Y(n),u)=f_*\Delta^!u'\in CH_0(X)[1/p]
\end{align}
where $f_*:CH_0(Y)[1/p]\to CH_0(X)[1/p]$ is the proper push-forward and $\Delta^!:CH_d(Y\times_XY)[1/p]\to CH_0(Y)[1/p]$ is the refined Gysin morphism (\cite[6.2]{Ful}) associated to the Cartesian square
\begin{align}
\begin{gathered}
  \xymatrix{
    Y \ar[r]^-{\delta_{Y/X}} \ar@{=}[d] \ar@{}[rd]|{\Delta} & Y\times_XY \ar[d]^-{}\\
    Y \ar[r]_-{\delta_{Y/k}} & Y\times_kY.
  }
\end{gathered}
\end{align}
There is an alternative description using the Euler class (i.e. top Chern class), see Proposition~\ref{prop:cc_endo_euler} below.
\end{enumerate}
\end{theorem} 
\subsubsection{} 
The idea is as follows: Bondarko's theory of weight structures (\cite{Bon}, \cite{BI}) implies that $\mathbf{DM}_{cdh,c}$ is generated by Chow motives not only as a thick triangulated category but also as a triangulated category, and therefore by additivity of traces, it suffices to compute the characteristic class for Chow motives, which can be achieved using intersection theory. This description also holds when we replace $\mathbf{DM}_{cdh,c}$ by homotopy category of $\mathbf{KGL}$-modules, since the Chow weight structure also exists by the results of \cite{BL}. In general, the characterization holds over the sub-triangulated category generated by direct summands of Chow motives.

\subsubsection{} 
While our result gives an abstract characterization for the characteristic class, we expect it to be related with the Grothendieck-Ogg-Shafarevich type results in \cite{AS}.
When the base field is not perfect, there is a similar characterization by perfection using the work of \cite{EK}, see Remark~\ref{rk:cc_char_perfection} below.

\subsubsection{} 
In Section~\ref{section:cc_rr} we show the compatibility between the characteristic class and Riemann-Roch transformations. If $X$ is a scheme and $M\in \mathbf{SH}_c(X)$, then we can canonically associate to $M$ a constructible element in the homotopy category of $\mathbf{KGL}$-modules over $X$, as well as an element in $\mathbf{DM}_{cdh,c}(X)$ (see~\ref{section:rr_eva}). Then the Riemann-Roch transformation
$$
\tau_X:G_0(X)\to\oplus_{i\in\mathbb{Z}}CH_i(X)_{\mathbb{Q}}
$$
constructed in~\cite[Theorem 18.3]{Ful} sends the characteristic class of the former to that of the latter. In Corollary~\ref{cor:rr_trace} we prove a more general version of such a result.

\subsection{The relative case}
\subsubsection{}
In Section \ref{sec:KFOGB} we prove some relative K\"unneth formulas, following the approach in~\cite{YZ18}. We first introduce the transversality conditions (Definition~\ref{def:Ftransversal}), which are closely related to the notion of purity in \cite{DJK} (see~\ref{recall_pur_trans}); instead of making use of the geometric notion of singular support as in \cite{YZ18}, our definition is a more categorical one extracted from the spirit of \cite{Sai}.
We show that under such conditions and some smoothness assumptions, the K\"unneth formulas \eqref{intro:KuHom00} and \eqref{intro:KuHom11} still holds over a general base scheme (see Theorem \ref{th:kun_rel}). The proof uses the K\"unneth formulas over a field in Section~\ref{chap_kun}. 
As a special case, we obtain the following result:
\begin{corollary}(see Corollary~\ref{cor:SSKunn})
Let $\mathbf{T}_c$ be the subcategory of constructible objects in a motivic triangulated category which satisfies resolution of singularities (see condition~\ref{resol} in \ref{par:devissage}).
Let $S$ be a smooth $k$-scheme, let $\pi\colon X\to S$ be a smooth morphism and let $F\in \mathbf{T}_c(X)$. 
If $\pi$ is universally $F$-transversal (see Definition \ref{def:Ftransversal} below), then there is a canonical isomorphism
\begin{align}
p_1^\ast F\otimes p^\ast_2\underline{Hom}(F, \pi^!\mathbbold{1}_S)
\xrightarrow{\sim}
\underline{Hom}(p_2^\ast F, p_1^! F)
\end{align}
where $p_i\colon X\times_S X\to X$ is the projection for $i=1,2$.
\end{corollary}
\subsubsection{}
These K\"unneth formulas are sufficient to define the relative Verdier pairing, as well as the relative characteristic class (Definition~\ref{def:cc_rel}), in the same way as the absolute case. 
In the case of $\mathbf{DM}_{cdh}$, if $S$ is a smooth scheme of dimension $n$, then the relative characteristic class is given by a $n$-cycle up to $p$-torsion which, via Fulton's specialization of cycles (\cite[Section 10.1]{Ful}), specializes to the $0$-cycles given by the characteristic class of its fibers. We prove a more general version in Proposition~\ref{prop:cc_rel_spec}. 

\subsubsection{}
In Section~\ref{section:purity_la_trans} we establish an equivalence between several notions of local acyclicity and transversality conditions (Proposition~\ref{prop:BGB2} and Proposition~\ref{prop:la_fib}). We also give an application using the Fulton style specialization map in \cite[4.5.6]{DJK} (Corollary~\ref{cor:sp_split}).

\subsubsection*{\bf Acknowledgments}This work is inspired by a project initiated by Denis-Charles Cisinski about the interaction between motives and ramification theory, and we would like to thank him for helpful discussions. The first named author would like to thank Fr\'ed\'eric D\'eglise and Adeel Khan for the collaboration in \cite{DJK} in which we constructed the intersection-theoretic tools in motivic homotopy theory that are used in this paper. We would like to thank Marc Levine for suggesting the formulation of Theorem~\ref{th:add_trace}. We would like to thank Lie Fu for telling us the trick in Lemma~\ref{Kunneth_formula}. We would like to thank Oliver R\"ondigs and Jakob Scholbach for helpful comments on a preliminary version. The first named author is partially supported by the DFG Priority Programme SPP 1786. 
The second named author is supported by Peking University's Starting Grant Nr.7101302006.
Part of this work is done while both authors were members of SFB 1085 Higher Invariants at Universit\"at Regensburg.

\subsubsection*{\bf Notation and Conventions}

\begin{enumerate}
\item Throughout the paper, we denote by $k$ a field, and a scheme stands for a separated scheme of finite type over $k$. The category of schemes is denoted by $Sch$.

\item For any pair of adjoint functors $(F,G)$ between two categories, we denote by $ad_{(F,G)}:1\to GF$ and $ad'_{(F,G)}:GF\to 1$ the unit and conuit maps of the adjunction.

\item We say that a morphism of schemes $f:X\to Y$ is \emph{local complete intersection} (abbreviated as ``lci'') if it factors as the composition of a regular closed immersion followed by a smooth morphism. 
\footnote{This notion is called ``smoothable lci'' in \cite{DJK}.}
We denote by $L_f$ or $L_{X/Y}$ its virtual tangent bundle in $K_0(X)$.

\item If $A,B$ are objects in a closed symmetric monoidal category, we denote by
\begin{align}
\eta_A:\mathbbold{1}
\xrightarrow{}
\underline{Hom}(A,A)
\end{align}
\begin{align}
\epsilon_A:A\otimes\underline{Hom}(A,B)
\xrightarrow{}
B
\end{align}
the unit and counit maps of the monoidal structure.
\end{enumerate}

\section{K\"unneth formulas for motives}
\label{chap_kun}
In this section we prove several K\"unneth formulas for motives whose analogues for $\ell$-adic \'etale sheaves are proven in \cite{SGA4.5} and \cite{SGA5}. We start by recalling the axioms of a \emph{motivic triangulated category} in the sense of \cite[Definition 2.4.45]{CD1}. We denote by $SMTR$ the $2$-category of symmetric monoidal triangulated categories with (strong) monoidal functors.

\begin{definition}
\label{def:mot_tri_cat}
A \textbf{motivic triangulated category} is a (non-strict) $2$-functor $(\mathbf{T},\otimes):Sch^{op}\to SMTR$ satisfying the following properties:
\begin{enumerate}
\item The value of $\mathbf{T}$ at the empty scheme $\mathbf{T}(\emptyset)$ is the zero category.
\item For every morphism of schemes $f:Y\to X$, the functor $f^*:\mathbf{T}(X)\to\mathbf{T}(Y)$ is monoidal.
\item For every smooth morphism $f:Y\to X$, the functor $f^*:\mathbf{T}(X)\to\mathbf{T}(Y)$ has a left adjoint $f_\#$, such that
\begin{enumerate}
\item For any commutative square of schemes
\begin{equation}
\begin{gathered}
  \xymatrix{
    Y \ar[r]^-{q} \ar[d]_-{g} \ar@{}[rd]|{\Delta} & X \ar[d]^-{f}\\
    T \ar[r]_-{p} & S
  }
\end{gathered}
\end{equation}
The following natural transformation is an isomorphism:
\begin{equation}
q_\#g^*
\xrightarrow{ad_{(p_\#,p^*)}}
q_\#g^*p^*p_\#
=
q_\#q^*f^*p_\#
\xrightarrow{ad'_{(q_\#,q^*)}}
f^*p_\#.
\end{equation}

\item For any smooth morphism $f:Y\to X$ and any objects $M\in\mathbf{T}(Y)$, $N\in\mathbf{T}(X)$, the following transformation is an isomorphism:
\begin{equation}
f_\#(M\otimes f^*N)
\xrightarrow{ad_{(f_\#,f^*)}}
f_\#(f^*f_\#M\otimes f^*N)
\simeq
f_\#f^*(f_\#M\otimes N)
\xrightarrow{ad'_{(f_\#,f^*)}}
f_\#M\otimes N.
\end{equation}

\end{enumerate}

\item For every morphism of schemes $f:Y\to X$, the functor $f^*:\mathbf{T}(X)\to\mathbf{T}(Y)$ has a right adjoint $f_*$.

\item For any scheme $S$ with $p:\mathbb{A}^1_S\to S$ the canonical projection, the unit map $1\xrightarrow{ad_{(p^*,p_*)}}p_*p^*$ is an isomorphism.
\item For every proper morphism $f:Y\to X$, the functor $f_*:\mathbf{T}(Y)\to\mathbf{T}(X)$ has a right adjoint $f^!$.
\item For every smooth morphism $f:X\to S$ with a section $s:S\to X$, the functor $f_\#s_*:\mathbf{T}(S)\to\mathbf{T}(S)$ is an equivalence of categories.
\item For any closed immersion $i:Z\to X$ with open complement $j:U\to X$, the pair of functors $(j^*,i^*)$ is conservative, and the counit map $i^*i_*\xrightarrow{ad_{(i^*,i_*)}}1$ is an isomorphism.
\end{enumerate}

\end{definition}

In other words, for any scheme $X$ we have a triangulated category $\mathbf{T}(X)$, and such a formation satisfies the six functors formalism by \cite[Theorem 2.4.50]{CD1}. Examples are given by the stable motivic homotopy category $\mathbf{SH}$, the category of $cdh$-motivic complexes $\mathbf{DM}_{cdh}$ or the category of modules over a motivic ring spectrum. We gradually recall the formal properties we need in this section.

For any scheme $X$, we denote by $\mathbbold{1}_X\in\mathbf{T}(X)$ the unit object. We are mainly interested in constructible motives, see the discussion in~\ref{par:devissage} below.

\subsection{Local acyclicity and K\"unneth formula for $f_*$}
In this section, we work with a motivic triangulated category $\mathbf{T}$.

\subsubsection{}
For any commutative square of schemes
\begin{equation}\label{Cart_diag}
\begin{gathered}
  \xymatrix{
    Y \ar[r]^-{q} \ar[d]_-{g} \ar@{}[rd]|{\Delta} & X \ar[d]^-{f}\\
    T \ar[r]_-{p} & S
  }
\end{gathered}
\end{equation}
there is a canonical natural transformation 
\begin{equation}
\label{Ex**}
f^*p_*
\xrightarrow{ad_{(q^*,q_*)}}
q_*q^*f^*p_*
=
q_*g^*p^*p_*
\xrightarrow{ad'_{(p^*,p_*)}}
q_*g^*
\end{equation}
which is compatible with horizontal and vertical compositions of squares. 

\subsubsection{}
The map~\eqref{Ex**} is an isomorphism if $f$ is smooth, or if $p$ is proper. More generally, we give the following definition:
\begin{definition}
Let $f:X\rightarrow S$ be a morphism of schemes. We say that the category $\mathbf{T}$ satisfies \textbf{$f$-base change} if for any morphism $p:T\to S$ with a Cartesian square $\Delta$ as in~\eqref{Cart_diag}, the map~\eqref{Ex**} is an isomorphism. If $\mathcal{S}$ is a class of morphisms, we say that $\mathbf{T}$ satisfies $\mathcal{S}$-base change if it satisfies $f$-base change for any $f\in\mathcal{S}$. 
\end{definition}

\subsubsection{}
For any morphism $q:Y\to X$ and objects $K\in\mathbf{T}(X)$, $K'\in\mathbf{T}(Y)$, there is a canonical natural transformation
\begin{equation}
\label{Ex_*tens}
K\otimes q_*K'
\xrightarrow{ad_{(q^*,q_*)}}
q_*q^*(K\otimes q_*K')
\simeq
q_*(q^*K\otimes q^*q_*K')
\xrightarrow{ad'_{(q^*,q_*)}}
q_*(q^*K\otimes K').
\end{equation}
The map~\eqref{Ex_*tens} is an isomorphism for any proper morphism $q$. 

\subsubsection{}
For any commutative square $\Delta$ as in~\eqref{Cart_diag}, 
any $K\in\mathbf{T}(X)$ and any $L\in\mathbf{T}(T)$, 
there is a canonical natural transformation
\begin{equation}
\label{ExDelta**}
Ex(\Delta^*_*,\otimes):
K\otimes f^*p_*L
\to 
q_*(q^*K\otimes g^*L)
\end{equation}
defined as the composition
\begin{align}
K\otimes f^*p_*L
\xrightarrow{\eqref{Ex**}}
K\otimes q_*g^*L
\xrightarrow{\eqref{Ex_*tens}}
q_*(q^*K\otimes g^*L).
\end{align}

\subsubsection{}
For any Cartesian square $\Delta$ as in~\eqref{Cart_diag} with $p$ proper, 
the map~\eqref{ExDelta**} is an isomorphism. 
Our aim is to study the cases where the map~\eqref{ExDelta**} is an isomorphism 
for arbitrary $p$.
The following definition is inspired by \cite[Th. finitude, D\'efinition 2.12]{SGA4.5}:
\begin{definition}
\label{def:loc_acyclic}
Let $f:X\to S$ be a morphism of schemes and $K\in\mathbf{T}(X)$. We say that $f$ is \textbf{strongly locally acyclic} relatively to $K$ if for any morphism $p:T\to S$ with a Cartesian square $\Delta$ as in~\eqref{Cart_diag} and any object $L\in\mathbf{T}(T)$, the map~\eqref{ExDelta**} is an isomorphism. We say that $f$ is \textbf{universally strongly locally acyclic} relatively to $K$ if for any Cartesian square
\begin{align}
\begin{gathered}
  \xymatrix{
    X' \ar[r]^-{\phi} \ar[d]_-{f'} & X \ar[d]^-{f}\\
    S' \ar[r]^-{} & S
  }
\end{gathered}
\end{align}
the base change $f'$ of $f$ is strongly locally acyclic relatively to $K_{|X'}=\phi^*K$.

\end{definition}

\begin{lemma}
\label{ula}
Let $\mathcal{S}$ is a class of morphisms which is stable by base change and suppose that $\mathbf{T}$ satisfies $\mathcal{S}$-base change.
Let $f:X\to S$ be a morphism, $\phi:W\to X$ be a proper morphism such that the composition $f\circ\phi:W\to S$ lies in $\mathcal{S}$. Then $f$ is universally strongly locally acyclic relatively to the object $\phi_*\mathbbold{1}_W$. 
\end{lemma}

\proof

Consider the following commutative diagram with    Cartesian squares
\begin{align}
\begin{gathered}
  \xymatrix{
 V \ar[r]^-{r} \ar[d]_-{\psi} & W' \ar[r]^-{} \ar[d]^-{\phi'} & W \ar[d]^-{\phi}\\
 Y \ar[r]^-{q} \ar[d]_-{g} \ar@{}[rd]|{\Delta} & X' \ar[r]^-{\xi} \ar[d]^-{f'} & X \ar[d]^-{f}\\
 T \ar[r]_-{p} & S' \ar[r]^-{} & S
  }
\end{gathered}
\end{align}
and let $L\in\mathbf{T}(T)$. We want to show that the map
\begin{equation}
\label{ex_ula}
Ex(\Delta^*_*,\otimes):
\xi^*\phi_*\mathbbold{1}_W\otimes f'^*p_*L
\to 
q_*(q^*\xi^*\phi_*\mathbbold{1}_W\otimes g^*L)
\end{equation}
is an isomorphism. On the left hand side, by assumptions we have
\begin{equation}
\label{lhs_ula}
\xi^*\phi_*\mathbbold{1}_W\otimes f'^*p_*L
\simeq
\phi'_*\mathbbold{1}_{W'}\otimes f'^*p_*L
\simeq
\phi'_*\phi'^*f'^*p_*L
\simeq
\phi'_*r_*\psi^*g^*L
\end{equation}
where we use the fact that the morphism $f'\circ\phi'$ lies in $\mathcal{S}$. For the right hand side of \eqref{ex_ula}, we have
\begin{equation}
\label{rhs_ula}
q_*(q^*\xi^*\phi_*\mathbbold{1}_W\otimes g^*L)
\simeq
q_*(\psi_*\mathbbold{1}_V\otimes g^*L)
\simeq
q_*\psi_*\psi^*g^*L
\simeq
\phi'_*r_*\psi^*g^*L.
\end{equation}
It is not hard to check that the composition of~\eqref{lhs_ula} and the inverse of~\eqref{rhs_ula} agrees with the map~\eqref{ex_ula}, and therefore~\eqref{ex_ula} is an isomorphism.
\endproof

\subsubsection{}
The category $\mathbf{T}$ has a family of \emph{Tate twists} $\mathbbold{1}(n)$, which are $\otimes$-invertible objects that form a Cartesian section. By \cite[3.2]{FHM}, we have a canonical isomorphism $f_*(M\otimes\mathbbold{1}(n))\simeq f_*(M)\otimes\mathbbold{1}(n)$, and therefore Tate twists commute with both $f^*$ and $f_*$ in a canonical way. More generally, all the six functors commute with Tate twists (\cite[Section 1.1.d]{CD1}) via canonical isomorphisms, and therefore it is safe to ignore them in the proof of K\"unneth formulas.

\begin{definition}
\label{def:dev}

\begin{enumerate}
\item For any scheme $X$, a \textbf{projective motive} in $\mathbf{T}(X)$ is an object of the form $\varphi_*\mathbbold{1}_W(n)$, where $\varphi:W\to X$ is a projective morphism, and a \textbf{primitive Chow motive} is a projective motive with $W$ smooth over a finite purely inseperable extension of $k$.
\footnote{In \cite{Jin} it is shown that for $\mathbf{T}=\mathbf{DM}_{cdh,c}$ and $X$ quasi-projective over a perfect field, the idempotent completion of the additive subcategory generated by primitive Chow motives is equivalent to the category of relative Chow motives over $X$ defined by Corti and Hanamura, whence the terminology.}
\item We say that $\mathbf{T}$ satisfies
 \textbf{weak devissage} if for any scheme $X$, 
 the category $\mathbf{T}(X)$ agrees with the 
 smallest thick triangulated subcategory which is stable by 
 direct sums and contains all projective motives.

\item We say that $\mathbf{T}$ satisfies \textbf{strong devissage} if for any scheme $X$, the category $\mathbf{T}(X)$ agrees with the smallest thick triangulated subcategory which is stable by direct sums and contains all primitive Chow motives.

\item
We denote by $\mathcal{S}_k$ the family of morphisms 
$p_2:X\times_kY\to Y$ where $X$, $Y$ are 
schemes and $p_2$ is the projection onto the second factor. 
We say that $\mathbf{T}$ satisfies \textbf{$\mathcal{S}_k$-strong local acyclicity} if 
any morphism of the form $f:X\to k$ is universally strongly locally acyclic 
relatively to any object in $\mathbf{T}(X)$.
\end{enumerate}
\end{definition}

We denote by $\mathcal{Rad}$ the family of finite surjective radicial morphisms, namely the family of universal homeomorphisms.
\begin{proposition}
\label{SD_ula}
We suppose that $\mathbf{T}$ satisfies 
weak devissage. Then
\begin{enumerate}
\item \label{ula=bc}
$\mathbf{T}$ satisfies $\mathcal{S}_k$-strong local acyclicity if and only if it satisfies $\mathcal{S}_k$-base change.

\item If $\mathbf{T}$ satisfies 
strong devissage and one of the following conditions hold:
\begin{enumerate}
\item $k$ is perfect;
\item $\mathbf{T}$ satisfies $\mathcal{Rad}$-base change.
\end{enumerate}
Then $\mathbf{T}$ satisfies $\mathcal{S}_k$-strong local acyclicity. 
\end{enumerate}
\end{proposition}
\proof
\begin{enumerate}
\item The $\mathcal{S}_k$-base change property is a particular case of $\mathcal{S}_k$-strong local acyclicity since strong local acyclicity relative to the unit object is equivalent to base change property. The other direction follows from weak devissage by applying Lemma~\ref{ula}.
\item Since 
strong local acyclicity is stable under distinguished triangles, direct summands, direct sums and Tate twists, the result is straightforward from Lemma~\ref{ula}.
\end{enumerate}
\endproof

\subsubsection{}
\label{par:devissage}
Following \cite[2.4.1]{BD}, we consider the following conditions on resolution of singularities: 
\begin{enumerate}[label=(RS \arabic*)]
    \item\label{resol1} The field $k$ is perfect, over which the \emph{strong resolution of singularities} holds, in the sense that
    \begin{enumerate}
    \item For every separated integral scheme $X$ of finite type over $k$, there exists a proper birational surjective morphism $X'\to X$ with $X'$ regular;
    \item For every separated integral regular scheme $X$ of finite type over $k$ and every nowhere dense closed subscheme $Z$ of $X$, there exists a proper birational surjective morphism $b:X'\to X$ such that $X'$ is regular, $b$ induces an isomorphism $b^{-1}(X-Z)\simeq X-Z$, and $b^{-1}(Z)$ is a strict normal crossing divisor in $X'$.
    \end{enumerate}
    \item\label{resol2} The category $\mathbf{T}$ is $\mathbb{Z}[1/p]$ linear where $p$ is the characteristic exponent of $k$, and there exists a premotivic adjunction $\mathbf{SH}\rightleftharpoons\mathbf{T}$.
\end{enumerate}
\begin{enumerate}[label=(RS)]
    \item\label{resol} We say that the category $\mathbf{T}$ satisfies \ref{resol} if it satisfies one of the above conditions \ref{resol1} and \ref{resol2}.
\end{enumerate}

\subsubsection{}
We recall the following facts about devissage:
\begin{enumerate}
\item (\cite[Lemme 2.2.23]{Ayo}) Any motivic triangulated category satisfies weak devissage.

\item (\cite[Corollary 2.4.8]{BD}, \cite[Proposition 3.1.3]{EK}) If $\mathbf{T}$ is a motivic triangulated category which satisfies the condition~\ref{resol}, then it satisfies strong devissage.
\end{enumerate}
By \cite[Remark 2.1.13]{EK}, if $\mathbf{T}$ satisfies the condition~\ref{resol2} above, then $\mathbf{T}$ satisfies $\mathcal{Rad}$-base change.
As a consequence, Proposition~\ref{SD_ula} implies that 
\begin{corollary}
If $\mathbf{T}$ is a motivic triangulated category which satisfies~\ref{resol} in~\ref{par:devissage}, 
then it satisfies $\mathcal{S}_k$-strong local acyclicity (or equivalently $\mathcal{S}_k$-base change).
\end{corollary}

\subsubsection{}
For any scheme $X$, we denote by $\mathbf{T}_c(X)$ the subcategory of \emph{constructible} objects, which is the thick triangulated subcategory of $\mathbf{T}(X)$ generated by elements of the form $f_\#\mathbbold{1}_Y(n)$, where $f:Y\to X$ is a smooth morphism (\cite[Definition 4.2.1]{CD1}). 
 By \cite[Theorem 6.4]{CD2}, if $\mathbf{T}$ satisfies~\ref{resol},  
 the six functors preserve constructible objects.
The devissage condition can be translated as follows:
\begin{enumerate}[label=(\arabic*')]
\item If $\mathbf{T}$ is motivic, then for any scheme $X$, the category $\mathbf{T}_c(X)$ agrees with the smallest thick triangulated subcategory which contains all projective motives;

\item If  $\mathbf{T}$ is motivic and satisfies the condition~\ref{resol} in 
\ref{par:devissage}, then for any scheme $X$, the category $\mathbf{T}_c(X)$ agrees
 with the smallest thick triangulated subcategory which contains all primitive 
 Chow motives.

\end{enumerate}

\begin{remark}
\begin{enumerate}
\item $\mathcal{S}_k$-strong local acyclicity can be generalized to quasi-compact quasi-separated schemes over a field by a passing to the limit argument $($\cite[Th. finitude Corollaire 2.16]{SGA4.5}, \cite[Appendix C]{Hoy}$)$.

\item 
In the derived category of \'etale sheaves, $\mathcal{S}_k$-base change is a particular case of Deligne's generic base change theorem (\cite[Th. finitude, Th\'eor\`eme 2.13]{SGA4.5}). This theorem is proved for $h$-motives in \cite{Cis} using a similar method.
\item 
Under the assumption~\ref{resol}, it follows from strong devissage that every object in $\mathbf{T}_c(k)$ is (strongly) dualizable, which is a well-known result (see for example \cite[Section 3]{Hoy} or \cite[Theorem 3.2.1]{EK}).
\end{enumerate}
\end{remark}

The following notation will be used repeatedly in the study of K\"unneth formulas:
\begin{notation}
\label{Kun_not}
Let $S$ be a scheme and let $f_1:X_1\to Y_1$, $f_2:X_2\to Y_2$ be two $S$-morphisms. Denote by $p_i:X_1\times_S X_2\to X_i$, $p'_i:Y_1\times_S Y_2\to Y_i$ the projections, and $f_1\times_S f_2:X_1\times_S X_2\to Y_1\times_S Y_2$ the fiber product. We have the following commutative diagram
\begin{align}
\begin{gathered}
  \xymatrix{
   X_1 \ar[d]_-{f_1} \ar@{}[rd]|{\Delta_1} & X_1\times_S X_2 \ar[l]_-{p_1} \ar[r]^-{p_2} \ar[d]_-{f} \ar@{}[rd]|{\Delta_2} & X_2 \ar[d]^-{f_2}\\
   Y_1 & Y_1\times_S Y_2 \ar[l]^-{p_1'} \ar[r]_-{p_2'} & Y_2.
  }
\end{gathered}
\end{align}
For $K_1\in {\mathbf T}(X_1)$ and $K_2\in {\mathbf T}(X_2)$, we denote
\begin{align}
K_1\boxtimes_S K_2\coloneqq p_1^\ast K_1\otimes p_2^\ast K_2 
\end{align}
which is an object of $\mathbf T(X_1\times_S X_2)$.
\end{notation}

\subsubsection{}
The first K\"unneth formula is related to the functor $f_*$. 
For any morphism $f:X\to S$ and any $A,B\in\mathbf{T}(X)$, there is a canonical map
\begin{equation}
\label{eq:Ex_lower*_tensor}
f_*A\otimes f_*B\to f_*(A\otimes B)
\end{equation}
defined as the composition
\begin{align}
\begin{split}
&f_*A\otimes f_*B
\xrightarrow{ad_{(f^*,f_*)}}
f_*f^*(f_*A\otimes f_*B)\\
=
&f_*(f^*f_*A\otimes f^*f_*B)\xrightarrow{ad'_{(f^*,f_*)}\otimes ad'_{(f^*,f_*)}}
f_*(A\otimes B).
\end{split}
\end{align}

\subsubsection{}
For $i=1,2$, let $L_i$ be an element of $\mathbf{T}(X_i)$. We have a canonical map
\begin{equation}
\label{Kunneth_*}
Kun_*:
p'^*_1f_{1*}L_1\otimes p'^*_2f_{2*}L_2
\to
f_*(p^*_1L_1\otimes p^*_2L_2)
\end{equation}
defined as the composition
\begin{align}
p'^*_1f_{1*}L_1\otimes p'^*_2f_{2*}L_2
\xrightarrow{Ex(\Delta^*_{1*})\otimes Ex(\Delta^*_{2*})}
f_*p_1^*L_1\otimes f_*p_2^*L_2
\xrightarrow{\eqref{eq:Ex_lower*_tensor}}
f_*(p^*_1L_1\otimes p^*_2L_2).
\end{align}
By a classical argument (see \cite[III 1.6.4]{SGA5}), the following is a consequence of $\mathcal{S}_k$-strong local acyclicity by Proposition~\ref{SD_ula}:
\begin{proposition}[K\"unneth formula for $f_*$]
\label{Kun_*}
Let $\mathbf{T}$ be a motivic triangulated category. If $X_i\to S$ is universally strongly locally acyclic relatively to $L_i$ for $i=1,2$, then the map~\eqref{Kunneth_*} is an isomorphism. 

In particular, if $S=\operatorname{Spec}(k)$ and if $\mathbf{T}$ satisfies the condition~\ref{resol} in \ref{par:devissage}, the map~\eqref{Kunneth_*} is an isomorphism for any $L_i\in\mathbf{T}(X_i)$, $i=1,2$.
\end{proposition}

\subsection{K\"unneth formula for $f_!$}
\label{Kun_lower!}
\subsubsection{}
The second K\"unneth formula is concerned with the exceptional direct image functor. As part of the six functors formalism, for any morphism of schemes $f:X\to S$, there is an exceptional direct image functor (or  direct image with compact support)
\begin{align}
f_!:\mathbf{T}(X)\to\mathbf{T}(S)
\end{align}
which is compatible with compositions, such that $f_*=f_!$ if $f$ is proper. We also have
\begin{enumerate}
\item For any morphism $f:X\to S$, any object $K\in\mathbf{T}(X)$ 
and any $L\in\mathbf{T}(S)$, there is an invertible natural transformation
\begin{equation}
\label{Ex_!tens}
Ex(f^*_!,\otimes):(f_!K)\otimes L\to f_!(K\otimes f^*L)
\end{equation}
which agrees with 
the map~\eqref{Ex_*tens} if $f$ is proper.
\item For any Cartesian square
\begin{equation}\label{Cart_diag2}
\begin{gathered}
  \xymatrix{
    Y \ar[r]^-{q} \ar[d]_-{g} \ar@{}[rd]|{\Delta} & X \ar[d]^-{f}\\
    T \ar[r]_-{p} & S
  }
\end{gathered}
\end{equation}
  there is an invertible natural transformation
\begin{equation}
\label{Ex*_!}
Ex(\Delta^*_!):f^*p_!\to q_!g^*
\end{equation}
which is compatible with horizontal and vertical compositions of squares, and agrees with the map
~\eqref{Ex**} if $p$ is proper.
\end{enumerate}

\subsubsection{}
We now state a K\"unneth formula for the functor $f_!$. We use the assumptions and notation as in Notation~\ref{Kun_not}, with the following diagram
\begin{align}
\begin{gathered}
  \xymatrix{
   X_1 \ar[d]_-{f_1} & X_1\times_S X_2 \ar[l]_-{p_1} \ar[r]^-{p_2} \ar[d]_-{f} & X_2 \ar[d]^-{f_2}\\
   Y_1 & Y_1\times_S Y_2 \ar[l]^-{p_1'} \ar[r]_-{p_2'} & Y_2.
  }
\end{gathered}
\end{align}
\begin{lemma}
\label{Kunneth_formula}
For $i=1,2$, let $L_i$ be an element of $\mathbf{T}(X_i)$. There is a canonical isomorphism
\begin{equation}
\label{Kun'}
Kun_{S,(f_1,f_2),!}(L_1,L_2):p'^*_1f_{1!}L_1\otimes p'^*_2f_{2!}L_2
\simeq
f_!(p^*_1L_1\otimes p^*_2L_2).
\end{equation}
\end{lemma}
\proof
We have the following commutative diagram 
\begin{align}
\begin{gathered}
  \xymatrix{
   X_1 \ar[d]_-{f_1} \ar@{}[rd]|{\Delta_1} & X_1\times_S X_2 \ar[l]_-{p_1} \ar[rd]^-{p_2} \ar[d]^-{f_1\times id} & \\
   Y_1 & Y_1\times_S X_2 \ar[l]_-{\underline{p'_1}} \ar[r]^-{\underline{p_2}} \ar[d]_-{id\times f_2} \ar@{}[rd]|{\Delta_2} & X_2 \ar[d]^-{f_2}\\
   & Y_1\times_S Y_2 \ar[lu]^-{p_1'} \ar[r]^-{p_2'} & Y_2
  }
\end{gathered}
\end{align}
and the following composition:
\begin{align}
\begin{split}
p'^*_1f_{1!}L_1\otimes p'^*_2f_{2!}L_2
&\xrightarrow{id\otimes Ex(\Delta^*_{2!})}
p'^*_1f_{1!}L_1\otimes (id\times f_2)_!\underline{p^*_2}L_2\\
&\xrightarrow{Ex((id\times f_2)^*_!,\otimes)}
(id\times f_2)_!(\underline{p^*_2}L_2\otimes (id\times f_2)^*p'^*_1f_{1!}L_1)\\
&=(id\times f_2)_!(\underline{p^*_2}L_2\otimes \underline{p'^*_1}f_{1!}L_1)\\
&\xrightarrow{Ex(\Delta^*_{1!})}
(id\times f_2)_!(\underline{p^*_2}L_2\otimes (f_1\times id)_!p_1^*L_1)\\
&\xrightarrow{Ex((f_1\times id)^*_!,\otimes)}
(id\times f_2)_!(f_1\times id)_!(p_1^*L_1\otimes (f_1\times id)^*\underline{p^*_2}L_2)\\
&=f_!(p^*_1L_1\otimes p^*_2L_2).
\end{split}
\end{align}
All the maps involved are isomorphisms, and the result follows.
\endproof

\begin{remark}
\label{onebyone}
\begin{enumerate}
\item The K\"unneth formula for $f_!$ is very formal and is valid over a general base $S$, while all the other ones are valid only when $S$ is the spectrum of a field.

\item By definition, the map~\eqref{Kun'}
agrees with the composition
\begin{align}
\begin{split}
 p'^*_1f_{1!}L_1\otimes p'^*_2f_{2!}L_2
&\xrightarrow{Kun_{S,(id_{Y_1},f_2),!}(f_{1!}L_1,L_2)}
(id_{Y_1}\times_S f_2)_!(\underline{p'^*_1}f_{1!}L_1\otimes\underline{p^*_2}L_2)\\
 &\xrightarrow{Kun_{S,(f_1,id_{X_2}),!}(L_1,L_2)}
f_!(p^*_1L_1\otimes p^*_2L_2).
\end{split}
\end{align}
In other words, the map~\eqref{Kun'} is the composition of two maps of the same type where one of the $f_i$'s equals identity.
\end{enumerate}
\end{remark}

\subsection{K\"unneth formula for $f^!$}

\subsubsection{}
\label{num:fct_upper!}
The third K\"unneth formula is concerned with the exceptional inverse image functor. We recall the following facts from the six functors formalism:
\begin{enumerate}
\item For any morphism of schemes $f:X\to S$, the functor $f_!$ has a right adjoint, with the following pair of adjoint functors
\begin{align}
f_!:\mathbf{T}(X)\rightleftharpoons\mathbf{T}(S):f^!,
\end{align}
such that $f^!=f^*$ if $f$ is \'etale. 

\item For any closed immersion $i$ with complementary open immersion $j$, the functor $i_*$ is conservative, and there is a canonical distinguished triangle
\begin{equation}
\label{eq:loc_seq}
i_*i^!
\xrightarrow{ad'_{(i_*,i^!)}}
1
\xrightarrow{ad_{(j^*,j_*)}}
j_*j^*
\xrightarrow{\partial}
i_*i^![1].
\end{equation}
\item 
\label{pt:orient_sq}
For any Cartesian square
\begin{equation}\label{Cart_diag3}
\begin{gathered}
  \xymatrix{
    Y \ar[r]^-{q} \ar[d]_-{g} \ar@{}[rd]|{\Delta} & X \ar[d]^-{f}\\
    T \ar[r]_-{p} & S
  }
\end{gathered}
\end{equation}
which will be denoted as $Y$-$X$-$S$-$T$ (this notation will be used in the proof of Proposition \ref{Kunneth_thm}), there is a canonical invertible natural transformation
\begin{equation}
\label{upper!lower*}
Ex(\Delta^!_*):q_*f^!\simeq g^!p_*.
\end{equation}
When $p$ and $q$ are proper, the map~\eqref{upper!lower*} agrees with the composition
\begin{align}
q_*f^!=q_!f^!\xrightarrow{ad_{(f_!,f^!)}}f^!f_!q_!g^!\simeq f^!p_!g_!g^!\xrightarrow{ad'_{(g_!,g^!)}}f^!p_!=f^!p_*.
\end{align}
\item 
For any Cartesian square as~\eqref{Cart_diag3}, we deduce from the map $Ex(\Delta^*_!)$ in~\eqref{Ex*_!} a canonical natural transformation
\begin{equation}
\label{Ex*!}
Ex(\Delta^{*!}):g^*p^!
\xrightarrow{ad_{(q_!,q^!)}}
q^!q_!g^*p^!
\stackrel{(Ex(\Delta^*_!))^{-1}}{\simeq}
q^!f^*p_!p^!
\xrightarrow{ad'_{(p_!,p^!)}}
q^!f^*,
\end{equation}
which is an isomorphism when $f$ is smooth. 
\end{enumerate}

\begin{lemma}
\label{lem:Ex*!_iso}
Take a Cartesian square $\Delta$ as in~\eqref{Cart_diag3}. Assume that $\mathbf{T}$ satisfies $f'$-base change for any morphism $f'$ obtained from $f$ by pull-back. Then the map $Ex(\Delta^{*!})$ in~\eqref{Ex*!} is an isomorphism.

\end{lemma}

\proof

By localizing $p$ we may suppose $p$ quasi-projective, and therefore we only need to deal with the cases where $p$ is either a closed immersion or a smooth morphism. The smooth case follows from purity, and the case of a closed immersion follows from the localization sequence~\ref{eq:loc_seq} and the base change property.
\endproof

In particular, if $\mathbf{T}$ is a motivic triangulated category which satisfies the 
condition \ref{resol} in \ref{par:devissage}, by Proposition~\ref{SD_ula}, Lemma~\ref{lem:Ex*!_iso} applies whenever $f\in\mathcal{S}_k$ (see Definition~\ref{def:dev}).

\subsubsection{}
We now state a K\"unneth formula for the functor $f^!$. We use the assumptions and 
notation as in Notation~\ref{Kun_not}, with the following diagram
\begin{align}
\begin{gathered}
  \xymatrix{
   X_1 \ar[d]_-{f_1} & X_1\times_S X_2 \ar[l]_-{p_1} \ar[r]^-{p_2} \ar[d]_-{f} & X_2 \ar[d]^-{f_2}\\
   Y_1 & Y_1\times_S Y_2 \ar[l]^-{p_1'} \ar[r]_-{p_2'} & Y_2.
  }
\end{gathered}
\end{align}
For $i=1,2$, let $M_i$ be an object of $\mathbf{T}(Y_i)$. Taking $L_i=f^!_iM_i$ in Lemma~\ref{Kunneth_formula}, we obtain a canonical map 
\begin{equation}
\label{Kunneth}
Kun^{!}_{(f_1,f_2)}(M_1,M_2):
p^*_1f^!_1M_1\otimes p^*_2f^!_2M_2
\to
f^!(p'^*_1M_1\otimes p'^*_2M_2)
\end{equation}
by the composition
\begin{align}
\begin{split}
        p^*_1f^!_1M_1\otimes p^*_2f^!_2M_2&\xrightarrow{ad_{(f^!,f_!)}} f^!f_!(p^*_1f^!_1M_1\otimes p^*_2f^!_2M_2)\\
        &\xrightarrow[\simeq]{\eqref{Kun'}} f^!(p'^*_1f_{1!}f_1^!M_1\otimes p'^*_2f_{2!}f_2^!M_2)\\
        &\xrightarrow{ad'_{(f_1^!,f_{1!})}\otimes ad'_{(f_2^!,f_{2!})}} f^!(p'^*_1M_1\otimes p'^*_2M_2).
        \end{split}
\end{align}
We now establish the same principle as in Remark~\ref{onebyone}. Consider the following diagram
\begin{align}
\begin{gathered}
  \xymatrix{
   X_1 \ar[d]_-{f_1} & X_1\times_S X_2 \ar[l]_-{p_1} \ar[rd]^-{p_2} \ar[d]^-{f_1\times id} & \\
   Y_1 & Y_1\times_S X_2 \ar[l]_-{\underline{p'_1}} \ar[r]^-{\underline{p_2}} \ar[d]_-{id\times f_2} & X_2 \ar[d]^-{f_2}\\
   & Y_1\times_S Y_2 \ar[lu]^-{p_1'} \ar[r]^-{p_2'} & Y_2
  }
\end{gathered}
\end{align}
As particular cases of the \eqref{Kunneth}, we have the following maps:
\begin{equation}
\label{part_kunneth1}
Kun^{!}_{(id_{Y_1},f_2)}(M_1,M_2):
\underline{p'^*_1}M_1\otimes\underline{p^*_2}f^!_2M_2
\to
(id_{Y_1}\times_Sf_2)^!(p'^*_1M_1\otimes p'^*_2M_2)
\end{equation}
\begin{equation}
\label{part_kunneth2}
Kun^{!}_{(f_1,id_{X_2})}(M_1,f_2^!M_2):
p^*_1f^!_1M_1\otimes p^*_2f^!_2M_2
\to
(f_1\times_Sid_{X_2})^!(\underline{p'^*_1}M_1\otimes\underline{p^*_2}f^!_2M_2)
\end{equation}

\begin{lemma}
\label{Ex!*_comp}
The map
\begin{equation}
\label{anotherKunneth}
\begin{split}
p^*_1f^!_1M_1\otimes p^*_2f^!_2M_2
\xrightarrow{\eqref{part_kunneth2}}
&(f_1\times_Sid_{X_2})^!(\underline{p'^*_1}M_1\otimes\underline{p^*_2}f^!_2M_2)\\
\xrightarrow{\eqref{part_kunneth1}}
&f^!(p'^*_1M_1\otimes p'^*_2M_2)
\end{split}
\end{equation}
obtained from \eqref{part_kunneth1} and \eqref{part_kunneth2} agrees with the map~\eqref{Kunneth}.
\end{lemma}
\proof
By definition, the map~\eqref{anotherKunneth} is the composition
\begin{align}
\begin{split}
p^*_1f^!_1M_1\otimes p^*_2f^!_2M_2 
&\xrightarrow{ad_{((f_1\times id)_!,(f_1\times id)^!)}}
(f_1\times id)^!(f_1\times id)_!(p^*_1f^!_1M_1\otimes p^*_2f^!_2M_2)\\
&\xrightarrow{\eqref{Kun'}}
(f_1\times id)^!(\underline{p'^*_1}f_{1!}f_1^!M_1\otimes\underline{p^*_2}f^!_2M_2)\\
&\xrightarrow{ad'_{(f_1^!,f_{1!})}}
(f_1\times id)^!(\underline{p'^*_1}M_1\otimes\underline{p^*_2}f^!_2M_2)\\
&\xrightarrow{ad_{((id\times f_2)_!,(id\times f_2)^!)}}
(f_1\times id)^!(id\times f_2)^!(id\times f_2)_!(\underline{p'^*_1}M_1\otimes\underline{p^*_2}f^!_2M_2)\\
&\xrightarrow{\eqref{Kun'}}
f^!(p_1'^*M_1\otimes p'^*_2f_{2!}f^!_2M_2)\xrightarrow{ad'_{(f_2^!,f_{2!})}}
f^!(p_1'^*M_1\otimes p'^*_2M_2).
\end{split}
\end{align}
The results then follows from Remark~\ref{onebyone} and the naturality of the functors $f\mapsto f^!$ and $f\mapsto f_!$.
\endproof

\begin{proposition}
\label{Kunneth_thm}
Suppose that $\mathbf{T}$ is a motivic triangulated category which satisfies the
condition ~\ref{resol} in \ref{par:devissage}. 
Then for $S=\operatorname{Spec}(k)$, the map~\eqref{Kunneth} is an isomorphism.
\end{proposition}

\proof
By Lemma~\ref{Ex!*_comp}, it suffices to show that the maps \eqref{part_kunneth1} 
and \eqref{part_kunneth2} are isomorphisms. By symmetry, it suffices to show that 
\eqref{part_kunneth1} is an isomorphism. 
In other words we can assume that $X_1=Y_1$ and $f_1=id_{X_1}$. 
By weak devissage we may suppose that $M_1=p_*\mathbbold{1}_W$, 
where $p:W\to Y_1=X_1$ is a proper morphism.  
The map~\eqref{Kunneth} then becomes the following
\begin{equation}
\label{Kunneth_special}
Kun^{!}:
p^*_1p_*\mathbbold{1}_W\otimes p^*_2f^!_2M_2
\to
(id_{X_1}\times_kf_2)^!(p'^*_1p_*\mathbbold{1}_W\otimes p'^*_2M_2).
\end{equation}
We have the following commutative diagram
\begin{align}
\begin{gathered}\label{*!_diagram}
\xymatrix{
W\ar[dd]_-p&&W\times_k X_2\ar[ll]_{p_{1W}}\ar[dd]_<(.6){p\times_k id}\ar[rr]^-{p_{2W}}
\ar[rrd]_<(.2){id\times_k f_2}&&X_2\ar[rrd]^-{f_2}\\
&&&&W\times_k Y_2\ar@/^1pc/[llllu]^<(.7){p'_{1W}}\ar[dd]_<(.4){p\times_k id}
\ar[rr]_<(0.3){p'_{2W}}&&Y_2\\
X_1&&X_1\times_kX_2\ar[ll]_-{p_1}\ar[rrd]^-{id\times_k f_2}\ar[uurr]_<(.3){p_2}\\
&&&&X_1\times_k Y_2\ar@/^1pc/[ullll]^<(.7){p'_1}\ar[uurr]_-{p'_2}.
}
\end{gathered}
\end{align}
and with the convention in~\ref{num:fct_upper!}~\eqref{pt:orient_sq} we denote by
\begin{itemize}
\item $\Delta_1$ the Cartesian square formed by $W\times_kY_2$-$X_1\times_kY_2$-$X_1$-$W$;
\item $\Delta_2$ the Cartesian square formed by $W\times_kX_2$-$X_1\times_kX_2$-$X_1$-$W$;
\item $\Delta_3$ the Cartesian square formed by $W\times_kX_2$-$W\times_kY_2$-$X_1\times_kY_2$-$X_1\times_kX_2$, and $\Delta'_3$ the transpose of the same square, oriented as $W\times X_2$-$X_1\times_kX_2$-$X_1\times_kY_2$-$W\times_kY_2$;
\item $\Delta_4$ the Cartesian square formed by $W\times_kX_2$-$W\times_kY_2$-$Y_2$-$X_2$;
\item $\Delta_5$ the Cartesian square formed by $X_1\times_kX_2$-$X_1\times_kY_2$-$Y_2$-$X_2$.
\end{itemize}
To show that the map~\eqref{Kunneth_special} is an isomorphism, we transform it into other maps which are know to be isomorphisms: we have the following two composition maps, where all arrows are isomorphisms by Lemma~\ref{lem:Ex*!_iso}:
\begin{equation}
\label{*!1}
\begin{split}
p^*_1p_*\mathbbold{1}_W\otimes p^*_2f^!_2M_2
&\xrightarrow{Ex(\Delta^*_{2!})}
(p\times_kid_{X_2})_*\mathbbold{1}_{W\times_kX_2}\otimes p^*_2f^!_2M_2 \\
&\xrightarrow{Ex((p\times_kid_{X_2})^*_!,\otimes)}
(p\times_kid_{X_2})_*(p\times_kid_{X_2})^*p^*_2f^!_2M_2\\
&=(p\times_kid_{X_2})_*p^*_{2W}f^!_2M_2\\
&\xrightarrow{(Ex(\Delta_4^{*!}))^{-1}}
(p\times id_{X_2})_*(id_W\times_kf_2)^!p'^*_{2W}M_2\\
&\xrightarrow{Ex(\Delta'^!_{3*})}
(id_{Y_1}\times f_2)^!(p\times_kid_{Y_2})_*p'^*_{2W}M_2;
\end{split}
\end{equation}

\begin{equation}
\label{*!2}
\begin{split}
(id_{Y_1}\times_kf_2)^!(p_1'^*p_*\mathbbold{1}_W\otimes p_2'^*M_2)&\xrightarrow{Ex(\Delta^*_{1!})}
(id_{Y_1}\times_kf_2)^!((p\times_kid_{Y_2})_*\mathbbold{1}_{W\times_kY_2}\otimes p_2'^*M_2) \\
&\xrightarrow{Ex((p\times_kid_{Y_2})^*_!,\otimes)}
(id_{Y_1}\times_kf_2)^!(p\times_kid_{Y_2})_*p'^*_{2W}M_2,
\end{split}
\end{equation}
using the fact that $p$ is proper and $p_{2W}:W\times X_2\to X_2$ belongs to $\mathcal{S}_k$.
\footnote{Alternatively, we could also use strong devissage and suppose that $p_{2W}$ is smooth.}
Therefore we are reduced to show that the following diagram is commutative:
\begin{align}
\begin{gathered}
  \xymatrix{
    & p^*_1p_*\mathbbold{1}_{W}\otimes p^*_2f^!_2M_2 \ar[d]^-{\eqref{*!1}} \ar[ld]_-{\eqref{Kunneth}}\\
    (id_{Y_1}\times_kf_2)^!(p_1'^*p_*\mathbbold{1}_W\otimes p_2'^*M_2) \ar[r]^-{~\eqref{*!2}~} & (id_{Y_1}\times_kf_2)^!(p\times_kid_{Y_2})_*p'^*_{2W}M_2.
  }
\end{gathered}
\end{align}
This follows from a diagram chase of the following form:
\begin{align*}
\begin{gathered}
\resizebox{\textwidth}{!}{%
  \xymatrix{
   (id_{X_1}\times_kf_2)_!(p_1^*p_*\mathbbold{1}_W\otimes p_2^*f_2^!M_2) \ar[r]^-{Ex(\Delta^*_{1!})}_-{\sim} \ar[d]^-{Ex(\Delta^*_{2!})}_-{\wr} & (id_{X_1}\times_kf_2)_!((id_{X_1}\times_kf_2)^*(p\times_kid_{Y_2})_*\mathbbold{1}_{W\times_kY_2}\otimes p_2^*f_2^!M_2) \ar[d]_-{\wr}^-{(Ex((id_{X_1}\times_kf_2)^*_!,\otimes))^{-1}} \\
   (id_{X_1}\times_kf_2)_!((p\times_kid_{X_2})_*\mathbbold{1}_{W\times X_2}\otimes p_2^*f_2^!M_2) \ar[ru]_-{(Ex(\Delta'^*_{3!}))^{-1}}^-{\sim} \ar[d]^-{Ex((p\times_kid_{X_2})^*_!,\otimes)}_-{\wr} & (p\times_kid_{Y_2})_*\mathbbold{1}_{W\times_kY_2}\otimes (id_{X_1}\times_kf_2)_!p_2^*f_2^!M_2 \ar[d]_-{\wr}^-{Ex((p\times_kid_{Y_2})^*_!,\otimes)} \\
   (id_{X_1}\times_kf_2)_!(p\times_kid_{X_2})_*p_{2W}^*f_2^!M_2 \ar@{=}[d] & (p\times_kid_{Y_2})_*(p\times_kid_{Y_2})^*(id_{X_1}\times_kf_2)_!p_2^*f_2^!M_2 \ar[d]_-{\wr}^-{(Ex(\Delta^*_{5!}))^{-1}} \\
   (p\times_kid_{Y_2})_*(id_W\times_kf_2)_!p_{2W}^*f_2^!M_2 \ar[d]^-{ad_{((id_W\times_kf_2)_!,(id_W\times_kf_2)^!)}} \ar[ru]^-{(Ex(\Delta^*_{3!}))^{-1}}_-{\sim} & (p\times_kid_{Y_2})_*p_{2W}'^*f_{2!}f_2^!M_2 \\
   (p\times_kid_{Y_2})_*(id_W\times_kf_2)_!(id_W\times_kf_2)^!(id_W\times_kf_2)_!p_{2W}^*f_2^!M_2 \ar[r]^-{(Ex(\Delta^*_{4!}))^{-1}}_-{\sim} & (p\times_kid_{Y_2})_*(id_W\times_kf_2)_!(id_W\times_kf_2)^!p_{2W}'^*f_{2!}f_2^!M_2. \ar[u]_-{ad'_{((id_W\times_kf_2)_!,(id_W\times_kf_2)^!)}}
  }
}
\end{gathered}
\end{align*}
The commutativity of the hexagon follows from the following general fact: for any Cartesian diagram of the form
\begin{align}
\begin{gathered}
  \xymatrix{
    Y \ar[r]^-{q} \ar[d]_-{g} \ar@{}[rd]|{\Delta} & X \ar[d]^-{f}\\
    T \ar[r]_-{p} & S
  }
\end{gathered}
\end{align}
and any object $M\in\mathbf{T}(X)$, we have a commutative diagram
\begin{align}
\begin{gathered}
  \xymatrix@C=40pt{
   f_!(q_!\mathbbold{1}_Y\otimes M) \ar[d]_-{Ex(q^*_!,\otimes)}^-{\wr} & f_!(f^*p_!(\mathbbold{1}_T\otimes M)) \ar[l]_-{Ex(\Delta'^*_!)}^-{\sim}  & p_!\mathbbold{1}_T\otimes f_!M \ar[l]_-{Ex(f^*_!,\otimes)}^-{\sim} \ar[d]^-{Ex(p^*_!,\otimes)}_-{\wr} \\
   f_!q_!q^*M \ar[r]^-{\sim} & p_!g_!q^*M & p_!p^*f_!M \ar[l]_-{Ex(\Delta^*_!)}^-{\wr}
  }
\end{gathered}
\end{align}
where the square $\Delta'$ is the transpose of the square $\Delta$, oriented as $Y$-$T$-$S$-$X$. The rest of the diagram follows from standard adjunctions of functors and the fact that the maps of the form $Ex(\Delta^*_!)$ are compatible with horizontal and vertical compositions of Cartesian squares, given that
\begin{itemize}
\item The square $\Delta_2$ is the composition of $\Delta_1$ and $\Delta'_3$;
\item The square $\Delta_4$ is the composition of $\Delta_3$ and $\Delta_5$.
\end{itemize}
\endproof

\subsection{K\"unneth formula for $\protect\underline{Hom}$}\

\subsubsection{}
The last K\"unneth formula is concerned with the internal Hom functor, the remaining one of the six functors. If $\mathbf{T}$ is motivic, its monoidal structure is closed, namely the tensor product on $\mathbf{T}$ has a right adjoint given by the internal Hom functor $\underline{Hom}(\cdot,\cdot)$, such that we have natural bijections
\begin{align}
Hom(A\otimes B,C)=Hom(A,\underline{Hom}(B,C)).
\end{align}
For any $A$, we denote by $\eta_A:B\to\underline{Hom}(A,A\otimes B)$ the unit (or coevaluation) map, and $\epsilon_A:A\otimes\underline{Hom}(A,B)\to B$ the counit (or evaluation) map. We have the following exchange structures: for any morphism $f:X\to Y$ and objects $K\in\mathbf{T}(X)$, $L,M\in\mathbf{T}(Y)$ we have the following natural isomorphisms:
\begin{equation}
\label{Ex_!Hom}
Ex(f_!,\underline{Hom}):\underline{Hom}(f_!K,L)\xrightarrow{\sim} f_*\underline{Hom}(K,f^!L);
\end{equation}
\begin{equation}
\label{Ex^!Hom}
Ex(f^!,\underline{Hom}):f^!\underline{Hom}(L,M)\xrightarrow{\sim} \underline{Hom}(f^*L,f^!M).
\end{equation}
We deduce the following isomorphism:
\begin{equation}
\label{eq:intern_Hom_part}
\underline{Hom}(f_!\mathbbold{1}_X,L)\xrightarrow{\sim} f_*\underline{Hom}(\mathbbold{1}_X,f^!L)=f_*f^!L.
\end{equation}
When $f$ is proper, the isomorphism~\eqref{eq:intern_Hom_part} fits into the two following commutative diagrams:
\begin{equation}\label{diag:compat_hom_lower!}
\begin{gathered}
  \xymatrix{
   \underline{Hom}(f_!\mathbbold{1}_X,L) \ar[d]_-{\eqref{eq:intern_Hom_part}}^-{\wr} \ar[r]^-{} & \underline{Hom}(f_!\mathbbold{1}_X,L)\otimes f_!\mathbbold{1}_X \ar[d]^-{\epsilon_{f_!\mathbbold{1}_X}} \\
   f_!f^!L \ar[r]^-{ad'_{(f_!,f^!)}} & L
  }
\end{gathered}
\end{equation}
\begin{equation}\label{diag:compat_hom_lower!2}
\begin{gathered}
  \xymatrix{
   f_!M \ar[r]^-{\eta_{f_!\mathbbold{1}_X}} \ar[d]_-{ad_{(f_!,f^!)}} & \underline{Hom}(f_!\mathbbold{1}_X,f_!M\otimes f_!\mathbbold{1}_X) \ar[d]^-{\eqref{eq:intern_Hom_part}}_-{\wr} \\
   f_!f^!f_!M \ar[r]^-{} & f_*f^!(f_!M\otimes f_!\mathbbold{1}_X)
  }
\end{gathered}
\end{equation}
where the two unlabeled horizontal maps are induced by the unit $\mathbbold{1}_Y\to f_*\mathbbold{1}_X=f_!\mathbbold{1}_X$.

\subsubsection{}
Let $Y_1$ and $Y_2$ be two $S$-schemes. For $i=1,2$, 
we denote by $p'_i:Y_1\times_S Y_2\to Y_i$ the canonical projection, 
and let $M_i$ and $N_i$ be objects of $\mathbf{T}(Y_i)$. We have a canonical map
\begin{equation}
\label{Kunneth_dual}
p_1'^*M_1\otimes p_1'^*\underline{Hom}(M_1, N_1)\otimes p_2'^*M_2\otimes p_2'^*\underline{Hom}(M_2, N_2)
\to
p_1'^*N_1\otimes p_2'^*N_2
\end{equation}
which induces a map
\begin{equation}
\label{Kunneth_part}
p_1'^*\underline{Hom}(M_1, N_1)\otimes p_2'^*\underline{Hom}(M_2, N_2)
\to\underline{Hom}(p_1'^*M_1\otimes p_2'^*M_2, p_1'^*N_1\otimes p_2'^*N_2).
\end{equation}
\begin{proposition}
\label{hom_product}
Suppose that $\mathbf{T}$ is a motivic triangulated category which satisfies 
the condition~\ref{resol} in \ref{par:devissage} and $M_1,M_2$ are constructible. Then for $S=\operatorname{Spec}(k)$, the map~\eqref{Kunneth_part} is an isomorphism.
\end{proposition}
\proof
By weak devissage we may suppose that $M_i=f_{i!}\mathbbold{1}_{X_i}$ where $f_i:X_i\to Y_i$ is a proper morphism. We use the following notations:
\begin{align}
\begin{gathered}
  \xymatrix{
   X_1 \ar[d]_-{f_1} & X_1\times_k X_2 \ar[l]_-{p_1} \ar[r]^-{p_2} \ar[d]_-{f} & X_2 \ar[d]^-{f_2}\\
   Y_1 & Y_1\times_k Y_2 \ar[l]^-{p_1'} \ar[r]_-{p_2'} & Y_2.
  }
  \end{gathered}
\end{align}
Then the map~\eqref{Kunneth_part} becomes
\begin{align}
p_1'^*\underline{Hom}(f_{1!}\mathbbold{1}_{X_1}, N_1)\otimes p_2'^*\underline{Hom}(f_{2!}\mathbbold{1}_{X_2}, N_2)
\to
\underline{Hom}(p_1'^*f_{1!}\mathbbold{1}_{X_1}\otimes p_2'^*f_{2!}\mathbbold{1}_{X_2}, p_1'^*N_1\otimes p_2'^*N_2).
\end{align}
As in the proof of Proposition~\ref{Kunneth_thm}, we transform this map into other isomorphisms.
By Lemma~\ref{Kunneth_formula} and Proposition~\ref{Kunneth_thm}, 
we have the following two composition maps, where all arrows are isomorphisms:
\begin{equation}
\label{Hom*_LHS}
\begin{split}
p_1'^*\underline{Hom}(f_{1!}\mathbbold{1}_{X_1}, N_1)\otimes p_2'^*\underline{Hom}(f_{2!}\mathbbold{1}_{X_2}, N_2)
&\overset{\eqref{eq:intern_Hom_part}}{\simeq}
p_1'^*f_{1*}f^!_1N_1\otimes p_2'^*f_{2*}f^!_2N_2\\
&\overset{\eqref{Kun'}}{\simeq}
f_*(p_1^*f^!_1N_1\otimes p_2^*f^!_2N_2)\\
&\overset{\eqref{Kunneth}}{\simeq}
f_*f^!(p_1'^*N_1\otimes p_2'^*N_2);
\end{split}
\end{equation}
\begin{equation}
\label{Hom*_RHS}
\begin{split}
\underline{Hom}(p_1'^*f_{1!}\mathbbold{1}_{X_1}\otimes p_2'^*f_{2!}\mathbbold{1}_{X_2}, p_1'^*N_1\otimes p_2'^*N_2)
&\overset{\eqref{Kun'}}{\simeq}
\underline{Hom}(f_!\mathbbold{1}_{X_1\times_kX_2}, p_1'^*N_1\otimes p_2'^*N_2)\\
&\overset{\eqref{eq:intern_Hom_part}}{\simeq}
f_*f^!(p_1'^*N_1\otimes p_2'^*N_2).
\end{split}
\end{equation}
Therefore we are reduced to show that the following diagram is commutative:
\begin{align}
\begin{gathered}
  \xymatrix{
   p_1'^*\underline{Hom}(f_{1!}\mathbbold{1}_{X_1}, N_1)\otimes p_2'^*\underline{Hom}(f_{2!}\mathbbold{1}_{X_2}, N_2) \ar[d]_-{\eqref{Kunneth_part}} \ar[rd]^-{\eqref{Hom*_LHS}} & \\
    \underline{Hom}(p_1'^*f_{1!}\mathbbold{1}_{X_1}\otimes p_2'^*f_{2!}\mathbbold{1}_{X_2}, p_1^*N_1\otimes p_2^*N_2) \ar[r]^-{\eqref{Hom*_RHS}} & f_*f^!(p'^*_1N_1\otimes p'^*_2N_2).
  }
\end{gathered}
\end{align}
By the same argument as in the proof of Proposition~\ref{Kunneth_thm}, we can assume that $X_1=Y_1$ and $f_1=id_{X_1}$. The result then follows from a diagram chase, using the following general fact: for any Cartesian diagram of the form
\begin{align}
\begin{gathered}
  \xymatrix{
    Y \ar[r]^-{g} \ar[d]_-{q} \ar@{}[rd]|{\Delta} & X \ar[d]^-{p}\\
    T \ar[r]_-{f} & S
  }
\end{gathered}
\end{align}
with $f$ and $g$ proper, and any objects $A\in\mathbf{T}(X)$, $B\in\mathbf{T}(S)$, there is a canonical isomorphism
\begin{equation}
\label{eq:Ex_twice_AB}
\begin{split}
\xymatrix{
A\otimes p^*f_!f^!B
\ar[r]^-{Ex(\Delta^*_!)}_-{\sim}
&A\otimes g_!q^*f^!B
\ar[r]^-{Ex(g^*_!,\otimes)}_-{\sim}
&g_!(g^*A\otimes q^*f^!B)
}
\end{split}
\end{equation}
and a commutative diagram of the form
$$
\resizebox{\textwidth}{!}{
  \xymatrix{
    A\otimes p^*\underline{Hom}(f_!\mathbbold{1}_T,B) \ar[r]^-{\eta_{p^*f_!\mathbbold{1}_T}} \ar[d]^-{\wr}_-{\eqref{eq:intern_Hom_part}} & \underline{Hom}(p^*f_!\mathbbold{1}_T,A\otimes p^*\underline{Hom}(f_!\mathbbold{1}_T,B)\otimes p^*f_!\mathbbold{1}_T) \ar[r]^-{\epsilon_{f_!\mathbbold{1}_T}} \ar[d]^-{Ex(\Delta^*_!)}_-{\wr} & \underline{Hom}(p^*f_!\mathbbold{1}_T,A\otimes p^*B) \ar[d]^-{Ex(\Delta^*_!)}_-{\wr}\\
    A\otimes p^*f_*f^!B \ar[d]^-{\wr}_-{\eqref{eq:Ex_twice_AB}} &  \underline{Hom}(g_!\mathbbold{1}_Y,A\otimes p^*\underline{Hom}(f_!\mathbbold{1}_T,B)\otimes p^*f_!\mathbbold{1}_T) \ar[r]^-{\epsilon_{f_!\mathbbold{1}_T}} \ar[d]^-{\eqref{eq:intern_Hom_part}}_-{\wr} & \underline{Hom}(g_!\mathbbold{1}_Y,A\otimes p^*B) \ar[d]^-{\eqref{eq:intern_Hom_part}} \\
    g_!(g^*A\otimes q^*f^!B) \ar[d]_-{ad_{(g_!,g^!)}} & g_*g^!(A\otimes p^*\underline{Hom}(f_!\mathbbold{1}_T,B)\otimes p^*f_!\mathbbold{1}_T) \ar[r]^-{\epsilon_{f_!\mathbbold{1}_T}} \ar[d]_-{\wr}^-{\eqref{eq:intern_Hom_part}} & g_*g^!(A\otimes p^*B) \\
    g_!g^!g_!(g^*A\otimes q^*f^!B) & g_*g^!(A\otimes p^*f_*f^!B\otimes p^*f_!\mathbbold{1}_T) & g_!g^!(A\otimes p^*f_*f^!B) \ar[u]_-{ad'_{(f_!,f^!)}} \ar[l]^-{} \ar@/^2pc/[ll]^-{\sim}_-{\eqref{eq:Ex_twice_AB}}
  }
}
$$
using the diagrams~\eqref{diag:compat_hom_lower!} and~\eqref{diag:compat_hom_lower!2}.
\endproof

\begin{remark}
When $\mathbf{T}=DM_h$ is the category of $h$-motives, Proposition~\ref{hom_product} also holds when $M_1$ and $M_2$ are \emph{locally constructible} (\cite[Definition 6.3.1]{CD3}). This is because to show that the map~\eqref{Kunneth_part} is an isomorphism it suffices to check it \'etale locally, and for any \'etale morphism $f:Y\to X$ and objects $M,N\in\mathbf{T}(X)$, the canonical map $f^*\underline{Hom}(M,N)\to\underline{Hom}(f^*M,f^*N)$ is an isomorphism.
\end{remark}

\subsubsection{}
We end this section by summarizing the K\"unneth formulas we have obtained:
\begin{theorem}
\label{th:kun}
Let $S$ be a scheme and let $f_1:X_1\to Y_1$, $f_2:X_2\to Y_2$ be two $S$-morphisms. Denote by $f:X_1\times_S X_2\to Y_1\times_S Y_2$ be their product. 
Let $\mathbf{T}$ be a motivic triangulated category. For $i=1,2$, consider objects $L_i\in\mathbf{T}(X_i)$ and $M_i,N_i\in\mathbf{T}(Y_i)$. Then with the notations of~\ref{Kun_not} there are canonical maps
\begin{align}
\label{eq:th_lower*}
f_{1\ast}L_1\boxtimes_S f_{2\ast}L_2
\to 
f_\ast(L_1\boxtimes_S L_2)
\end{align}
\begin{align}
\label{eq:th_lower!}
f_{1!}L_1\boxtimes_S f_{2!}L_2
\to 
f_!(L_1\boxtimes_S L_2)
\end{align}
\begin{align}
\label{eq:th_upper!}
f^!_1M_1\boxtimes_S f^!_2M_2
\to
f^!(M_1\boxtimes_S M_2)
\end{align}
\begin{align}
\label{eq:th_hom}
\underline{Hom}(M_1, N_1)\boxtimes_S \underline{Hom}(M_2, N_2)
\to
\underline{Hom}(M_1\boxtimes_S M_2, N_1\boxtimes_S N_2)
\end{align}
such that
\begin{enumerate}
\item The map~\eqref{eq:th_lower!} is an isomorphism.
\item If $\mathbf{T}$ satisfies the condition~\ref{resol} in \ref{par:devissage} and 
$S=\operatorname{Spec}(k)$, 
then the maps~\eqref{eq:th_lower*},~\eqref{eq:th_upper!} are isomorphisms.
\item If $\mathbf{T}$ satisfies the condition~\ref{resol} in \ref{par:devissage},
 $S=\operatorname{Spec}(k)$ and $M_1$, $M_2$ are constructible, 
 then the map~\eqref{eq:th_hom} is an isomorphism.
\end{enumerate}
\end{theorem}

\section{The Verdier pairing}
\label{section:pairing}

In this section, we define the Verdier pairing using the K\"unneth formulas in 
Section~\ref{chap_kun}, following \cite[III]{SGA5}. This pairing is compatible 
with proper direct images, which implies the Lefschetz-Verdier formula. 
We assume that $\mathbf{T}_c$ is a motivic triangulated category 
of constructible objects 
which satisfies the condition~\ref{resol} in \ref{par:devissage}.
\subsection{The pairing} 
\label{section_vpairing}
\subsubsection{}
\label{3.1.1}
Let $X_1$ and $X_2$ be two schemes. We denote by $X_{12}=X_1\times_k X_2$ 
and $p_i:X_{12}\to X_i$ the projections. Let $L_i\in\mathbf{T}_c(X_i)$ and
let $q_i:X_i\to {\rm Spec}(k)$ be the structure map for $i=1,2$. 
We denote by $\mathbb{D}(L_i)=\underline{Hom}(L_i, \mathcal{K}_{X_i})=\underline{Hom}(L_i, q_i^!\mathbbold{1}_k)$. 
See \cite[Section 4.4]{CD1} for the detailed duality formalism related to this functor.
By Proposition~\ref{Kunneth_thm} and Proposition~\ref{hom_product}, we have the following isomorphism
\begin{equation}
\label{dual_hom}
p_1^*\mathbb{D}(L_1)\otimes p_2^*L_2
\simeq
\underline{Hom}(p_1^*L_1, p_2^!L_2),
\end{equation}
given by the composition
\begin{align}
\begin{split}
p_1^*\mathbb{D}(L_1)\otimes p_2^*L_2&=p_1^*\underline{Hom}(L_1,\mathcal{K}_{X_1})\otimes p_2^*L_2\\
&\overset{\eqref{Kunneth_part}}{\simeq}
\underline{Hom}(p_1^*L_1, p_1^*\mathcal{K}_{X_1}\otimes p_2^*L_2)
\overset{\eqref{Kunneth}}{\simeq}
\underline{Hom}(p_1^*L_1, p_2^!L_2).
\end{split}
\end{align}
By symmetry, we also get an isomorphism
\begin{equation}
p_2^*\mathbb{D}(L_2)\otimes p_1^*L_1
\simeq
\underline{Hom}(p_2^*L_2, p_1^!L_1).
\end{equation}
The particular case of~\eqref{dual_hom} for $L_1=\mathbbold{1}_{X_1}$ and $L_2=\mathcal{K}_{X_2}$ gives an isomorphism
\begin{equation}
\label{Kunneth_K}
p_1^*\mathcal{K}_{X_1}\otimes p_2^*\mathcal{K}_{X_2}
\simeq
\mathcal{K}_{X_{12}}.
\end{equation}

\subsubsection{}
\label{abstract_map}
Now let $c:C\to X_{12}$ and $d:D\to X_{12}$ be two morphisms, and form the following Cartesian square
\begin{align}
\begin{gathered}
  \xymatrix{
    E \ar[r]^-{c'} \ar[d]_-{d'} \ar@{}[rd]|{\Delta} & D \ar[d]^-{d}\\
    C \ar[r]_-{c} & X_{12}
  }
\end{gathered}
\end{align}
Denote by $e:E\to X_{12}$ the canonical morphism. For any two objects $P,Q\in\mathbf{T}(X_{12})$, we have a canonical map
\begin{equation}
\label{eq:cdPQ}
\begin{split}
d'^*c^!P\otimes c'^*d^!Q
&\xrightarrow{ad_{(e_!,e^!)}}
e^!e_!(d'^*c^!P\otimes c'^*d^!Q)\\
&\xrightarrow{\eqref{Kun'}}
e^!(c_!c^!P\otimes d_!d^!Q)
\xrightarrow{ad'_{(c_!,c^!)}\otimes ad'_{(d_!,d^!)}}
e^!(P\otimes Q)
\end{split}
\end{equation}
where we use Lemma~\ref{Kunneth_formula} relative to the base scheme $X_{12}$ and the morphisms $c$ and $d$. Therefore we deduce a canonical map
\begin{equation}
\label{eePQ}
c_*c^!P\otimes d_*d^!Q\to e_*e^!(P\otimes Q)
\end{equation}
which is given by the composition
\begin{align}
\begin{split}
c_*c^!P\otimes d_*d^!Q
&\xrightarrow{ad_{(e^*,e_*)}}
e_*e^*(c_*c^!P\otimes d_*d^!Q)
=e_*(e^*c_*c^!P\otimes e^*d_*d^!Q)\\
&=e_*(d'^*c^*c_*c^!P\otimes c'^*d^*d_*d^!Q)\\
&\xrightarrow{ad'_{(c^*,c_*)}\otimes ad'_{(d^*,d_*)}}
e_*(d'^*c^!P\otimes c'^*d^!Q)\xrightarrow{\eqref{eq:cdPQ}}
e_*e^!(P\otimes Q).
\end{split}
\end{align}

\subsubsection{}
For $i=1,2$ denote by $c_i=p_i\circ c:C\to X_i$, $d_i=p_i\circ d:D\to X_i$, and let $L_i\in\mathbf{T}_c(X_i)$. By~\ref{3.1.1} and the projection formula for $\underline{Hom}$ we have an isomorphism 
\begin{equation}
\label{Homc12}
\begin{split}
\underline{Hom}(c_1^*L_1, c_2^!L_2)
&=\underline{Hom}(c^*p_1^*L_1, c^!p_2^!L_2)\\
\overset{\eqref{Ex^!Hom}}{\simeq}
&c^!\underline{Hom}(p_1^*L_1, p_2^!L_2)
\overset{\eqref{dual_hom}}{\simeq} 
c^!(p_1^*\mathbb{D}(L_1)\otimes p_2^*L_2).
\end{split}
\end{equation}
By symmetry we have
\begin{equation}
\label{Homd12}
\underline{Hom}(d_2^*L_2, d_1^!L_1)\simeq d^!(p_2^*\mathbb{D}(L_2)\otimes p_1^*L_1).
\end{equation}
Then by~\eqref{Homc12}, \eqref{Homd12} and \eqref{eePQ}, we define a map
\begin{equation}
\label{e*KE}
c_*\underline{Hom}(c_1^*L_1, c_2^!L_2)\otimes d_*\underline{Hom}(d_2^*L_2, d_1^!L_1)
\to 
e_*\mathcal{K}_E
\end{equation}
as the composition
\begin{align}
\begin{split}
&c_*\underline{Hom}(c_1^*L_1, c_2^!L_2)\otimes d_*\underline{Hom}(d_2^*L_2, d_1^!L_1)\\
\simeq
&c_*c^!(p_1^*\mathbb{D}(L_1)\otimes p_2^*L_2)\otimes d_*d^!(p_2^*\mathbb{D}(L_2)\otimes p_1^*L_1)\\
\xrightarrow{\eqref{eePQ}}
&e_*e^!(p_1^*\mathbb{D}(L_1)\otimes p_2^*L_2\otimes p_2^*\mathbb{D}(L_2)\otimes p_1^*L_1)\\
\simeq
&e_*e^!(p_1^*L_1\otimes p_1^*\mathbb{D}(L_1)\otimes p_2^*L_2\otimes p_2^*\mathbb{D}(L_2))\\
\xrightarrow{\epsilon_{L_1}\otimes\epsilon_{L_2}}
&e_*e^!(p_1^*\mathcal{K}_{X_1}\otimes p_2^*\mathcal{K}_{X_2})
\overset{\eqref{Kunneth_K}}{\simeq}
e_*e^!\mathcal{K}_{X_{12}}=e_*\mathcal{K}_E.
\end{split}
\end{align}

\subsubsection{}
\label{num:vp_proper}
We now construct a proper functoriality of the map~\eqref{e*KE}.
Let $X'_1$ and $X'_2$ be two schemes, and denote by $X'_{12}=X'_1\times_kX'_2$. 
Let $f_1:X_1\to X'_1$ and $f_2:X_2\to X'_2$ be two morphisms, and 
denote by $f=f_1\times_kf_2:X_{12}\to X'_{12}$ their product. 
For $i=1,2$, we denote by $p'_i:X'_{12}\to X_i$ the projection, 
and let $L_i\in\mathbf{T}_c(X_i)$. 
There is a canonical isomorphism
\begin{equation}
\label{eq:3.3.1}
f_*\underline{Hom}(p_1^*L_1,p_2^!L_1)
\simeq
\underline{Hom}(p'^*_1f_{1!}L_1,p'^!_2f_{2*}L_2)
\end{equation}
given by the composition
\begin{align}
\begin{split}
f_*\underline{Hom}(p_1^*L_1,p_2^!L_1)
&\overset{\eqref{dual_hom}}{\simeq}
f_*(p_1^*\mathbb{D}(L_1)\otimes p_2^*L_2)
\overset{\eqref{Kunneth_*}}{\simeq}
p'^*_1f_{1*}\mathbb{D}(L_1)\otimes p'^*_2f_{2*}L_2\\
&\quad\simeq
 p'^*_1\mathbb{D}(f_{1!}L_1)\otimes p'^*_2f_{2*}L_2
\overset{\eqref{dual_hom}}{\simeq}
\underline{Hom}(p'^*_1f_{1!}L_1,p'^!_2f_{2*}L_2)
\end{split}
\end{align}
where we use the canonical duality isomorphism $\mathbb{D}(f_{1!}L_1)\simeq f_{1*}\mathbb{D}(L_1)$ (\cite[Corollary 4.4.24]{CD1}).

\subsubsection{}
Now we consider another Cartesian square of schemes
\begin{align}
\begin{gathered}
  \xymatrix{
    E' \ar[r]^-{} \ar[d]_-{} & D' \ar[d]^-{}\\
    C' \ar[r]^-{} & X'_{12}
  }
\end{gathered}
\end{align}
such that there is a commutative cube
\begin{align}
\begin{gathered}
  \xymatrix{
    E \ar[rr]^-{} \ar[dd]_-{} \ar[rd]^-{f_E} &  & D \ar[rd]^-{f_D} \ar[dd]^-{d} &\\
    & E' \ar[rr]^-{} \ar[dd]_-{} & & D' \ar[dd]^-{d'} \\
    C \ar[rr]^-{\ \ c} \ar[rd]_-{f_C} & & X_{12} \ar[rd]^-{f} &\\
    & C' \ar[rr]^-{c'} & & X'_{12}.
  }
\end{gathered}
\end{align}
Assume that $f_1,f_2,c,c',d,d'$ are proper. Consider the following commutative diagram
\begin{align}
\begin{gathered}
  \xymatrix{
  C \ar[r]_-{\underline{c}} \ar@/^-.5pc/[rd]^-{f_C} \ar@/^1pc/[rr]^-{c} & C'\times_{X'_{12}}X_{12} \ar[r]_-{c'_f} \ar[d]^-{f_{C'}} & X_{12} \ar[d]^-{f}\\
   & C' \ar[r]^-{c'} & X'_{12}
  }
\end{gathered}
\end{align}
where $\underline{c}$ is necessarily proper. There is a natural transformation of functors
\begin{equation}
\label{eq:trans_fcc}
f_*c_*c^!=c'_*f_{C'*}\underline{c}_*\underline{c}^!c'^!_f
\xrightarrow{ad'_{(\underline{c}_*,\underline{c}^!)}}
c'_*f_{C'*}c'^!_f
\overset{\eqref{upper!lower*}}{\simeq}
c'_*c'^!f_*.
\end{equation}

For $i=1,2$, denote by $c'_i=p'_i\circ c':C'\to X'_i$, $d'_i=p'_i\circ d':D'\to X'_i$. Then there is a canonical map
\begin{equation}
\label{eq:pf_corr1}
f_*c_*\underline{Hom}(c_1^*L_1, c_2^!L_2)
\to
c'_*\underline{Hom}(c'^*_1f_{1*}L_1, c'^!_2f_{2*}L_2)
\end{equation}
given by the composition
\begin{align}
\begin{split}
f_*c_*\underline{Hom}(c_1^*L_1, c_2^!L_2)
&\overset{\eqref{Ex^!Hom}}{\simeq}
f_*c_*c^!\underline{Hom}(p_1^*L_1, p_2^!L_2)
\xrightarrow{\eqref{eq:trans_fcc}}
c'_*c'^!f_*\underline{Hom}(p_1^*L_1, p_2^!L_2)\\
&\overset{\eqref{eq:3.3.1}}{\simeq}
c'_*c'^!\underline{Hom}(p'^*_1f_{1!}L_1,p'^!_2f_{2*}L_2)
\overset{\eqref{Ex^!Hom}}{\simeq}
c'_*\underline{Hom}(c'^*_1f_{1*}L_1, c'^!_2f_{2*}L_2).
\end{split}
\end{align}
By symmetry we have a map
\begin{equation}
\label{eq:pf_corr2}
f_*d_*\underline{Hom}(d_2^*L_2, d_1^!L_1)
\to
d'_*\underline{Hom}(d'^*_2f_{2*}L_2, d'^!_1f_{1*}L_1).
\end{equation}

\begin{proposition}
\label{proper_pf}
The square
\begin{align}
\begin{gathered}
  \xymatrix{
    f_*c_*\underline{Hom}(c_1^*L_1, c_2^!L_2)\otimes f_*d_*\underline{Hom}(d_2^*L_2, d_1^!L_1) \ar[r]^-{} \ar[d]_-{} & f_*e_*\mathcal{K}_E \ar[d]^-{}\\
    c'_*\underline{Hom}(c'^*_1f_{1*}L_1, c'^!_2f_{2*}L_2)\otimes d'_*\underline{Hom}(d'^*_2f_{2*}L_2, d'^!_1f_{1*}L_1) \ar[r]^-{\eqref{e*KE}} & e'_*\mathcal{K}_{E'}
  }
\end{gathered}
\end{align}
is commutative, where the left vertical map is given by the maps~\eqref{eq:pf_corr1} and~\eqref{eq:pf_corr2}, and the upper horizontal row is the composition
\begin{align}
\begin{split}
        &f_*c_*\underline{Hom}(c_1^*L_1, c_2^!L_2)\otimes f_*d_*\underline{Hom}(d_2^*L_2, d_1^!L_1)\\
      &\xrightarrow{\eqref{eq:Ex_lower*_tensor}}  f_*(c_*\underline{Hom}(c_1^*L_1, c_2^!L_2)\otimes d_*\underline{Hom}(d_2^*L_2, d_1^!L_1)) \xrightarrow{\eqref{e*KE}} f_*e_*\mathcal{K}_E.
 \end{split}
\end{align}
\end{proposition}

The proof is identical to the one given in \cite[III 4.4]{SGA5}, see also \cite[Theorem 3.3.2]{YZ18}.

\begin{remark}
\label{rk:lci_contravariant}
Following \cite[4.2.1]{DJK} (see also~\ref{recall_pur_trans} below), there is a natural transformation $f^*\to f^!$, given a lci morphism $f$ together with an isomorphism $Th_X(L_f)\simeq\mathbbold{1}_X$, where the left hand side is the Thom space of the trivial virtual tangent bundle.
\footnote{In particular this condition means that $f$ has relative dimension $0$. If $\mathbf{T}_c$ is \emph{oriented} (i.e. endowed with a trivialization of Thom spaces of all vector bundles), then any lci morphism of relative dimension $0$ satisfy this property.}
For example \'etale morphisms satisfy this property, but the class of such morphisms is bigger.
The map~\eqref{e*KE} is also compatible with pull-backs along these morphisms, which we do not develop here.
\end{remark}

\begin{definition}
\label{def:verdier_pairing}
With the notations of~\eqref{abstract_map}, for two maps $u:c_1^*L_1\to c_2^!L_2$ and $v:d_2^*L_2\to d_1^!L_1$ which are called \textbf{(cohomological) correspondences}, we define the \textbf{Verdier pairing}
\begin{equation}
\label{verdier_pairing}
\langle u,v\rangle:\mathbbold{1}_{E}\to\mathcal{K}_E
\end{equation}
obtained by adjunction from the composition
\begin{align}
\begin{split}
&\mathbbold{1}_{X_{12}}
\to
c_*\mathbbold{1}_C\otimes d_*\mathbbold{1}_D
\xrightarrow{c_*\eta_{L_1}\otimes d_*\eta_{L_2}}
c_*\underline{Hom}(c_1^*L_1,c_1^*L_1)\otimes d_*\underline{Hom}(d_2^*L_2,d_2^*L_2)\\
&\xrightarrow{u_*\otimes v_*}
c_*\underline{Hom}(c_1^*L_1,c_2^!L_2)\otimes d_*\underline{Hom}(d_2^*L_2,d_1^!L_1)
\xrightarrow{\eqref{e*KE}}
e_*\mathcal{K}_E.
\end{split}
\end{align}
\end{definition}

\begin{remark}
\begin{enumerate}
\item The map $\langle u,v\rangle$ can be seen as an element of the \emph{bivariant group} $H_0(X/k)$, and the map~\eqref{eq:prop_pf_1k} below corresponds to the natural proper functoriality (see \cite{DJK}). We will come back to this point of view later in Section~\ref{section:CC}.
\item In the case where every scheme is equal to $\operatorname{Spec}(k)$, 
the assumptions imply that every object $L\in\mathbf{T}_c(k)$ is dualizable and 
$\mathbb{D}(L)=L^\vee$ is the dual object. In this case, given two maps 
$u:L_1\to L_2$, $v:L_2\to L_1$, the map $\langle u,v\rangle$ is the composition
\begin{align}
\begin{split}
\mathbbold{1}_k
\xrightarrow{\eta_{L_1}\otimes\eta_{L_2}}
&L_1^\vee\otimes L_1\otimes L_2^\vee\otimes L_2
\xrightarrow{u\otimes1\otimes v\otimes1}
L_1^\vee\otimes L_2\otimes L_2^\vee\otimes L_1\\
\simeq
&L_1\otimes L_1^\vee\otimes L_2\otimes L_2^\vee
\xrightarrow{\epsilon_{L_1}\otimes\epsilon_{L_2}}
\mathbbold{1}_k.
\end{split}
\end{align}
It is not hard to check that $\langle u,v\rangle$ agrees with the trace of the composition $v\circ u:L_1\to L_1$. We will see a more general result in Proposition~\ref{prop:bil_tr} below.

\item For $\mathbf{T}_c=\mathbf{DM}_{cdh,c}$, we recover the construction in \cite[5.8]{Ols} as a particular case of the Verdier pairing.
\end{enumerate}
\end{remark}

\subsubsection{}
Consider the setting of~\ref{num:vp_proper}. We can define the \textbf{proper direct image} of correspondences as follows: given two correspondences $u:c_1^*L_1\to c_2^!L_2$ and $v:d_2^*L_2\to d_1^!L_1$, using the maps~\eqref{eq:pf_corr1} and~\eqref{eq:pf_corr2}, we obtain the following maps
\begin{align}
&f_*u:c'^*_1f_{1*}L_1\to c'^!_2f_{2*}L_2,\\
&f_*v:d'^*_2f_{2*}L_2\to d'^!_1f_{1*}L_1
\end{align}
regarded as correspondences between $f_{1*}L_1$ and $f_{2*}L_2$. On the other hand, since the canonical morphism $f_E:E\to E'$ is proper, for any map $w:\mathbbold{1}_{E}\to\mathcal{K}_E$ we define the proper direct image as
\begin{equation}
\label{eq:prop_pf_1k}
f_{E*}w:
\mathbbold{1}_{E'}
\xrightarrow{}
f_{E*}\mathbbold{1}_{E}
\xrightarrow{w}
f_{E*}\mathcal{K}_E
=
f_{E*}f_E^!\mathcal{K}_{E'}
\xrightarrow{ad'_{(f_{E*},f_E^!)}}
\mathcal{K}_{E'}.
\end{equation}
With the conventions above, the following is a consequence of Proposition~\ref{proper_pf}:
\begin{corollary}
\label{cor:vp_proper}
The Verdier pairing is compatible with proper direct images, i.e. we have
\begin{align}
\langle f_*u,f_*v\rangle=f_{E*}\langle u,v\rangle:\mathbbold{1}_{E'}\to\mathcal{K}_{E'}.
\end{align}

\end{corollary}

When $X'_1=X'_2=C'=D'=\operatorname{Spec}(k)$, the maps $f_*u:f_{1*}L_1\to f_{2*}L_2$ and $f_*v:f_{2*}L_2\to f_{1*}L_1$ are maps between dualizable objects in $\mathbf{T}_c(\operatorname{Spec}(k))$, and the map $f_{E*}w$ is known as the degree map and is traditionally written as $\int_Ew$.
As a particular case of Corollary~\ref{cor:vp_proper}, we obtain the following Lefschetz-Verdier formula (\cite[III 4.7]{SGA5}):
\begin{corollary}[Lefschetz-Verdier formula]
\label{Lef_Ver}
When $X'_1=X'_2=C'=D'=\operatorname{Spec}(k)$, we have the following equality
\begin{align}
Tr(f_*v\circ f_*u)=\int_E\langle u,v\rangle
\end{align}
as an endomorphism of $\mathbbold{1}_{k}$.

\end{corollary}

\begin{remark}

The formula in \cite[Theorem 1.3]{Hoy} has a similar appearance, but is indeed of different nature.

\end{remark}

\subsection{Composition of correspondences and generalized traces}

In this section, we study compositions of correspondences and show that the general Verdier pairing~\eqref{verdier_pairing} can be reduced to a generalized trace map, which will be a key ingredient for the general form of additivity of traces in Section~\ref{section:additivity}.

\subsubsection{}

Let $X_1,X_2$ and $X_3$ be three schemes. Denote by $X_{ijl}=X_i\times_kX_j\times_kX_l$, and $p^{ijl}_i:X_{ijl}\to X_i$ and $p^i:X_i\to k$ the canonical projections, and we use similar notations for other schemes and morphisms.

For $i\in\{1,2,3\}$, let $L_i$ be an object of $\mathbf{T}_c(X_{i})$. We have a canonical map
\begin{equation}
\label{eq:hom_comp}
p^{123*}_{12}\underline{Hom}(p^{12*}_1L_1,p^{12!}_2L_2)
\otimes
p^{123*}_{23}\underline{Hom}(p^{23*}_2L_2,p^{23!}_3L_3)
\to
p^{123!}_{13}\underline{Hom}(p^{13*}_1L_1,p^{13!}_3L_3)
\end{equation}
given by the composition
\begin{align}
\begin{split}
&p^{123*}_{12}\underline{Hom}(p^{12*}_1L_1,p^{12!}_2L_2)
\otimes
p^{123*}_{23}\underline{Hom}(p^{23*}_2L_2,p^{23!}_3L_3)\\
&\overset{\eqref{dual_hom}}{\simeq}
p^{123*}_{1}\mathbb{D}(L_1)\otimes p^{123*}_{2}L_2\otimes p^{123*}_{2}\mathbb{D}(L_2)\otimes p^{123*}_{3}L_3\\
&\xrightarrow{\epsilon_{L_2}}
p^{123*}_{1}\mathbb{D}(L_1)\otimes p^{123*}_{2}\mathcal{K}_{X_2}\otimes p^{123*}_{3}L_3
\simeq
p^{123*}_{2}\mathcal{K}_{X_2}\otimes p^{123*}_{1}\mathbb{D}(L_1)\otimes p^{123*}_{3}L_3\\
&\overset{\eqref{Kunneth}}{\simeq}
p^{123!}_{13}(p^{13*}_{1}\mathbb{D}(L_1)\otimes p^{13*}_{3}L_3)
\overset{\eqref{dual_hom}}{\simeq}
p^{123!}_{13}\underline{Hom}(p^{13*}_1L_1,p^{13!}_3L_3).
\end{split}
\end{align}

\subsubsection{}
\label{num:comp_corr}
Now consider two morphisms $c_{12}:C_{12}\to X_{12}$ and $c_{23}:C_{23}\to X_{23}$. Let $C_{13}=C_{12}\times_{X_2}C_{23}$ together with a canonical morphism $c^{13}_{123}:C_{13}\to X_{123}$. We denote by $c_{13}=p^{123}_{13}\circ c^{13}_{123}:C_{13}\to X_{13}$, and $c^{ij}_i=p^{ij}_i\circ c_{ij}:C_{ij}\to X_i$, etc. Consider the following diagram
\begin{align}
\begin{gathered}
  \xymatrix{
    C_{13} \ar[d]_-{q'_1} \ar[r]^-{q'_3} \ar@/^.65cm/[rr]^-{q_3} \ar@/_.6cm/[dd]_-{q_1} \ar[rd]^-{c^{13}_{123}\ \ \ \ \ \ } & C'_{23} \ar[d]^-{c'_{23}} \ar[r]^-{q_{23}} & C_{23} \ar[d]^-{c_{23}} \\
    C'_{12} \ar[r]_-{c'_{12}} \ar[d]^-{q_{12}} & X_{123} \ar[r]^-{p^{123}_{23}} \ar[d]^-{p^{123}_{12}} & X_{23} \ar[d]^-{p^{23}_{2}}\\
  C_{12} \ar[r]^-{c_{12}} & X_{12} \ar[r]^-{p^{12}_{2}} & X_{2}
  }
\end{gathered}
\end{align}
where all the four squares are Cartesian. For any objects $P\in\mathbf{T}_c(X_{12})$ and $Q\in\mathbf{T}_c(X_{23})$, we have a canonical map
\begin{equation}
\label{eq:trick2}
\begin{split}
q_1^*c_{12}^!P\otimes q_3^*c_{23}^!Q
&=q_1'^*q_{12}^*c_{12}^!P\otimes q_3'^*q_{23}^*c_{23}^!Q\\
&\xrightarrow{\eqref{Ex*!}}
q_1'^*c_{12}'^!p_{12}^{123*}P\otimes q_3'^*c_{23}'^!p_{23}^{123*}Q
\xrightarrow{\eqref{eq:cdPQ}}
c_{123}^{13!}(p_{12}^{123*}P\otimes p_{23}^{123*}Q).
\end{split}
\end{equation}
Therefore we deduce a map
\begin{equation}
\label{eq:c_comp}
q_1^*\underline{Hom}(c_1^{12*}L_1,c_2^{12!}L_2)
\otimes
q_3^*\underline{Hom}(c_2^{23*}L_2,c_3^{23!}L_3)
\to
\underline{Hom}(c_1^{13*}L_1,c_3^{13!}L_3)
\end{equation}
given by the composition
\begin{align}
\begin{split}
&q_1^*\underline{Hom}(c_1^{12*}L_1,c_2^{12!}L_2)
\otimes
q_3^*\underline{Hom}(c_2^{23*}L_2,c_3^{23!}L_3)\\
&\overset{\eqref{Ex^!Hom}}{\simeq}
q_1^*c_{12}^!\underline{Hom}(p_1^{12*}L_1,p_2^{12!}L_2)
\otimes
q_3^*c_{23}^!\underline{Hom}(p_2^{23*}L_2,p_3^{23!}L_3)\\
&\xrightarrow{\eqref{eq:trick2}}
c_{123}^{13!}(
p^{123*}_{12}\underline{Hom}(p^{12*}_1L_1,p^{12!}_2L_2)
\otimes
p^{123*}_{23}\underline{Hom}(p^{23*}_2L_2,p^{23!}_3L_3))\\
&\xrightarrow{\eqref{eq:hom_comp}}
c_{123}^{13!}p^{123!}_{13}\underline{Hom}(p^{13*}_1L_1,p^{13!}_3L_3)\\
&=
c_{13}^!\underline{Hom}(p^{13*}_1L_1,p^{13!}_3L_3)
\overset{\eqref{Ex^!Hom}}{\simeq}
\underline{Hom}(c^{13*}_1L_1,c^{13!}_3L_3).
\end{split}
\end{align}
\begin{definition}
\label{def:comp_corr}
Given two correspondences $u:c_1^{12*}L_1\to c_2^{12!}L_2$ and $v:c_2^{23*}L_2\to c_3^{23!}L_3$, we deduce from the map~\eqref{eq:c_comp} a correspondence from $L_1$ to $L_3$
\begin{align}
vu:c^{13*}_1L_1\to c^{13!}_3L_3
\end{align}
which we call the \textbf{composition} of the correspondences $u$ and $v$.
\end{definition}

\subsubsection{}
Now assume that $X_1=X_3$ and $L_1=L_3$. Consider two morphisms $c:C\to X_{12}$, $d:D\to X_{21}\simeq X_{12}$. For $i\in\{1,2\}$, we denote by $c_i=p^{12}_{i}\circ c:C\to X_i$ and $d_i=p^{21}_{i}\circ d:D\to X_i$ the canonical maps. Let $F=C\times_{X_2}D$ and $E=C\times_{X_{12}}D$, then there is a Cartesian diagram of the form
\begin{align}
\begin{gathered}
  \xymatrix{
    E \ar[r]^-{f'} \ar[d]_-{\delta'} & X_{12} \ar[r]^-{p^{12}_1} \ar[d]_-{\delta_1} & X_1 \ar[d]^-{\delta}\\
    F \ar[r]^-{f} & X_{121} \ar[r]^-{p^{121}_{11}} & X_{11}
    }
\end{gathered}
\end{align}
and $E$ is canonically identified with the fiber product $F\times_{X_{11}}X_1$ via the diagonal map $\delta:X_1\to X_{11}$. As in~\ref{num:comp_corr} we denote by $q_1:F\to C$ and $q_3:F\to D$ the canonical maps, and $f_1,f_3$ the compositions
\begin{align}
&f_1:F
\xrightarrow{q_1}
C
\xrightarrow{c}
X_{12}
\xrightarrow{p^{12}_1}
X_1\\
&f_3:F
\xrightarrow{q_3}
D
\xrightarrow{d}
X_{21}
\xrightarrow{p^{21}_1}
X_1.
\end{align}
Suppose that $u:c_1^*L_1\to c_2^!L_2$ and $v:d_2^*L_2\to d_1^!L_1$ are two correspondences. By Definition~\ref{def:comp_corr}, there is a composition $vu:f_1^*L_1\to f_3^!L_1$. The following is stated in \cite[III (5.2.10)]{SGA5} without proof:
\begin{proposition}
\label{prop:bil_tr}
The Verdier pairing~\eqref{verdier_pairing} satisfies the following equality
\begin{align}
\langle u,v\rangle=\langle vu,1\rangle:\mathbbold{1}_E\to\mathcal{K}_E
\end{align}
where $1$ is the identity correspondence $id:\mathbbold{1}_{X_1}\to\mathbbold{1}_{X_1}$.
\end{proposition}

\proof

Denote by $p^{11}_1,p^{11}_{3}:X_{11}\to X_1$ the projections to the first and the second summands. We have a canonical map
\begin{equation}
\label{eq:new_unit}
\mathbbold{1}_{X_1}
\xrightarrow{\eta_{L_1}}
\underline{Hom}(L_1,L_1)
\overset{\eqref{Ex^!Hom}}{\simeq}
\delta^!\underline{Hom}(p^{11*}_3L,p^{11!}_1L)
\overset{\eqref{dual_hom}}{\simeq}
\delta^!(p^{11*}_3\mathbb{D}(L)\otimes p^{11*}_1L),
\end{equation}
from which we deduce a canonical map
\begin{equation}
\label{eq:trick_unit}
\delta'^*f^*f_*\underline{Hom}(f_1^*L_1,f_3^!L_1)\to f'^!p_1^{12!}(\mathbb{D}(L_1)\otimes L_1)
\end{equation}
given by the composition
\begin{align}
\begin{split}
&\delta'^*f^*f_*\underline{Hom}(f_1^*L_1,f_3^!L_1)
\overset{\eqref{Ex^!Hom}}{\simeq}
\delta'^*f^*f_*f^!p^{121!}_{11}\underline{Hom}(p^{11*}_1L_1,p^{11!}_3L_1)\\
&\overset{\eqref{Ex^!Hom}}{\simeq}
\delta'^*f^*f_*f^!p^{121!}_{11}(p^{11*}_1\mathbb{D}(L_1)\otimes p^{11*}_3L_1)\\
&\xrightarrow{\eqref{eq:new_unit}}
\delta'^*f^*(f_*f^!p^{121!}_{11}(p^{11*}_1\mathbb{D}(L_1)\otimes p^{11*}_3L_1)\otimes \delta_{1*}p^{12*}_1\delta^!(p^{11*}_3\mathbb{D}(L)\otimes p^{11*}_1L))\\
&\xrightarrow{\eqref{Ex*!}}
\delta'^*f^*(f_*f^!p^{121!}_{11}(p^{11*}_1\mathbb{D}(L_1)\otimes p^{11*}_3L_1)\otimes \delta_{1*}\delta_1^!p^{121*}_{11}(p^{11*}_3\mathbb{D}(L)\otimes p^{11*}_1L))\\
&\xrightarrow{\eqref{eePQ}}
\delta'^*f^*f_*\delta'_*\delta'^!f^!(p^{121!}_{11}(p^{11*}_1\mathbb{D}(L_1)\otimes p^{11*}_3L_1)\otimes p^{121*}_{11}(p^{11*}_3\mathbb{D}(L)\otimes p^{11*}_1L)))\\
&\xrightarrow{ad'_{((f\circ\delta')^*,(f\circ\delta')_*)}}
\delta'^!f^!(p^{121!}_{11}(p^{11*}_1\mathbb{D}(L_1)\otimes p^{11*}_3L_1)\otimes p^{121*}_{11}(p^{11*}_3\mathbb{D}(L)\otimes p^{11*}_1L)))\\
&\xrightarrow{\eqref{eq:nat_upper*!_upper!}}
\delta'^!f^!p^{121!}_{11}(p^{11*}_1\mathbb{D}(L_1)\otimes p^{11*}_3L_1\otimes p^{11*}_3\mathbb{D}(L)\otimes p^{11*}_1L)\\
&\xrightarrow{\epsilon_{L_1}}
\delta'^!f^!p^{121!}_{11}(p^{11*}_1\mathbb{D}(L_1)\otimes p^{11*}_1L\otimes p^{11*}_3\mathcal{K}_{X_1})\\
&\overset{\eqref{Kunneth}}{\simeq}
\delta'^!f^!p^{121!}_{11}p^{11!}_1(\mathbb{D}(L_1)\otimes L_1)
=
f'^!p_1^{12!}(\mathbb{D}(L_1)\otimes L_1).
\end{split}
\end{align}
Note that we have natural transformations $f'^*c_*\to f^*f_*q^*_1$ and $f'^*d_*\to f^*f_*q^*_3$. We want to show that both the maps $\langle u,v\rangle$ and $\langle vu,1\rangle$ are equal to the following composition
$$
\mathbbold{1}_E
\xrightarrow{}
\delta'^*f^*f_*\mathbbold{1}_F
\xrightarrow{vu}
\delta'^*f^*f_*\underline{Hom}(f_1^*L_1,f_3^!L_1)
\xrightarrow{\eqref{eq:trick_unit}}
f'^!p_1^{12!}(\mathbb{D}(L_1)\otimes L_1)
\xrightarrow{\epsilon_{L_1}}
f'^!p_1^{12!}\mathcal{K}_{X_1}
=
\mathcal{K}_E.
$$
This follows from the commutativty of the following diagram
$$
\resizebox{\textwidth}{!}{
  \xymatrix{
    \mathbbold{1}_E \ar[r]^-{} \ar[d]_-{} \ar[rd]^-{} & f'^*(c_*\mathbbold{1}_C\otimes d_*\mathbbold{1}_D) \ar[r]^-{u\otimes v} & f'^*(c_*\underline{Hom}(c^{*}_1L_1,c^{!}_2L_2)\otimes d_*\underline{Hom}(d^{*}_2L_2,d^{!}_1L_1)) \ar[d]_-{} \ar@/^5cm/[dd]^-{\eqref{e*KE}} \\
    \delta'^*f^*p^{121*}_{11}p^{121}_{11*}f_*\mathbbold{1}_F \ar[r]^-{} \ar[d]_-{vu} & \delta'^*f^*f_*\mathbbold{1}_F \ar[d]_-{vu} & \delta'^*f^*f_*(q_1^*\underline{Hom}(c^{*}_1L_1,c^{!}_2L_2)\otimes q_3^*d_*\underline{Hom}(d^{*}_2L_2,d^{!}_1L_1)) \ar[ld]_-{\eqref{eq:c_comp}} \\
     \delta'^*f^*p^{121*}_{11}p^{121}_{11*}f_*\underline{Hom}(f_1^*L_1,f_3^!L_1) \ar[r]^-{} \ar@/_.7cm/[rr]^-{} & \delta'^*f^*f_*\underline{Hom}(f_1^*L_1,f_3^!L_1) \ar[r]^-{\eqref{eq:trick_unit}} & f'^!p_1^{12!}(\mathbb{D}(L_1)\otimes L_1)
  }
}
$$
where each subdiagram follows either from definition or from a straightforward check.
\endproof

\begin{remark}
Proposition~\ref{prop:bil_tr} says that the Verdier pairing~\eqref{verdier_pairing} can always be reduced to the case where $X_1=X_2$, $L_1=L_2$ and one of the correspondences is the identity. As such we reduce the pairing with two entries to a generalized trace map, therefore making it much easier to deal with additivity along distinguished triangles.

\end{remark}

\subsubsection{}
\label{num:gen_tr}
We now give a more explicit description of the map $\langle u,1\rangle$ 
(see \cite[Proposition 2.1.7]{AS}
\footnote{The statement in loc. cit. holds for $c$ a closed immersion, 
but our modified version holds in general.}
). Let $X$ be a scheme and $c:C\to X\times_kX$ be a morphism. We use the notation
 in~\ref{abstract_map}, with $D=X$, $d=\delta:X\to X\times_kX$ 
 the diagonal morphism and the Cartesian diagram
\begin{align}
\begin{gathered}
  \xymatrix{
    E \ar[r]^-{c'} \ar[d]_-{\delta'} \ar@{}[rd]|{\Delta} & X \ar[d]^-{\delta}\\
    C \ar[r]_-{c} & X\times_kX.
  }
\end{gathered}
\end{align}
\begin{proposition}
\label{prop:pair_trace}
Let $L\in\mathbf{T}_c(X)$ and $u:c_1^*L\to c_2^!L$ be a correspondence. Denote by $1=id_L:L\to L$ the identity correspondence, and by $u'$ the following map:
\begin{equation}
\begin{split}
\mathbbold{1}_C
&\xrightarrow{\eta_{c_1^*L}}
\underline{Hom}(c_1^*L,c_1^*L)
\xrightarrow{u_*}
\underline{Hom}(c_1^*L,c_2^!L)\\
&\overset{\eqref{Ex^!Hom}}{\simeq}
c^!\underline{Hom}(p_1^*L,p_2^!L)
\overset{\eqref{dual_hom}}{\simeq}
c^!(p_1^*\mathbb{D}(L)\otimes p_2^*L).
\end{split}
\end{equation}
Then the map $\langle u,1\rangle:\mathbbold{1}_E\to\mathcal{K}_E$ is obtained by adjunction from the map
\begin{equation}
\label{eq:gen_tr}
\begin{split}
c'_!\mathbbold{1}_E
&\overset{(Ex(\Delta^*_!))^{-1}}{\simeq}
\delta^*c_!\mathbbold{1}_C
\xrightarrow{u'}
\delta^*c_!c^!(p_1^*\mathbb{D}(L)\otimes p_2^*L)\\
&\xrightarrow{ad'_{(c_!,c^!)}}
\delta^*(p_1^*\mathbb{D}(L)\otimes p_2^*L)
=
\mathbb{D}(L)\otimes L
\simeq
L\otimes\mathbb{D}(L)
\xrightarrow{\epsilon_L}
\mathcal{K}_X.
\end{split}
\end{equation}
\end{proposition}

\proof Similarly to the map~\eqref{eq:new_unit}, we denote by $\eta'$ the following map
\begin{align}
\begin{split}
\mathbbold{1}_X
\xrightarrow{\eta_L}
&\underline{Hom}(L,L)
=
\underline{Hom}(\delta^*p_2^*L,\delta^!p_1^!L)\\
\overset{\eqref{Ex^!Hom}}{\simeq}
&\delta^!\underline{Hom}(p_2^*L,p_1^!L)
\overset{\eqref{dual_hom}}{\simeq}
\delta^!(p_2^*\mathbb{D}(L)\otimes p_1^*L).
\end{split}
\end{align}
We are reduced to show the commutativity of the following diagram:
$$
\resizebox{\textwidth}{!}{
  \xymatrix{
    \mathbbold{1}_E \ar[r]^-{ad_{(c'_!,c'^!)}} \ar[d]_-{u'} & c'^!c'_!\mathbbold{1}_E \ar[r]^-{(Ex(\Delta^*_!))^{-1}}_-{\sim} & c'^!\delta^*c_!\mathbbold{1}_C \ar[d]^-{u'}\\
    \delta'^*c^!(p_1^*\mathbb{D}(L)\otimes p_2^*L) \ar[r]^-{\eqref{Ex*!}} \ar[d]_-{\eta'} & c'^!\delta^*(p_1^*\mathbb{D}(L)\otimes p_2^*L) \ar[d]_-{\eta'} \ar@{=}[rd] & c'^!\delta^*c_!c^!(p_1^*\mathbb{D}(L)\otimes p_2^*L) \ar[l]_-{ad'_{(c_!,c^!)}}\\
   \delta'^*c^!(p_1^*\mathbb{D}(L)\otimes p_2^*L)\otimes c'^*\delta^!(p_2^*\mathbb{D}(L)\otimes p_1^*L) \ar[r]^-{\eqref{Ex*!}} \ar[d]_-{\eqref{eq:cdPQ}} & c'^!\delta^*(p_1^*\mathbb{D}(L)\otimes p_2^*L)\otimes c'^*\delta^!(p_2^*\mathbb{D}(L)\otimes p_1^*L) \ar[ld]_-{\eqref{eq:nat_upper*!_upper!}} & c'^!(\mathbb{D}(L)\otimes L) \ar[d]^-{\wr} \\
  c'^!\delta^!(p_1^*\mathbb{D}(L)\otimes p_2^*L\otimes p_2^*\mathbb{D}(L)\otimes p_1^*L) \ar[r]^-{\sim} & c'^!\delta^!(p_1^*L\otimes p_1^*\mathbb{D}(L)\otimes p_2^*L\otimes p_2^*\mathbb{D}(L)) \ar[d]_-{\epsilon_L\otimes\epsilon_L} & c'^!(L\otimes\mathbb{D}(L)) \ar[d]^-{\epsilon_L} \\
  & c'^!\delta^!(p_1^*\mathcal{K}_X\otimes p_2^*\mathcal{K}_X) \ar[r]^-{\sim} & c'^!\mathcal{K}_X.
  }
}
$$
We show the commutativity of the octagon, while the rest of the diagram follows from a straightforward check. We are reduced to the following diagram:
$$
\resizebox{\textwidth}{!}{
  \xymatrix{
    \delta^*(p_1^*\mathbb{D}(L)\otimes p_2^*L) \ar@{=}[r]^-{} \ar[d]_-{\eta'} \ar[rd]^-{ad_{(\delta_!,\delta^!)}} & \mathbb{D}(L)\otimes L \ar[r]^-{\sim} & L\otimes\mathbb{D}(L) \ar[d]_-{ad_{(\delta_!,\delta^!)}} \ar@/^1.6cm/[ddd]_-{\epsilon_L} \\
    \delta^*(p_1^*\mathbb{D}(L)\otimes p_2^*L)\otimes\delta^!(p_2^*\mathbb{D}(L)\otimes p_1^*L)  \ar[d]_-{\eqref{eq:nat_upper*!_upper!}} & \delta^!\delta_!\delta^*(p_1^*\mathbb{D}(L)\otimes p_2^*L) \ar[r]^-{\sim} & \delta^!\delta_!(L\otimes\mathbb{D}(L)) \ar[d]_-{ad_{(q_!,q^!)}} \\
    \delta^!(p_1^*\mathbb{D}(L)\otimes p_2^*L\otimes p_2^*\mathbb{D}(L)\otimes p_1^*L) \ar[d]_-{\wr} & \delta^!(p_1^*\mathbb{D}(L)\otimes p_2^*L\otimes\delta_!\mathbbold{1}_X) \ar[u]^-{Ex(\delta^*_!,\otimes)}_-{\wr} \ar[ld]_-{\eta'} & p^!p_!(L\otimes\mathbb{D}(L)) \ar[d]_-{\epsilon_L} \\
    \delta^!(p_1^*L\otimes p_1^*\mathbb{D}(L)\otimes p_2^*L\otimes p_2^*\mathbb{D}(L)) \ar[r]^-{\epsilon_L\otimes\epsilon_L} \ar[rru]_-{\epsilon_L} & \delta^!(p_1^*\mathcal{K}_X\otimes p_2^*\mathcal{K}_X) \ar[r]^-{\eqref{Kunneth_K}}_-{\sim} & \mathcal{K}_X
  }
}
$$
where $p:X\to\operatorname{Spec}(k)$ and $q:X\times_kX\to\operatorname{Spec}(k)$ are structural morphisms. We are reduced to the commutativity of the pentagon, namely the following diagram:
$$
\resizebox{\textwidth}{!}{
  \xymatrix{
    \delta_!\delta^*(p_1^*\mathbb{D}(L)\otimes p_2^*L) \ar[r]^-{\sim} & \delta_!(L\otimes\mathbb{D}(L)) \ar[r]^-{ad_{(q_!,q^!)}} & q^!q_!\delta_!(L\otimes\mathbb{D}(L)) \\
    p_1^*\mathbb{D}(L)\otimes p_2^*L\otimes\delta_!\mathbbold{1}_X \ar[d]_-{\eta'} \ar[r]^-{ad_{(q_!,q^!)}} \ar[u]^-{Ex(\delta^*_!,\otimes)}_-{\wr} & q^!q_!(p_1^*\mathbb{D}(L)\otimes p_2^*L\otimes\delta_!\mathbbold{1}_X) \ar[d]_-{\eta'} \ar[ru]^-{\sim} & \\
    p_1^*\mathbb{D}(L)\otimes p_2^*L\otimes\delta_!\delta^!(p_2^*\mathbb{D}(L)\otimes p_1^*L) \ar[r]^-{ad_{(q_!,q^!)}} \ar[d]_-{ad'_{(\delta_!,\delta^!)}} & q^!q_!(p_1^*\mathbb{D}(L)\otimes p_2^*L\otimes\delta_!\delta^!(p_2^*\mathbb{D}(L)\otimes p_1^*L)) \ar[d]_-{ad_{(\delta_!,\delta^!)}} & q^!(p_!(L\otimes\mathbb{D}(L))\otimes p_!\mathcal{K}_X) \ar[d]_-{\eqref{Kun'}}^-{\wr} \ar[uu]^-{ad'_{(p_!,p^!)}} \\
    p_1^*L\otimes p_1^*\mathbb{D}(L)\otimes p_2^*L\otimes p_2^*\mathbb{D}(L) \ar[r]^-{ad_{(q_!,q^!)}} & q^!q_!(p_1^*L\otimes p_1^*\mathbb{D}(L)\otimes p_2^*L\otimes p_2^*\mathbb{D}(L)) \ar[r]^-{\epsilon_L} & q^!q_!(p_1^*(L\otimes\mathbb{D}(L))\otimes p_2^*\mathcal{K}_X).
  }
}
$$
By considering the hexagon in the diagram above, we are reduced to the following diagram:
$$
\resizebox{\textwidth}{!}{
  \xymatrix{
    q_!(p_1^*\mathbb{D}(L)\otimes p_2^*L\otimes\delta_!\mathbbold{1}_X) \ar[r]^-{\sim} \ar[d]_-{\eta'} & q_!\delta_!(L\otimes\mathbb{D}(L)) \ar[d]_-{ad'_{(\delta_!,\delta^!)}} & \\
    q_!(p_1^*\mathbb{D}(L)\otimes p_2^*L\otimes\delta_!\delta^!(p_2^*\mathbb{D}(L)\otimes p_1^*L)) \ar[d]_-{ad'_{(\delta_!,\delta^!)}} & q_!p_1^!(L\otimes\mathbb{D}(L)) & p_!(L\otimes\mathbb{D}(L))\otimes p_!\mathcal{K}_X \ar[lu]_-{ad'_{(p_!,p^!)}} \ar[ld]_-{\eqref{Kun'}}^-{\sim}\\
    q_!(p_1^*\mathbb{D}(L)\otimes p_2^*L\otimes p_2^*\mathbb{D}(L)\otimes p_1^*L) \ar[r]^-{\epsilon_L} & q_!(p_1^*(L\otimes\mathbb{D}(L))\otimes p_2^*\mathcal{K}_X). \ar[u]^-{\eqref{Kunneth}}_-{\wr}  &  
    }
}
$$
The right part is a straightforward check, and the left part is reduced to the following diagram:
$$
  \xymatrix{
    p_2^*L\otimes\delta_!\mathbbold{1}_X \ar[r]^-{\eta'} \ar[d]_-{Ex(\delta^*_!,\otimes)}^-{\wr} & p_2^*L\otimes p_2^*\mathbb{D}(L)\otimes p_1^*L \ar[r]^-{\epsilon_L} & p_1^*L\otimes p_2^*\mathcal{K}_X \ar[d]^-{\eqref{Kunneth}}_-{\wr}\\
    \delta_!L \ar[rr]^-{ad'_{(\delta_!,\delta^!)}} & & p_1^!L
    }
$$
which commutes since the composition
\begin{align}
\delta_!L
=
\delta_!\delta^*p_2^*L
\xrightarrow{}
\delta_!(\delta^*p_2^*L\otimes\underline{Hom}(\delta^*p_2^*L,\delta^!p_1^!L))
\xrightarrow{}
\delta_!\delta^!p_1^!L
=
\delta_!L
\end{align}
is the identity map.
\endproof

\begin{remark}
In the special case where $c=\delta:X\to X\times_kX$ is the diagonal map, the K\"unneth formulas indeed produce a map
\begin{equation}
\label{eq:form_counit}
\mathbbold{1}_X
=
\delta^*\delta_!\mathbbold{1}_X
\xrightarrow{\eta'}
\delta^*(p_2^*\mathbb{D}(L)\otimes p_1^*L)
=
\mathbb{D}(L)\otimes L.
\end{equation}
Together with the map $\epsilon_L:L\otimes\mathbb{D}(L)\to\mathcal{K}_X$, they can be seen as the counit and unit maps of a duality formalism similar to the usual (strong) duality, where the usual dualizing functor is replaced by the local duality functor $\mathbb{D}$, which gives rise to trace maps without requiring $L$ to be strongly dualizable; the general Verdier pairing is a more general form of the trace map in such a duality formalism. In Section~\ref{section:additivity} we will combine this point of view with the approach in \cite{May} in the study of additivity of traces. In particular, as mentioned in the introduction, the \emph{characteristic class} of $L$ is the composition
\begin{align}
\mathbbold{1}_X
\xrightarrow{\eqref{eq:form_counit}}
\mathbb{D}(L)\otimes L
\simeq
L\otimes\mathbb{D}(L)
\xrightarrow{\epsilon_L}
\mathcal{K}_X
\end{align}
which will be studied in more details in Section~\ref{section:CC}.

\end{remark}

\section{Additivity of the Verdier pairing}
\label{section:additivity}
In this section we prove the additivity of the Verdier pairing following \cite{May} and \cite{GPS}, using the language of derivators (\cite{Ayo}, \cite{GPS}).

\subsection{May's axioms in stable derivators}
\label{sect:stab_der}

In this section we recall the notion of closed symmetric monoidal stable derivators and obtain May's axioms following \cite{GPS}; the statements we need are slightly different from {\it loc.cit.} and can be obtained with very minor changes from the original proof. Since we are mostly interested in constructible motives, we only consider finite diagrams for convenience.
\begin{notation}

We denote by $FinCat$ the $2$-category of finite categories, $CAT$ the $2$-category of categories and $TR$ the subcategory of $CAT$ of triangulated categories and triangulated functors.

We denote by $\emptyset$ the empty category, $\underline{\mathbf{0}}$ the terminal category and $\underline{\mathbf{1}}=(0\to 1)$ the category with two objects and one non-identity morphism between them. We denote by $\Box$ the category $\underline{\mathbf{1}}\times\underline{\mathbf{1}}$, written as
\begin{align}
\begin{gathered}
  \xymatrix@=10pt{
   (0,0) \ar[d]_-{} \ar[r]_-{} & (0,1)  \ar[d]_-{} \\
   (1,0) \ar[r]_-{} & (1,1).
  }
\end{gathered}
\end{align}
We denote by $\lefthalfcap$ and respectively $\righthalfcup$ the full subcategories $\Box\setminus\{(1,1)\}$ and $\Box\setminus\{(0,0)\}$, with $i_\lefthalfcap:\lefthalfcap\to\Box$ and $i_\righthalfcup:\righthalfcup\to\Box$ the inclusions.

\end{notation}

\begin{definition}
\label{def:stable_der}
A \textbf{(strong) stable derivator} is a (non-strict) $2$-functor $\mathcal{T}_c:FinCat^{op}\to CAT$ satisfying the following properties:
\begin{enumerate}
\item $\mathcal{T}_c$ sends coproducts to products. In particular $\mathcal{T}_c(\emptyset)=\underline{\mathbf{0}}$.
\item 
For any $I,J\in FinCat$, the canonical functor $\mathcal{T}_c(I\times J)\to Fun(I^{op},\mathcal{T}_c(J))$ is conservative.
\item For any $A\in FinCat$, the canonical functor $\mathcal{T}_c(A\times\underline{\mathbf{1}})\to Fun(\underline{\mathbf{1}},\mathcal{T}_c(A))$ is full and essentially surjective.
\item For every functor $u:A\to B$ in $FinCat$, the functor $u^*:\mathcal{T}_c(B)\to\mathcal{T}_c(A)$ has a right adjoint $u_*$ and a left adjoint $u_\#$.
\item For any functor $u:A\to B$ and object $b$ of $B$, denote by $j_{A/b}:A/b\to A$, $j_{b\textbackslash A}:b\textbackslash A\to A$, $p_{A/b}:A/b\to\underline{\mathbf{0}}$ and $p_{b\textbackslash A}:b\textbackslash A\to\underline{\mathbf{0}}$ the canonical projections. Then the following canonical transformations are invertible:
\begin{align}
&b^*u_*
\to
p_{b\textbackslash A*}p_{b\textbackslash A}^*b^*u_*
\to
p_{b\textbackslash A*}j_{A/b}^*u^*u_*
\to
p_{b\textbackslash A*}j_{A/b}^*\\
&
p_{b\textbackslash A\#}j_{b\textbackslash A}^*
\to
p_{b\textbackslash A\#}j_{b\textbackslash A}^*u^*u_\#
\to
p_{b\textbackslash A\#}p_{b\textbackslash A}^*b^*u_\#
\to
b^*u_\#.
\end{align}

\item For any $I\in FinCat$ the category $\mathcal{T}_c(I)$ has a zero object, i.e. an object which is both initial and terminal.

\item For any $I\in FinCat$ and $X$ an object of $\mathcal{T}_c(\Box\times I)$, $X$ is Cartesian (i.e. the canonical map $X\to i_{\righthalfcup*}i_\righthalfcup^*X$ is invertible) if and only if $X$ is coCartesian (i.e. the canonical map $i_{\lefthalfcap\#}i_\lefthalfcap^*X\to X$ is invertible). We also say that $X$ is \textbf{biCartesian}.
\end{enumerate}

An object in $\mathcal{T}_c(A)$ is called an \textbf{($A$-shaped) coherent diagram}, who has an underlying \textbf{incoherent diagram} in $Fun(A,\mathcal{T}_c(\underline{\mathbf{0}}))$.

The \textbf{cofiber functor} $\operatorname{cof}:\mathcal{T}_c(\underline{\mathbf{1}})\to\mathcal{T}_c(\underline{\mathbf{1}})$ is the composition 
\begin{align}
\mathcal{T}_c(\underline{\mathbf{1}})
\xrightarrow{(0,\cdot)_*}
\mathcal{T}_c(\lefthalfcap)
\xrightarrow{(i_{\lefthalfcap})_\#}
\mathcal{T}_c(\Box)
\xrightarrow{(\cdot,1)^*}
\mathcal{T}_c(\underline{\mathbf{1}}).
\end{align}

If $\mathcal{T}_c$ is a stable derivator, for any $I\in FinCat$ we define a functor $\Sigma:\mathcal{T}_c(I)\to\mathcal{T}_c(I)$ by setting $\Sigma=(1,1)^*(i_\lefthalfcap)_\#(i_\lefthalfcap)^*(0,0)_*$. For a biCartesian $X\in\mathcal{T}_c(\Box\times I)$ depicted as
\begin{align}
\begin{gathered}
  \xymatrix{
   x \ar[d]_-{} \ar[r]^-{f} & y \ar[d]^-{g} \\
   0 \ar[r]_-{} & z
  }
\end{gathered}
\end{align}
with $x,y,z\in\mathcal{T}_c(I)$, we define a canonical map $z\xrightarrow{h}\Sigma x$ as follows:
\begin{align}
h:z
=
(1,1^*)X
\xleftarrow{\sim}
(1,1^*)(i_\lefthalfcap)_\#(i_\lefthalfcap)^*X
\xrightarrow{}
(1,1^*)(i_\lefthalfcap)_\#(i_\lefthalfcap)^*(0,0)_*(0,0)^*X
=
\Sigma(0,0)^*X
=
\Sigma x.
\end{align}
Then the category $\mathcal{T}_c(I)$ has the structure of a triangulated category by considering $\Sigma$ as the shift functor and letting distinguished triangles to be the ones isomorphic to a triangle of the form
\begin{align}
x\xrightarrow{f}y\xrightarrow{g}z\xrightarrow{h}\Sigma x.
\end{align}
Therefore we can also see stable derivators as $2$-functors $\mathcal{T}_c:FinCat^{op}\to TR$.

\end{definition}

\begin{notation}
We denote by $SMTR$ the $2$-category of symmetric monoidal triangulated categories with (strong) monoidal functors.

As in \cite{GPS}, if $\odot:\mathcal{T}_{c,1}\times\mathcal{T}_{c,2}\to\mathcal{T}_{c,3}$ is a two-variable morphism of stable derivators, it induces an \textbf{internal product} denoted as 
\begin{align}
\odot_A:\mathcal{T}_{c,1}(A)\times\mathcal{T}_{c,2}(A)\to\mathcal{T}_{c,3}(A).
\end{align}
We define the \textbf{external product} $\odot:\mathcal{T}_{c,1}(A)\times\mathcal{T}_{c,2}(B)\to\mathcal{T}_{c,3}(A\times B)$ as the composition
\begin{align}
\mathcal{T}_{c,1}(A)\times\mathcal{T}_{c,2}(B)
\xrightarrow{\pi_A^*\times\pi_B^*}
\mathcal{T}_{c,1}(A\times B)\times\mathcal{T}_{c,2}(A\times B)
\xrightarrow{\odot_{A\times B}}
\mathcal{T}_{c,3}(A\times B).
\end{align}
where $\pi_A:A\times B\to A$ and $\pi_B:A\times B\to B$ are the canonical projections.
\end{notation}

\begin{definition}
\label{def:sym_mon_stab_der}
A \textbf{symmetric monoidal stable derivator} is a $2$-functor $(\mathcal{T}_c,\otimes):FinCat^{op}\to SMTR$ such that
\begin{enumerate}
\item The composition $FinCat^{op}\xrightarrow{\mathcal{T}_c}SMTR\xrightarrow{}TR$ is a stable derivator.
\item For any $A,B,C\in FinCat$, $u:A\to B$, $X\in\mathcal{T}_c(A)$, $Y\in\mathcal{T}_c(C)$, the following canonical map of external products is an isomorphism:
\begin{align}
\begin{split}
(u\times1)_{\#}(X\otimes Y)
&\to
(u\times1)_{\#}(u^*u_\#X\otimes Y)\\
&\xrightarrow{\sim}
(u\times1)_{\#}(u\times1)^*(u_\#X\otimes Y)
\to
u_\#X\otimes Y.
\end{split}
\end{align}

\end{enumerate}
It is \textbf{closed} if the functor $\otimes$ has a right adjoint $\underline{Hom}(\cdot,\cdot)$.
\end{definition}

The following is \cite[Corollary 4.5]{GPS}:
\begin{proposition}[TC1]

Let $\mathcal{T}_c$ be a symmetric monoidal stable derivator and denote by $\mathbbold{1}$ the unit in $\mathcal{T}_c(\underline{\mathbf{0}})$.Then for any object $x$ in $\mathcal{T}_c(\underline{\mathbf{0}})$ there is a natural equivalence $\alpha:\Sigma x\simeq x\otimes\Sigma\mathbbold{1}$ such that the composition
\begin{align}
\Sigma\Sigma\mathbbold{1}
\xrightarrow{\alpha}
\Sigma\mathbbold{1}\otimes\Sigma\mathbbold{1}
\overset{ s}{\simeq}
\Sigma\mathbbold{1}\otimes\Sigma\mathbbold{1}
\xrightarrow{\alpha^{-1}}
\Sigma\Sigma\mathbbold{1}
\end{align}
is the multiplication by $-1$, where $ s$ is the isomorphism that exchanges the two summands.

\end{proposition}

The following is \cite[Theorems 4.8 and 9.12]{GPS}:

\begin{proposition}[TC2]
\label{prop:TC2}
Let $\mathcal{T}_c$ be a closed symmetric monoidal stable derivator. Then for any distinguished triangle
\begin{align}
x\xrightarrow{f}y\xrightarrow{g}z\xrightarrow{h}\Sigma x
\end{align}
in $\mathcal{T}_c(\underline{\mathbf{0}})$ and any $t\in\mathcal{T}_c(\underline{\mathbf{0}})$, the following triangles are distinguished:
\begin{align}
x\otimes t
\xrightarrow{f\otimes 1}
y\otimes t
\xrightarrow{g\otimes 1}
z\otimes t
\xrightarrow{h\otimes 1}
\Sigma(x\otimes t).
\end{align}
\begin{align}
t\otimes x
\xrightarrow{1\otimes f}
t\otimes y
\xrightarrow{1\otimes g}
t\otimes z
\xrightarrow{1\otimes h}
\Sigma(t\otimes x).
\end{align}
\begin{align}
\Sigma^{-1}\underline{Hom}(x,t)
\xrightarrow{-\underline{Hom}(h,t)}
\underline{Hom}(z,t)
\xrightarrow{\underline{Hom}(g,t)}
\underline{Hom}(y,t)
\xrightarrow{\underline{Hom}(f,t)}
\underline{Hom}(x,t).
\end{align}
\begin{align}
\underline{Hom}(t,x)
\xrightarrow{\underline{Hom}(t,f)}
\underline{Hom}(t,y)
\xrightarrow{\underline{Hom}(t,g)}
\underline{Hom}(t,z)
\xrightarrow{\underline{Hom}(t,h)}
\Sigma\underline{Hom}(t,z).
\end{align}
\end{proposition}

\subsubsection{}
The following notions are specific to derivators and play a key role in the additivity of traces:

\begin{definition}
\label{def:cancel_prod}
Let $A$ be a small category. The \textbf{twisted arrow category} $tw(A)$ is defined as follows: its objects are morphisms $a\xrightarrow{f}b$ in $A$, and morphisms from $a_1\xrightarrow{f_1}b_1$ to $a_2\xrightarrow{f_2}b_2$ are pairs of morphisms $b_1\xrightarrow{h}b_2$ and $a_2\xrightarrow{g}b_1$ such that $f_2=hf_1g$. There is a canonical map $(t^{op},s^{op}):tw(A)^{op}\to A^{op}\times A$ sending $a\xrightarrow{f}b$ to $(b,a)$.

If $\mathcal{T}_c$ is a stable derivator and $X\in\mathcal{T}_c(A^{op}\times A)$, the \textbf{coend} of $X$ is defined as
\begin{align}
\int^AX:=(\pi_{tw(A)^{op}})_\#(t^{op},s^{op})^*X\in\mathcal{T}_c(\underline{\mathbf{0}})
\end{align}
where $\pi_{tw(A)^{op}}:tw(A)^{op}\to\underline{\mathbf{0}}$ is the canonical map.

If $\odot:\mathcal{T}_{c,1}\times\mathcal{T}_{c,2}\to\mathcal{T}_{c,3}$ is a two-variable morphism of stable derivators and $X\in\mathcal{T}_{c,1}(A^{op})$, $Y\in\mathcal{T}_{c,2}(A)$, the \textbf{cancelling tensor product} of $X$ and $Y$ is defined as
\begin{align}
X\odot_{[A]}Y:=\int^A(X\odot Y)\in\mathcal{T}_{c,3}(\underline{\mathbf{0}}).
\end{align}

\end{definition}

If $\mathcal{T}_c$ is a closed symmetric monoidal derivator, we have two natural two-variable morphisms $\otimes:\mathcal{T}_c\times\mathcal{T}_c\to\mathcal{T}_c$ and $\underline{Hom}:\mathcal{T}_c^{op}\times\mathcal{T}_c\to\mathcal{T}_c$. We denote by $X\otimes_{[A]}Y$ and $\underline{Hom}_{[A]}(X,Y)$ respectively the corresponding cancelling tensor products.

\begin{proposition}[TC3]
\label{prop:TC3}
Let $\mathcal{T}_c$ be a closed symmetric monoidal derivator and let $x\xrightarrow{f}y$ and $x'\xrightarrow{f'}y'$ be two maps which give rise to distinguished triangles in $\mathcal{T}_c(\underline{\mathbf{0}})$
\begin{align}
x\xrightarrow{f}y\xrightarrow{g}z\xrightarrow{h}\Sigma x.
\end{align}
\begin{align}
x'\xrightarrow{f'}y'\xrightarrow{g'}z'\xrightarrow{h'}\Sigma x'.
\end{align}
Then the following properties hold:
\begin{enumerate}
\item
For 
$v:=\underline{Hom}_{[\underline{\mathbf{1}}]}(\operatorname{cof}(f),f')$ there are distinguished triangles
\begin{align}
\underline{Hom}(y,x')\xrightarrow{p_1}v\xrightarrow{j_1}\underline{Hom}(z,z')\xrightarrow{\underline{Hom}(g,h')}\Sigma\underline{Hom}(y,x')
\end{align}
\begin{align}
\Sigma^{-1}\underline{Hom}(x,z')\xrightarrow{p_2}v\xrightarrow{j_2}\underline{Hom}(y,y')\xrightarrow{-\underline{Hom}(f,g')}\underline{Hom}(x,z')
\end{align}
\begin{align}
\underline{Hom}(z,y')\xrightarrow{p_3}v\xrightarrow{j_3}\underline{Hom}(x,x')\xrightarrow{\underline{Hom}(h,f')}\Sigma\underline{Hom}(z,y')
\end{align}
with a coherent diagram of the form
$$
\resizebox{\textwidth}{!}{
  \xymatrix{
    \Sigma^{-1}\underline{Hom}(y,z') \ar[rd]_(.15){\Sigma^{-1}\underline{Hom}(f,1_{z'})\ \ \ \ } \ar@/_2cm/[d]_-{\Sigma^{-1}\underline{Hom}(1_y,h')} & \underline{Hom}(z,x') \ar[ld]_(.85){\underline{Hom}(g,1_{x'})\ \ } \ar[rd]^(.85){\ \ \ \ \underline{Hom}(1_z,f')} & \Sigma^{-1}\underline{Hom}(x,y') \ar@/^2cm/[d]^-{\underline{Hom}(h,1_{y'})} \ar[ld]^(.15){\ \ \ \ \Sigma^{-1}\underline{Hom}(1_x,g')}\\
    \underline{Hom}(y,x') \ar@/_2cm/[dd]_-{\underline{Hom}(f,1_{x'})} \ar[rd]^(.4){p_1} \ar[rdd]_(.3){\underline{Hom}(1_y,f')} & \Sigma^{-1}\underline{Hom}(x,z') \ar[d]^-{p_2} \ar[rdd]^(.7){\underline{Hom}(h,1_{z'})} \ar[ldd]_(.7){\Sigma^{-1}\underline{Hom}(1_x,h')} & \underline{Hom}(z,y') \ar[ld]_(.4){p_3} \ar@/^2cm/[dd]^-{\underline{Hom}(1_z,g')} \ar[ldd]^(.3){\underline{Hom}(g,1_{y'})}\\
      & v \ar[ld]^(.6){j_3} \ar[d]^-{j_2} \ar[rd]_(.6){j_1} & \\
    \underline{Hom}(x,x') \ar@/_2cm/[d]_-{\underline{Hom}(1_x,f')} \ar[rd]_(.15){\Sigma\underline{Hom}(h,1_{x'})\ \ \ \ } & \underline{Hom}(y,y') \ar[rd]^(.85){\ \ \ \ \underline{Hom}(1_y,g')} \ar[ld]_(.85){\underline{Hom}(f,1_{y'})\ \ } & \underline{Hom}(z,z') \ar@/^2cm/[d]^-{\underline{Hom}(g,1_{z'})} \ar[ld]^(.15){\ \ -\underline{Hom}(1_z,h')}\\
    \underline{Hom}(x,y')  & \Sigma\underline{Hom}(z,x') & \underline{Hom}(y,z').
  }
}
$$
\item
For $w:=\underline{Hom}_{[\underline{\mathbf{1}}]}(\operatorname{cof}(f),\operatorname{cof}(f'))$, there are distinguished triangles
\begin{align}
\underline{Hom}(z,z')\xrightarrow{k_1}w\xrightarrow{q_1}\underline{Hom}(x,y')\xrightarrow{\underline{Hom}(h,g')}\Sigma\underline{Hom}(z,z')
\end{align}
\begin{align}
\underline{Hom}(y,y')\xrightarrow{k_2}w\xrightarrow{q_2}\Sigma\underline{Hom}(z,x')\xrightarrow{-\Sigma\underline{Hom}(g,f')}\Sigma\underline{Hom}(y,y')
\end{align}
\begin{align}
\underline{Hom}(x,x')\xrightarrow{k_3}w\xrightarrow{q_3}\underline{Hom}(y,z')\xrightarrow{\underline{Hom}(f,h')}\Sigma\underline{Hom}(x,x')
\end{align}
with a similar coherent diagram.

\item
For $u:=f\otimes_{[\underline{\mathbf{1}}]}\operatorname{cof}(f')$,
there are distinguished triangles
\begin{align}
x\otimes z'\xrightarrow{l_1}u\xrightarrow{r_1}z\otimes y'\xrightarrow{h\otimes g'}\Sigma x\otimes z'
\end{align}
\begin{align}
y\otimes y'\xrightarrow{l_2}u\xrightarrow{r_2}\Sigma x\otimes x'\xrightarrow{-\Sigma(g\otimes f')}\Sigma y\otimes y'
\end{align}
\begin{align}
z\otimes x'\xrightarrow{l_3}u\xrightarrow{r_3}y\otimes z'\xrightarrow{f\otimes h'}\Sigma z\otimes x'
\end{align}
with a similar coherent diagram.
\end{enumerate}

We call these statements respectively $(TC3D)$, $(TC3D')$ and $(TC3')$.

\end{proposition}

\proof
The proof of $(TC3D)$ is the same as that of \cite[Theorem 6.2]{GPS}, where we replace everywhere $\otimes$ by $\underline{Hom}$.
The statement of $(TC3')$ is proved in \cite[Section 7]{GPS}, and $(TC3D')$ follows from a similar argument.
\endproof

The following has the same proof as \cite[Theorem 7.3]{GPS}:
\begin{proposition}[TC4]
\label{prop:TC4}
With the notations in Proposition~\ref{prop:TC3}, there is a biCartesian square
\begin{align}
\begin{gathered}
  \xymatrix{
    v \ar^-{j_2}[r] \ar_-{(j_1,j_3)}[d] & \underline{Hom}(y,y') \ar^-{k_2}[d]\\
    \underline{Hom}(z,z')\oplus\underline{Hom}(x,x') \ar^-{k_1+k_3}[r] & w.
  }
\end{gathered}
\end{align}
\end{proposition}
Note that when $x,y,z$ are dualizable, up to taking their duals, Propositions~\ref{prop:TC3} and~\ref{prop:TC4} are precisely \cite[Theorems 6.2 and 7.3]{GPS}.

\subsubsection{}
Let $\mathcal{T}_c$ be a closed symmetric monoidal derivator. Consider a distinguished triangle in $\mathcal{T}_c(\underline{\mathbf{0}})$
\begin{align}
x\xrightarrow{f}y\xrightarrow{g}z\xrightarrow{h}\Sigma x
\end{align}
and let $t\in\mathcal{T}_c(\underline{\mathbf{0}})$. By Proposition~\ref{prop:TC2}, we have a distinguished triangle
\begin{align}
\underline{Hom}(z,t)
\xrightarrow{\underline{Hom}(g,t)}
\underline{Hom}(y,t)
\xrightarrow{\underline{Hom}(f,t)}
\underline{Hom}(x,t)
\xrightarrow{\underline{Hom}(\Sigma^{-1}h,t)}
\Sigma\underline{Hom}(z,t).
\end{align}
By \cite[Lemma 7.1]{GPS}, we have 
\begin{align}
\underline{Hom}(f,t)\otimes_{[\underline{\mathbf{1}}]}f
\simeq
\underline{Hom}(g,t)\otimes_{[\underline{\mathbf{1}}]}g
\simeq
\underline{Hom}(h,t)\otimes_{[\underline{\mathbf{1}}]}h
\end{align}
and we denote by $u$ this object. The following follows from the proof of \cite[Theorem 10.3]{GPS}:
\begin{proposition}[TC5a]
\label{prop:unit_dual}

With the notations above, there is a map $\bar{\epsilon}:u\to t$ in $\mathcal{T}_c(\underline{\mathbf{0}})$ such that the following incoherent diagrams commute:
\begin{align}
\begin{gathered}
  \xymatrix{
    \underline{Hom}(x,t)\otimes x \ar_-{\epsilon_L}[rd] \ar^-{l_1}[r] & u \ar^-{\bar{\epsilon}}[d] & \underline{Hom}(y,t)\otimes y \ar_-{\epsilon_y}[rd] \ar^-{l_2}[r] & u \ar^-{\bar{\epsilon}}[d] & \underline{Hom}(z,t)\otimes z \ar_-{\epsilon_z}[rd] \ar^-{l_3}[r] & u \ar^-{\bar{\epsilon}}[d] \\
    & t & & t & & t.
  }
\end{gathered}
\end{align}

\end{proposition}

\subsubsection{}
We now discuss the last one of May's axioms, where we work with local duality instead of the usual duality. The following definition is standard (\cite[Definition 4.4.4]{CD1}):
\begin{definition}

We say that an object $t\in\mathcal{T}_c(\underline{\mathbf{0}})$ is \textbf{dualizing} if for any $x\in\mathcal{T}_c(\underline{\mathbf{0}})$, the following canonical map is an isomorphism:
\begin{align}
\label{eq:biduality}
x
\to
\underline{Hom}(\underline{Hom}(x,t),t).
\end{align}
We denote by $\mathbb{D}_t:=\underline{Hom}(\cdot,t):\mathcal{T}_c^{op}\to\mathcal{T}_c$ the \textbf{$t$-dual} functor. We have clearly $\mathbb{D}_t\circ\mathbb{D}_t=id$.
\end{definition}

\begin{lemma}
\label{lemma:dual_hom}
If $t\in\mathcal{T}_c(\underline{\mathbf{0}})$ is a dualizing object, then for any $a\in\mathcal{T}_c(A)$ and $b\in\mathcal{T}_c(B)$, the following canonical map is an isomorphism in $\mathcal{T}_c(A^{op}\times B)$:
\begin{equation}
\label{eq:bidual_hom}
\underline{Hom}(a,b)
\to
\mathbb{D}_t(a\otimes\mathbb{D}_t(b)).
\end{equation}
\proof

The proof is the same as \cite[Corollary 4.4.24]{CD1}: we have a canonical isomorphism
$$
\mathbb{D}_t(a\times c)
\simeq
\underline{Hom}(a,\mathbb{D}_t(c))
$$
and since $t$ is dualizing, the result follows by replacing $c$ by $\mathbb{D}_t(b)$ in the previous map.
\endproof

\end{lemma}

Thanks to Lemma~\ref{lemma:dual_hom}, the proof of the following is similar to \cite[Theorem 11.12]{GPS}:
\begin{proposition}[TC5b]

Consider a distinguished triangle
\begin{align}
x\xrightarrow{f}y\xrightarrow{g}z\xrightarrow{h}\Sigma x
\end{align}
in $\mathcal{T}_c(\underline{\mathbf{0}})$ and let $t\in\mathcal{T}_c(\underline{\mathbf{0}})$ be a dualizing object. Then the (TC3') diagram specified in (TC5a) for the triangles $(\mathbb{D}_tg,\mathbb{D}_tf,\mathbb{D}_t\Sigma^{-1}h)$ and $(f,g,h)$ is isomorphic to the $t$-dual of the (TC3D) diagram for the triangles $(f,g,h)$ and $(f,g,h)$.

\end{proposition}

\begin{remark}
The original (TC5b) statement also requires the (TC4) axiom for the dual diagram up to an involution, which is also true in our case if we write down the corresponding (TC3) and (TC3D') diagrams; but such a fact is not used in \cite{GPS} for the proof of the additivity of traces.
\end{remark}

The following is obtained by taking the $t$-dual of the (TC5a):
\begin{corollary}
Assume that $\mathcal{T}_c(\underline{\mathbf{0}})$ has a dualizing object and consider a distinguished triangle in $\mathcal{T}_c(\underline{\mathbf{0}})$
\begin{align}
x\xrightarrow{f}y\xrightarrow{g}z\xrightarrow{h}\Sigma x.
\end{align}
Let $v$ be element specified in the (TC3D) diagram for the triangles $(f,g,h)$ and $(f,g,h)$.
Then there is a map $\bar{\eta}:\mathbbold{1}\to v$ in $\mathcal{T}_c(\underline{\mathbf{0}})$ such that the following incoherent diagrams commute:
\begin{align}
\begin{gathered}
  \xymatrix{
    \mathbbold{1} \ar^-{\eta_x}[rd] \ar_-{\bar{\eta}}[d] & & \mathbbold{1} \ar^-{\eta_y}[rd] \ar_-{\bar{\eta}}[d] & & \mathbbold{1} \ar^-{\eta_z}[rd] \ar_-{\bar{\eta}}[d] & \\
    v \ar_-{j_3}[r] & \underline{Hom}(x,x) & v \ar_-{j_2}[r] & \underline{Hom}(y,y) & v \ar_-{j_1}[r] & \underline{Hom}(z,z).
  }
\end{gathered}
\end{align}
\end{corollary}

\subsection{Motivic derivators and additivity of traces}

\begin{definition}

The $2$-category $DiaSch$ is defined as follows:
\begin{itemize}
\item An object of $DiaSch$ is a pair $(F,J)$ where $J\in FinCat$ and $F:J\to Sch$ is a covariant functor.
\item An $1$-morphism from $(G,J')$ to $(F,J)$ is the data of a functor $\alpha:J'\to J$ together with a natural transformation of functors $f:G\to F\circ\alpha$.
\item A $2$-morphism from $(f,\alpha)$ to $(f',\alpha')$ as above is a natural transformation $t:\alpha\to\alpha'$ such that $f'=t\circ f$.
\end{itemize}

We say that a $1$-morphism $(f,\alpha):(G,J')\to(F,J)$ is \textbf{Cartesian} if $\alpha$ is an equivalence of categories and for any morphism $i\to j$ in $J'$, the following square is Cartesian:
\begin{align}
\begin{gathered}
  \xymatrix{
    G(i) \ar^-{}[r] \ar_-{}[d] & G(j) \ar^-{}[d]\\
    F(\alpha(i)) \ar^-{}[r] & F(\alpha(i)).
  }
\end{gathered}
\end{align}

If $X\in Sch$ and $J\in FinCat$, we denote by $(X,J)$ the object in $DiaSch$ with constant value $X$.

\end{definition}

The following definition is almost identical to \cite[Definition 2.4.13]{Ayo}:
\begin{definition}
\label{def:stab_alg_der}
A \textbf{stable algebraic derivator} is a (non-strict) $2$-functor $\mathcal{T}_c:(DiaSch)^{op}\to TR$ satisfying the following properties:
\begin{enumerate}
\item $\mathcal{T}_c$ sends coproducts to products.
\item For any $1$-morphism $(f,\alpha):(F,J)\to(G,J')$ in $DiaSch$, the functor $(f,\alpha)^*$ has a right adjoint $(f,\alpha)_*$.
\item For any $1$-morphism $(f,\alpha):(F,J)\to(G,J')$ in $DiaSch$ which is termwise smooth, the functor $(f,\alpha)^*$ has a left adjoint $(f,\alpha)_\#$.

\item \label{enum:ex_der}
If $f:G\to F$ is a morphism of $J$-diagrams of schemes and $\alpha:J'\to J$ is a functor in $FinCat$, then the exchange $2$-morphism
\begin{align}
\alpha^*f_*\to(f_{|J'})_*\alpha^*
\end{align}
associated to the following Cartesian square in $DiaSch$ is invertible:
\begin{equation}
\label{diag:cart_sm_der}
\begin{gathered}
  \xymatrix{
    (G\circ\alpha,J') \ar^-{\alpha}[r] \ar_-{f_{|J'}}[d] & (G,J) \ar^-{f}[d]\\
    (F\circ\alpha,J') \ar^-{\alpha}[r] & (F,J).
  }
\end{gathered}
\end{equation}

\item In the situation of~\eqref{enum:ex_der}, if $f$ is Cartesian and termwise smooth, then the following exchange $2$-morphism associated the square~\eqref{diag:cart_sm_der} is invertible:
$$
(f_{|J'})_\#\alpha^*\to\alpha^*f_\#.
$$

\item For any $X\in Sch$, the $2$-functor
\begin{align}
\begin{split}
\mathcal{T}_c(X,\cdot):FinCat^{op}&\to TR\\
J&\mapsto\mathcal{T}_c(X,J)
\end{split}
\end{align}
is a stable derivator (Definition~\ref{def:stable_der}).
\item The $2$-functor
\begin{align}
\begin{split}
\mathcal{T}_c(\cdot,\underline{\mathbf{0}}):Sch^{op}&\to TR\\
X&\mapsto\mathcal{T}_c(X,\underline{\mathbf{0}})
\end{split}
\end{align}
is the subcategory of constructible objects in a motivic triangulated category $\mathbf{T}$ (\cite[Definition 4.2.1]{CD1}). We call $\mathbf{T}$ the \textbf{underlying motivic triangulated category}.
\end{enumerate}
\end{definition}

Note that by \cite[Corollary 4.4.24]{CD1} and \cite[Theorem 7.3]{CD2}, if the underlying motivic triangulated category satisfies \ref{resol}, then for any scheme $X$, $\mathcal{K}_X$ is a dualizing object of $\mathcal{T}_c(X,\underline{\mathbf{0}})$. We will then write $\mathbb{D}$ for $\mathbb{D}_{\mathcal{K}_X}$ in coherence with the previous sections.

\subsubsection{}
\label{num:4fct_ext}
By \cite[Section 2.4.4]{Ayo}, given a stable algebraic derivator $\mathcal{T}_c$, one can extend the four functors $f^,f_*,f^!,f_!$ to diagrams of schemes in the following form:
\begin{itemize}
\item For any $1$-morphism $(f,\alpha):(F,J)\to(G,J')$ in $DiaSch$, there is a pair of adjoint functors
\begin{align}
(f,\alpha)^*:\mathcal{T}_c(G,J')\rightleftharpoons\mathcal{T}_c(F,J):(f,\alpha)_*.
\end{align}
\item For any $J\in FinCat$ and any Cartesian $J$-shaped $1$-morphism $f:(F,J)\to(G,J)$ in $DiaSch$, there is a pair of adjoint functors
\begin{align}
f_!:\mathcal{T}_c(F,J)\rightleftharpoons\mathcal{T}_c(G,J):f^!.
\end{align}
\end{itemize}
If $f:X\to Y$ is a morphism of schemes, then the four functors associated to morphisms of the form $(f,J):(X,J)\to(Y,J)$ commute with finite limit and colimits along diagrams, and therefore commute with the coend construction (Definition~\ref{def:cancel_prod}): we have a commutative diagram
\begin{align}
\begin{gathered}
  \xymatrix{
    \mathcal{T}_c(Y,A^{op}\times A) \ar^-{f^*}[r] \ar_-{\int^A}[d] & \mathcal{T}_c(X,A^{op}\times A) \ar^-{\int^A}[d]\\
    \mathcal{T}_c(Y,\underline{\mathbf{0}}) \ar^-{f^*}[r] & \mathcal{T}_c(X,\underline{\mathbf{0}})
  }
\end{gathered}
\end{align}
and similarly for the other functors $f_*$, $f^!$ and $f_!$.

We now deal with the monoidal structure.
\begin{definition}
\label{def:cons_mot_der}
A \textbf{constructible motivic derivator} is a (non-strict) $2$-functor $(\mathcal{T}_c,\otimes):DiaSch^{op}\to SMTR$ satisfying the following properties:
\begin{enumerate}
\item The composition $DiaSch^{op}\xrightarrow{\mathcal{T}_c}SMTR\to TR$ is a stable algebraic derivator (Definition~\ref{def:stab_alg_der}), and the monoidal structure agrees with the one on the underlying motivic triangulated category.
\item For any scheme $X$, the $2$-functor
\begin{align}
\begin{split}
(\mathcal{T}_c(X,\cdot),\otimes):FinCat^{op}&\to SMTR\\
J&\mapsto(\mathcal{T}_c(X,J),\otimes)
\end{split}
\end{align}
is a closed symmetric monoidal stable derivator (Definition~\ref{def:sym_mon_stab_der}).
\item For any $J\in FinCat$, any Cartesian $J$-shaped $1$-morphism $f:(F,J)\to(G,J)$ in $DiaSch$ and any pair of objects $(A,B)\in\mathcal{T}_c(G,J)\times\mathcal{T}_c(F,J)$, the following canonical map is an isomorphism:
\begin{align}
f_\#(f^*A\otimes_{G,J}B)\to A\otimes_{F,J}f_\#B.
\end{align}
\end{enumerate}
\end{definition}

\subsubsection{}
We now apply the formalism to the generalized trace map~\eqref{eq:gen_tr}. The following is similar to \cite[Theorem 12.1]{GPS}:
\begin{proposition}
\label{prop:add_tr}
Let $\mathcal{T}_c$ be a constructible motivic derivator whose underlying motivic triangulated category $\mathbf{T}$ satisfies~\ref{resol}. We use the notations in~\ref{num:gen_tr}. Let 
\begin{align}
\begin{gathered}
  \xymatrix{
    L \ar^-{f}[r] \ar_-{}[d] \ar@{}[rd]|{\Gamma} & M \ar^-{g}[d]\\
    \ast \ar_-{}[r] & N
  }
\end{gathered}
\end{align}
be a biCartesian square in $\mathcal{T}_c(X,\Box)$, and let $\phi:c_1^*\Gamma\to c_2^!\Gamma$ be a morphism of squares in $\mathcal{T}_c(C,\Box)$. Then the pairing~\eqref{eq:gen_tr} satisfies
\begin{align}
\langle \phi_M ,1 \rangle=\langle \phi_L ,1 \rangle+\langle \phi_N ,1 \rangle
\end{align}
where $\phi_M:c_1^*M\to c_2^!M$ is the restriction of $\phi$ to a map in $\mathbf{T}_c(C)$, and similarly for the other maps.

\end{proposition}

\proof
It suffices to construct the following incoherent diagram in $\mathcal{T}_c(X,\underline{\mathbf{0}})$
\begin{align}\nonumber
\begin{gathered}
\xymatrix{
    & \delta^*c_!\mathbbold{1}_C \ar^-{\bar{\eta}}[d] \ar^-{\eta_M}[rd] \ar_-{(\eta_L,\eta_N)}[ld] &\\
   \delta^*c_!\left(\underline{Hom}(c_1^*L,c_1^*L)\oplus\underline{Hom}(c_1^*N,c_1^*N)\right) \ar^-{\phi_{L*}\oplus \phi_{N*}}[dd] \ar^-{}[rd]
   & \delta^*c_!v \ar^-{}[r] \ar^-{}[l] 
   & \delta^*c_!\underline{Hom}(c_1^*M,c_1^*M) \ar^-{}[ld] \ar_-{\phi_{M*}}[dd]\\
   & \delta^*c_!w \ar@{.>}^-{}[d]
   &\\
   \delta^*c_!\left(\underline{Hom}(c_1^*L,c_2^!L)\oplus\underline{Hom}(c_1^*N,c_2^!N)\right) \ar^-{}[r] 
   & \delta^*c_!w'
   & \delta^*c_!\underline{Hom}(c_1^*M,c_2^!M) \ar^-{}[l]\\
   \delta^*c_!c^!\left(\underline{Hom}(p_1^*L,p_2^!L)\oplus\underline{Hom}(p_1^*N,p_2^!N)\right) \ar^-{}[r] \ar^-{\wr}_-{Ex}[u] \ar^-{}[d]
   & \delta^*c_!c^!w'' \ar^-{\wr}[u] \ar^-{}[d]
   & \delta^*c_!c^!\underline{Hom}(p_1^*M,p_2^!M) \ar^-{}[l] \ar^-{\wr}_-{Ex}[u] \ar^-{}[d]\\
   \delta^*\left(\underline{Hom}(p_1^*L,p_2^!L)\oplus\underline{Hom}(p_1^*N,p_2^!N)\right) \ar^-{}[r]  
   & \delta^*w'' 
   & \delta^*\underline{Hom}(p_1^*M,p_2^!M) \ar^-{}[l] \\
   \left(L\otimes\mathbb{D}(L)\right)\oplus\left(N\otimes\mathbb{D}(N)\right) \ar^-{}[r] \ar_-{\wr}^{{s}\oplus{s}}[d] \ar^-{\wr}_-{Kun}[u] 
   & u \ar^-{\wr}[u] \ar_-{\wr}[d]
   & M\otimes\mathbb{D}(M) \ar^-{}[l] \ar^-{\wr}_-{Kun}[u] \ar_-{\wr}^{{s}}[d] \\
   \left(\mathbb{D}(L)\otimes L\right)\oplus\left(\mathbb{D}(N)\otimes N\right) \ar^-{}[r] \ar_-{\epsilon_L+\epsilon_N}[rd] 
   & u' \ar^-{\bar{\epsilon}}[d]
   & \mathbb{D}(M)\otimes M \ar^-{}[l] \ar^-{\epsilon_M}[ld] \\
   & \mathcal{K}_X &
 }
 \end{gathered}
\end{align}
where the objects $v$, $w$, $w'$, $w''$, $u$, $u'$ are specified by May's axioms in Section~\ref{sect:stab_der}:
\begin{itemize}
\item $v:=\underline{Hom}_{[\underline{\mathbf{1}}]}(c_1^*g,c_1^*f)$ is specified in the (TC3D) diagram for the triangles $(c_1^*f,c_1^*g,c_1^*h)$ and $(c_1^*f,c_1^*g,c_1^*h)$.
\item $w:=\underline{Hom}_{[\underline{\mathbf{1}}]}(c_1^*g,c_1^*g)$ is specified in the (TC3D') diagram for the triangles $(c_1^*f,c_1^*g,c_1^*h)$ and $(c_1^*f,c_1^*g,c_1^*h)$.
\item $w':=\underline{Hom}_{[\underline{\mathbf{1}}]}(c_1^*g,c_2^!g)$ is specified in the (TC3D') diagram for the triangles $(c_1^*f,c_1^*g,c_1^*h)$ and $(c_2^!f,c_2^!g,c_2^!h)$.
\item $w'':=\underline{Hom}_{[\underline{\mathbf{1}}]}(p_1^*g,p_2^!g)$ is specified in the (TC3D') diagram for the triangles $(p_1^*f,p_1^*g,p_1^*h)$ and $(p_2^!f,p_2^!g,p_2^!h)$.
\item $u:=g\otimes_{[\underline{\mathbf{1}}]}\mathbb{D}g$ is specified in the (TC3D') diagram for the triangles $(f,g,h)$ and $(\mathbb{D}g,\mathbb{D}f,\mathbb{D}\Sigma^{-1}h)$.
\item $u':=\mathbb{D}g\otimes_{[\underline{\mathbf{1}}]}g$ is specified in the (TC3') diagram for the triangles $(\mathbb{D}g,\mathbb{D}f,\mathbb{D}\Sigma^{-1}h)$ and $(f,g,h)$.
\end{itemize}
The commutativity of the diagram follows from the following facts:
\begin{itemize}
\item The two triangles on the top commute by (TC5b), and the two on the bottom commute by (TC5a); the quadrilateral involving $\delta^*c_!v$ and $\delta^*c_!w$ commute by (TC4).
\item Since $\phi$ is a map of biCartesian squares, the functoriality of the (TC3D') diagram gives rise to a dotted map $w\to w'$ making the two adjacent trapezoids commute.
\item By \ref{num:4fct_ext}, the functor $c^!$ commutes with the coend construction, and the functoriality of the (TC3D') diagram together with the exchange isomorphisms~\eqref{Ex^!Hom} give rise to an isomorphism $c^!w''\simeq w'$ making the two adjacent squares commute.
\item Similarly, the functor $\delta^*$ commutes with the coend construction, and the functoriality of the (TC3D') diagram together with the K\"unneth formulas~\eqref{dual_hom} give rise to an isomorphism $u\simeq \delta^*w''$ making the two adjacent squares commute.
\item The map $\delta^*c_!c^!w''\to\delta^*w''$ is simply the conuit of the adjunction, which clearly makes the two squares in the middle commute.
\item By \cite[Lemma 6.9]{GPS}, there exists an isomorphism $u\simeq u'$ making the two adjacent squares commute.
\end{itemize}
\endproof

\begin{lemma}
\label{lem:comp_corr_sq}
Consider the setting in~\ref{num:comp_corr}. For $i\in\{1,2,3\}$, let
\begin{align}
\begin{gathered}
  \xymatrix{
    L_i \ar^-{}[r] \ar_-{}[d] \ar@{}[rd]|{\Gamma_i} & M_i \ar^-{}[d]\\
    \ast \ar_-{}[r] & N_i
  }
\end{gathered}
\end{align}
be a biCartesian square in $\mathcal{T}_c(X_i,\Box)$.  Let $\phi:c_1^{12*}\Gamma_1\to c_2^{12!}\Gamma_2$ and $\psi:c_2^{23*}\Gamma_2\to c_3^{23!}\Gamma_3$ be morphisms of squares in $\mathcal{T}_c(C_{12},\Box)$ and $\mathcal{T}_c(C_{23},\Box)$. Then the composition of correspondences (Definition~\ref{def:comp_corr}) can be lifted to a morphism of squares
\begin{align}
\psi\phi:c_1^{13*}\Gamma_1\to c_3^{13!}\Gamma_3.
\end{align}
\end{lemma}
\proof

This is because we can lift the six functors to the level of diagrams in $\mathcal{T}_c(\cdot,\Box)$, and consequently the same is true for the composition of correspondences.
\endproof

From Proposition~\ref{prop:bil_tr}, Proposition~\ref{prop:add_tr} and Lemma~\ref{lem:comp_corr_sq} we deduce the following
\begin{theorem}
\label{th:add_trace}
Let $\mathcal{T}_c$ be a constructible motivic derivator whose underlying motivic triangulated category $\mathbf{T}$ satisfies~\ref{resol}. We use the notations in~\ref{abstract_map}. For $i\in\{1,2\}$, let
\begin{align}
\begin{gathered}
  \xymatrix{
    L_i \ar^-{}[r] \ar_-{}[d] \ar@{}[rd]|{\Gamma_i} & M_i \ar^-{}[d]\\
    \ast \ar_-{}[r] & N_i
  }
\end{gathered}
\end{align}
be a biCartesian square in $\mathcal{T}_c(X_i,\Box)$. Let $\phi:c_1^*\Gamma_1\to c_2^!\Gamma_2$ and $\psi:d_2^*\Gamma_2\to d_1^!\Gamma_1$ be morphisms of squares in $\mathcal{T}_c(C,\Box)$ and $\mathcal{T}_c(D,\Box)$. 
Then the pairing~\eqref{verdier_pairing} satisfies
\begin{align}
\langle \phi_M ,\psi_M \rangle=\langle \phi_L ,\psi_L \rangle+\langle \phi_N ,\psi_N \rangle,
\end{align}
where $\phi_M:c_1^*M_1\to c_2^!M_2$ is the restriction of $\phi$, and similarly for the other maps.
\end{theorem}

\begin{remark}

One could also formulate the additivity of traces using the language of motivic $(\infty,1)$-categories developped in \cite{Kha}. Given our result for derivators, it suffices to check that the homotopy derivator of a motivic $(\infty,1)$-category of coefficients (\cite[Chapter 2 3.5.2]{Kha}) satisfies the axioms of a constructible motivic derivator. As remarked in \cite{GPS}, it is expected but is yet to be verified that monoidal structures are carried through this construction.

\end{remark}

\section{The characteristic class of a motive}
\label{section:CC}
In this section we come back to $1$-categorical concerns. Let $\mathbf{T}$ be the underlying motivic triangulated category of a constructible motivic derivator which satisfies the condition~\ref{resol} in \ref{par:devissage}. 
\subsection{The characteristic class and first properties}
\subsubsection{}
We recall briefly the formalism we need from \cite{DJK}:
\footnote{The formalism in loc. cit. is constructed for the stable motivic homotopy category $\mathbf{SH}$, but since the construction is quite formal, it can also be done in any motivic triangulated category.}
\begin{recall}
For any scheme $X$, the \textbf{Thom space} construction is a well-defined group homomorphism $Th_X:K_0(X)\to Pic(\mathbf{T}(X))$ (\cite[2.1.4]{DJK}).

Let $f:X\to S$ be a morphism of schemes and let $V$ be a virtual vector bundle over $X$. The \textbf{(twisted) bivariant group} is defined as
\footnote{This is a particular case of the general formalism, see \cite[Definition 2.12]{DJK}.}
\begin{align}
H_0(X/S,V)=Hom_{\mathbf{T}_c(X)}(Th_X(V),f^!\mathbbold{1}_S).
\end{align}
For any proper morphism $p:Y\to X$, there is a proper covariant functoriality
\begin{align}
p_*:H_0(Y/k,p^*V)\to H_0(X/k,V).
\end{align}
If $V$ is a vector bundle over $X$, the \textbf{Euler class} $e(V):\mathbbold{1}_X\to Th_X(V)$ is an analogue of the top Chern class in the classical setting (\cite[Definition 3.1.2]{DJK}). 
When $V$ is the trivial virtual bundle, we use the notation $H_0(X/k)=H_0(X/k,0)$. If $X$ is a smooth scheme, the class $e(L_{X/k}):\mathbbold{1}_X\to Th_X(L_{X/k})\simeq\mathcal{K}_X$ is an element of $H_0(X/k)$.
\footnote{Note that for $\mathbf{T}_c=\mathbf{DM}_{cdh,c}$, the group $H_0(X/k,V)$ is the Chow group of algebraic cycles of dimension the virtual rank of $V$, and we recover the formalism in \cite{Ful}. 
The Euler class in Chow groups is the top Chern class.}
\end{recall}

\begin{definition}
\label{def:cc}
Let $X$ be a scheme and $M\in\mathbf{T}_c(X)$. The Verdier pairing in Definition~\ref{def:verdier_pairing} in the particular case where $C=D=X_1=X_2=X$ and $L_1=L_2=M$ is a pairing
\begin{align}
\langle\ ,\ \rangle:Hom(M,M)\otimes Hom(M,M)\to H_0(X/k).
\end{align}
For any endomorphism $u\in Hom(M,M)$, the \textbf{characteristic class} of $u$ 
is defined as the element $C_X(M,u)\coloneqq\langle u,1_M \rangle\in H_0(X/k)$.
The characteristic class of a motive $M$ is the characteristic class of the identity $C_X(M)\coloneqq C_X(M,1_M)$.
\end{definition}

We now list some elementary properties of the characteristic class.

\subsubsection{}
Since identity maps are particularly good choices of morphisms of distinguished triangles, Theorem~\ref{th:add_trace} implies the additivity of the characteristic class:
\begin{corollary}
\label{additivity}
 Let $X$ be a scheme and let $L\to M\to N\to L[1]$ be a distinguished triangle in $\mathbf{T}_c(X)$. Then $C_X(M)=C_X(L)+C_X(N)$.
\end{corollary}

\begin{remark}
\begin{enumerate}
\item
For every scheme $X$, the additivity of traces yields a well-defined homomorphism of abelian groups
\begin{align}
\label{eq:K0_biv}
\begin{split}
K_0(\mathbf{T}_c(X))&\to H_0(X/k)\\
[M]&\mapsto C_X(M)
\end{split}
\end{align}
where the left hand side is the Grothendieck group of the triangulated category $\mathbf{T}_c(X)$. For $\mathbf{T}_c=\mathbf{DM}_{cdh,c}$, we have $H_0(X/k)\simeq CH_0(X)$ is the Chow group of zero-cycles over $X$ (up to $p$-torsion). It is conjectured that the map~\eqref{eq:K0_biv} is related to the $0$-dimensional part of the closure of the characteristic cycle (\cite[Conjecture 6.8]{Sai}).

\item
One can also relate the characteristic class with the Grothendieck group of varieties over a base.  Composing the map~\eqref{eq:K0_biv} with the obvious ring homomorphism
\begin{align}
\begin{split}
K_0(Var/X)&\to K_0(\mathbf{T}_c(X))\\
[f:Y\to X]&\mapsto [f_!\mathbbold{1}_Y]
\end{split}
\end{align}
where $K_0(Var/X)$ is the Grothendieck group of varieties over $X$, we obtain a well-defined homomorphism of abelian groups
\begin{equation}
\label{eq:K0_CC}
K_0(Var/X)\to H_0(X/k)
\end{equation}
and gives an additive invariant on $K_0(Var/X)$. In the case $X=\operatorname{Spec}(k)$, the group $H_0(X/k)$ is equal to the endomorphism ring $End(\mathbbold{1}_k)$, and the map~\eqref{eq:K0_CC} is a ring homomorphism, which defines a motivic measure on $K_0(Var/k)$. When $\mathbf{T}=\mathbf{SH}$ and $k$ has characteristic $0$, this agrees with the construction in \cite{Ron}.
\end{enumerate}
\end{remark}

\subsubsection{}
It follows from  
Proposition~\ref{proper_pf} that the characteristic class is compatible with the proper functoriality:
\begin{corollary}
\label{cc_proper}
Let $f:X\to Y$ be a proper morphism. Then we have $f_*C_X(M,u)=C_Y(f_*M,f_*u)$.
\end{corollary}

\subsubsection{}
The following lemma shows a relation between the characteristic class and the trace map:
\begin{lemma}
\label{lem:cc_tr}
Let $X$ be a scheme and let $M,N$ be two objects in $\mathbf{T}_c(X)$ such that $M$ is dualizable. Let $u$ be an endomorphism of $M$ and let $v$ be an endomorphism of $N$. Then $C_X(M\otimes N,u\otimes v)=Tr(u)\cdot C_X(N,v)$.

\end{lemma}
\proof

This follows from Proposition~\ref{prop:pair_trace}, using the fact that the canonical map $M^\vee\otimes\mathbb{D}(N)\to\mathbb{D}(M\otimes N)$ is an isomorphism.
\endproof

\subsubsection{}
\label{num:cc_trace_field}
By Corollary~\ref{cc_proper} and Lemma~\ref{lem:cc_tr}, if $f:X\to\operatorname{Spec}(k)$ is a proper morphism, $M\in\mathbf{T}_c(X)$ and $u$ is an endomorphism of $M$,
then the degree of the class $C_X(M,u)$ is the trace of the map $f_*u:f_*M\to f_*M$. In particular, the degree of the class $C_X(M)$ is the \emph{Euler characteristic} of $M$.

\subsubsection{}
We denote by $\langle-1\rangle_k=\chi(\mathbbold{1}_k(1))\in End(\mathbbold{1}_k)$ the \emph{Euler characteristic} of the Tate twist, defined as the trace of the identity map.
\footnote{For $\mathbf{T}_c=\mathbf{SH}_c$ it is well-known that $\langle-1\rangle_k\in End(\mathbbold{1}_k)=GW(k)$ corresponds to the quadratic form $x\mapsto -x^2$ (\cite[Example 1.7]{Hoy}). This element is reduced to identity in $\mathbf{DM}_{cdh,c}$ or $\ell$-adic \'etale cohomology.}
Then for any morphism $f:X\to {\rm Spec}(k)$, 
we have $\chi(\mathbbold{1}_X(n))=\langle-1\rangle_X^n$, 
where $\langle-1\rangle_X=f^*\langle-1\rangle_k\in End(\mathbbold{1}_X)$. 
As a consequence we obtain
\begin{corollary}
\label{cor:cc_tate}
Let $X$ be a scheme and let $M\in\mathbf{T}_c(X)$. Let $u\in End(M)$ be an endomorphism and denote by $u(n)$ be the corresponding endomorphism of $M(n)$. Then $C_X(M(n),u(n))=\langle-1\rangle_X^n\cdot C_X(M,u)$.

\end{corollary}

\subsubsection{}
We can compute characteristic classes for endomorphisms of primitive Chow motives as follows:
\begin{proposition}
\label{prop:cc_endo_euler}
 let $X$ be a smooth $k$-scheme with tangent bundle $T_{X/k}$ and let $p:X\to S$ be a proper morphism. Then for any endomorphism $u$ of $p_*\mathbbold{1}_X$, the characteristic class $C_X(p_*\mathbbold{1}_X,u)$ is given by the composition
\begin{equation}
\label{eq:chow_cc}
\mathbbold{1}_S
\xrightarrow{}
p_*\mathbbold{1}_X
\xrightarrow{u}
p_*\mathbbold{1}_X
\xrightarrow{p_*e(T_{X/k})}
p_*\mathcal{K}_X
\xrightarrow{ad'_{(p_*,p^!)}}
\mathcal{K}_S.
\end{equation}

\end{proposition}

Note that formula~\eqref{eq:chow_cc} is quite similar to the motivic Gauss-Bonnet formula (\cite[Theorem 4.4.1]{DJK}); when $S=\operatorname{Spec}(k)$ and $u=id$ the two formulas are the same.

\proof
We need to show the commutativity of the following diagram:
\begin{align}
\label{diag:euler_diag}
\begin{gathered}
\xymatrix{
\mathbbold{1}_S \ar[r]^-{u} \ar[d]_-{} & \underline{Hom}(p_*\mathbbold{1}_X,p_*\mathbbold{1}_X) \ar[r]^-{} \ar[d]_-{} & \mathbb{D}(p_*\mathbbold{1}_X)\otimes p_*\mathbbold{1}_X \ar[r]^-{\sim} & p_*\mathbbold{1}_X\otimes\mathbb{D}(p_*\mathbbold{1}_X) \ar[d]^-{\epsilon_{p_*\mathbbold{1}_X}} \\
p_*\mathbbold{1}_X \ar[r]^-{u} & p_*\mathbbold{1}_X \ar[r]^-{p_*e(L_{X/k})} & p_*\mathcal{K}_X \ar[r]^-{ad'_{(p_*,p^!)}} \ar[u]^-{} & \mathcal{K}_S.
}
\end{gathered}
\end{align}
While the two squares on the left and on the right are straightforward, it remains to show the commutativity of the square in the middle. Denote by $\delta:X\to X\times_kX$ the diagonal morphism and $p_1,p_2:X\times_kX\to X$ the projections. 
By the self-intersection formula (\cite[Example 3.2.9]{DJK}) as explained in the proof of \cite[Theorem 4.6.1]{DJK}, the Euler class $e(L_{X/k}):\mathbbold{1}_X\xrightarrow{}\mathcal{K}_X$ agrees with the composition
\begin{align}
\mathbbold{1}_X
\xrightarrow{}
\delta^!p_1^*\mathcal{K}_X
\xrightarrow{}
\delta^*p_1^*\mathcal{K}_X
\simeq
\mathcal{K}_X
\end{align}
where the first map is induced by the fundamental class of the morphism $\delta$, and the second map is induced by the natural transformation $\delta^!\to\delta^*$ given by $\delta^!\simeq\delta^*\delta_*\delta^!\xrightarrow{ad'_{(\delta_*,\delta^!)}}\delta^*$. The commutativity then follows from a standard diagram chase.

\endproof

\begin{example}

As a particular case of Proposition~\ref{prop:cc_endo_euler}, for any smooth $k$-scheme $X$ we have $C_X(\mathbbold{1}_X)=e(L_{X/k})$. The class $e(L_{X/k})$ can be understood as the class of the``self-intersection of the diagonal'', 
and by~\ref{num:cc_trace_field} 
we recover the slogan ``The degree of the self-intersection of the diagonal is the Euler characteristic'' for smooth and proper schemes.
\end{example}

\subsubsection{}
\label{num:cc_ref_gysin}
Alternatively, there is a more geometric description of the characteristic class 
$C_X(p_*\mathbbold{1}_X,u)$ using the refined Gysin map: 
denote by $p_1:X\times_SX\to X$ the projection onto the first summand. 
Then base change and purity isomorphisms induce a canonical isomorphism
\begin{equation}
\label{end_BM}
End(p_*\mathbbold{1}_X)\simeq H_0(X\times_SX/k,p_1^{-1}L_{X/k}).
\end{equation}
Also consider the Cartesian diagram
\begin{align}
\begin{gathered}
  \xymatrix{
    X \ar[r]^-{\delta_{X/S}} \ar@{=}[d] \ar@{}[rd]|{\Delta} & X\times_SX \ar[d]^-{}\\
    X \ar[r]_-{\delta_{X/k}} & X\times_kX
  }
\end{gathered}
\end{align}
since $\delta_{X/k}:X\to X\times_kX$ is a regular closed immersion, we have a refined Gysin map (\cite[Definition 4.3.1]{DJK})
\begin{align}
\Delta^!:H_0(X\times_SX/k,p_1^{-1}L_{X/k})\to H_0(X/k).
\end{align}
Then 
for any endomorphism $u$ of $p_*\mathbbold{1}_X$ we have 
\begin{align}
\label{eq:ref_gysin_endo}
C_X(p_*\mathbbold{1}_X,u)=p_*\Delta^!u'
\end{align}
where $u'\in H_0(X\times_SX/k,p_1^{-1}L_{X/k})$ is the image of $u$ by the map~\eqref{end_BM}.

\subsection{A characterization} In this section we give a characterization of the characteristic class for constructible motives.
\subsubsection{}

Let $\mathbf{T}$ be a motivic triangulated category and let $X$ be a scheme. We denote by $\mathbf{T}_{Chow}(X)$ the idempotent completion of the additive subcategory generated by all primitive Chow motives over $X$, and $\langle\mathbf{T}_{Chow}(X)\rangle$ the triangulated subcategory of $\mathbf{T}(X)$ generated by $\mathbf{T}_{Chow}(X)$.

\subsubsection{}
If if the condition~\ref{resol1} in \ref{par:devissage} holds, then $\langle\mathbf{T}_{Chow}(X)\rangle$ contains all \emph{strictly constructible motives} in the sense of \cite[2.2.3]{Ayo} by \cite[Proposition 2.2.27]{Ayo}. We now show that under some cancellation assumptions, it indeed contains all constructible motives:

\begin{definition}
If $\mathcal{C}$ is a triangulated category, we say that a collection of objects $\mathcal{H}$ of $\mathcal{C}$ is \textbf{negative} if for any $A,B\in\mathcal{H}$ any any integer $i>0$ we have
$$
Hom_{\mathcal{C}}(A,B[i])=0.
$$
\end{definition}

\begin{lemma}
\label{lem:chow_neg}
Let $\mathbf{T}$ be a motivic triangulated category which satisfies the condition~\ref{resol} in \ref{par:devissage}. If $X$ is a scheme such that the collection of primitive Chow motives over $X$ is negative, then we have $\langle\mathbf{T}_{Chow}(X)\rangle=\mathbf{T}_c(X)$.

\end{lemma}

\proof

By strong devissage and \cite[Theorem 4.3.2 and Proposition 5.2.2]{Bon}, if $X$ is a scheme such that the collection of primitive Chow motives over $X$ is negative, there exists a unique bounded weight structure on $\mathbf{T}_c(X)$ (called the \emph{Chow weight structure}) whose heart is $\mathbf{T}_{Chow}(X)$, and we conclude using \cite[Corollary 1.5.7]{Bon}.
\endproof

\subsubsection{}
The condition in Lemma~\ref{lem:chow_neg} is satisfied for $\mathbf{T}=\mathbf{DM}_{cdh}$ (\cite{BI}) or the homotopy category of $\mathbf{KGL}$-modules (\cite{BL}), for every scheme $X$. 

\begin{theorem}
\label{th:uniqueness_cc}
Assume that the base field $k$ is perfect. Let $\mathbf{T}$ be the underlying motivic triangulated category of 
a constructible motivic derivator which satisfies the condition~\ref{resol} in \ref{par:devissage}. Let $X$ be a scheme.
\begin{enumerate}
\item
The map
\begin{align}
\begin{split}
\langle\mathbf{T}_{Chow}(X)\rangle&\to H_0(X/k)\\
M&\mapsto C_X(M)
\end{split}
\end{align}
is the unique map satisfying the following properties:
\begin{enumerate}
\item For any distinguished triangle $L\to M\to N\to L[1]$ in $\langle\mathbf{T}_{Chow}(X)\rangle$, $C_X(M)=C_X(L)+C_X(N)$.
\item If $p:Y\to X$ is a proper morphism with $Y$ smooth over $k$ and $M$ is the direct summand of a primitive Chow motive $p_*\mathbbold{1}_Y(n)$ defined by an endomorphism $u$, then $C_X(M)=C_X(p_*\mathbbold{1}_Y(n),u)$ is described by Corollary~\ref{cor:cc_tate} and Proposition~\ref{prop:cc_endo_euler} (or~\ref{num:cc_ref_gysin}). %

\end{enumerate}
\item If the collection of primitive Chow motives over $X$ is negative, then we can replace $\langle\mathbf{T}_{Chow}(X)\rangle$ by $\mathbf{T}_c(X)$ in the statement above.
\end{enumerate}
\end{theorem}

\proof
On the one hand, the characteristic class satisfies these properties by Corollary~\ref{additivity}, Corollary~\ref{cor:cc_tate} and Proposition~\ref{prop:cc_endo_euler}. On the other hand, the second property determine uniquely the characteristic class of all primitive Chow motives, and the uniqueness extends to $\langle\mathbf{T}_{Chow}(X)\rangle$ by additivity of traces. The second part follows from Lemma~\ref{lem:chow_neg}.
\endproof

\begin{remark}
\label{rk:cc_char_perfection}
\begin{enumerate}

\item Alternatively, we can use \cite[5.3.1]{Bon} instead of Lemma~\ref{lem:chow_neg} in the proof.

\item When the field $k$ is not perfect, the following description is suggested to us by D.-C. Cisinski:

Assume that $\mathbf{T}$ is the underlying motivic triangulated category of a constructible motivic derivator which satisfies the condition~\ref{resol2} in \ref{par:devissage}. Assume in addition that $\mathbf{T}$ is extended to noetherian $k$-schemes and is continuous (\cite[Definition 4.3.2]{CD1}). Let $k'$ be the perfect closure of $k$. Then for any scheme $X$, the canonical morphism $\phi_X:X_{k'}=X\times_{k}k'\to X$ is a universal homeomorphism, and by \cite[Theorem 2.1.1]{EK} the functor $\phi_X^*:\mathbf{T}_c(X)\to\mathbf{T}_c(X_{k'})$ is an equivalence of categories. By \cite[Remark 2.1.13]{EK} for any finite surjective radicial morphism $f:Y\to X$ we have a canonical identification $f^*=f^!$, and therefore the Verdier pairing is contravariant for such morphisms (see Remark~\ref{rk:lci_contravariant}). If the collection of primitive Chow motives over $X$ is negative, we conclude that the characteristic class of elements in $\mathbf{T}_c(X)$ is uniquely determined by the functor $\phi_X^*$ and the description in Theorem~\ref{th:uniqueness_cc} for the perfect field $k^s$.

\end{enumerate}
\end{remark}

\subsection{The characteristic class and Riemann-Roch-transformations}
\label{section:cc_rr}
In this section we study the compatibility between the characteristic class and Riemann-Roch-transformations. 

\subsubsection{}
\label{section:rr_eva}
Assume that the pair $(\mathbf{SH},k)$ satisfies condition~\ref{resol} in \ref{par:devissage}. 
Let $\mathbb{E}\in\mathbf{SH}(k)$ be a ring spectrum endowed with a unital associative commutative multiplication. Let $f:S\to k$ be a morphism. Following \cite[7.2.2]{CD1}, the homotopy category of modules over $\mathbb{E}_S=f^*\mathbb{E}$, $Ho(\mathbb{E}_S-Mod)$ is a motivic triangulated category, and the functor
\begin{align}
\label{eq:ass_Emod}
\begin{split}
\mathbf{SH}(S)&\to Ho(\mathbb{E}_S-Mod)\\
M&\mapsto M\otimes\mathbb{E}_S
\end{split}
\end{align}
is a left adjoint of the forgetful functor $Ho(\mathbb{E}_S-Mod)\to\mathbf{SH}(S)$, which preserves constructible objects. The unit map $\phi:\mathbbold{1}_S\to\mathbb{E}_S$ induces the \textbf{$\mathbb{A}^1$-regulator map} (\cite[Definition 4.1.2]{DJK})
\begin{align}
\label{eq:unit_biv}
\phi_*:H_0(S/k)\to \mathbb{E}_0(S/k)
\end{align}
where $\mathbb{E}_0(S/k)=Hom_{\mathbf{SH}(S)}(\mathbbold{1}_S,f^!\mathbb{E})$. 

\subsubsection{}
\label{section:rr_trace}
Let $M\in\mathbf{SH}_c(S)$ be a constructible motivic spectrum, and let $u:M\to M$ be an endomorphism of $M$ in $\mathbf{SH}_c(S)$. Then $u$ induces an endomorphism $u^\mathbb{E}:M\otimes\mathbb{E}_S\to M\otimes\mathbb{E}_S$ in $Ho(\mathbb{E}_S-Mod)_c$. By Definition~\ref{def:cc}, we have the characteristic class $C^{\mathbf{SH}}_S(M,u)\in H_0(S/k)$ in $\mathbf{SH}$, as well as the characteristic class $C^{\mathbb{E}}_S(M\otimes\mathbb{E}_S,u^\mathbb{E})\in \mathbb{E}_0(S/k)$ in $Ho(\mathbb{E}_S-Mod)$.
\begin{proposition}
Via the map~\eqref{eq:unit_biv}, the two classes above satisfy the identity
\begin{align}
\phi_*C^{\mathbf{SH}}_S(M,u)=C^{\mathbb{E}}_S(M\otimes\mathbb{E}_S,u^\mathbb{E})
\end{align}
in $\mathbb{E}_0(S/k)$.
\end{proposition}
\proof
We denote by $\underline{Hom}_{\mathbb{E}}$ the internal $Hom$ functor in $Ho(\mathbb{E}_S-Mod)$ and $\underline{Hom}$ the internal $Hom$ functor in $\mathbf{SH}$. We have a canonical identification $\underline{Hom}_{\mathbb{E}}(A\otimes\mathbb{E}_S,B)\simeq\underline{Hom}(A,B)$. The result then follows from the following commutative diagram:
$$
\resizebox{\textwidth}{!}{
  \xymatrix{
    \mathbbold{1}_S \ar[r]^-{u} \ar[d]_-{\phi} & \underline{Hom}(M,M) \ar[d]_-{\phi} \ar[r]^-{} 
    & M\otimes\mathbb{D}(M) \ar[r]^-{\epsilon_M} \ar[d]_-{\phi} & \mathcal{K}_S \ar[d]_-{\phi} \\
    \mathbb{E}_S \ar[rd]_-{u^\mathbb{E}} & \underline{Hom}(M,M\otimes\mathbb{E}_S) \ar[r]^-{} \ar[d]_-{\wr} 
    & M\otimes\mathbb{D}(M)\otimes\mathbb{E}_S \ar[r]^-{\epsilon_M} \ar[d]_-{p^!\phi} & \mathcal{K}_S\otimes\mathbb{E}_S \ar[d]_-{p^!\phi} \\
    & \underline{Hom}_{\mathbb{E}}(M\otimes\mathbb{E}_S,M\otimes\mathbb{E}_S) \ar[r]^-{} \ar[rd]_-{} 
    & M\otimes\underline{Hom}(M,p^!\mathbb{E})\otimes\mathbb{E}_S \ar[r]^-{\epsilon_M} \ar[d]_-{\wr} & p^!\mathbb{E}\otimes\mathbb{E}_S \ar[d]_-{\eqref{eq:nat_upper*!_upper!}} \\
    & 
    & M\otimes\mathbb{E}_S\otimes\underline{Hom}_{\mathbb{E}}(M\otimes\mathbb{E}_S,p^!\mathbb{E}) \ar[r]^-{\epsilon_{M\otimes\mathbb{E}_S}} & p^!\mathbb{E}
  }
}
$$
\endproof

\begin{corollary}
\label{cor:rr_trace}
Let $\mathbb{E},\mathbb{F}\in\mathbf{SH}(k)$ be two ring spectra endowed with unital associative commutative multiplication. Let $\phi:\mathbb{E}\to\mathbb{F}$ be a morphism of ring spectra. With the notations in~\ref{section:rr_trace}, we have
\begin{align}
\phi_*C^{\mathbb{E}}_S(M\otimes\mathbb{E}_S,u^\mathbb{E})
=
C^{\mathbb{F}}_S(M\otimes\mathbb{F}_S,u^\mathbb{F})
\end{align}
where $\phi_*:\mathbb{E}_0(S/k)\to\mathbb{F}_0(S/k)$ is the map induced by $\phi$.
\end{corollary}

\begin{example}

Let $\mathbb{E}=\mathbf{KGL}$ be the algebraic $K$-theory spectrum, $\mathbb{F}=\oplus_{i\in\mathbb{Z}}\mathbf{H\mathbb{Q}}(i)[2i]$ be the periodized rational motivic Eilenberg-Mac Lane spectrum, and $\phi=ch:\mathbf{KGL}\to\oplus_{i\in\mathbb{Z}}\mathbf{H\mathbb{Q}}(i)[2i]$ be the Chern character. Then for any scheme $X$, the map $ch_*=\tau_X:G_0(X)\to\oplus_{i\in\mathbb{Z}}CH_i(X)_{\mathbb{Q}}$ is the Riemann-Roch transformation in \cite[Theorem 18.3]{Ful} (\cite[Example 3.3.12]{Deg}).

If $X$ is a smooth $k$-scheme of dimension $d$, then the $\mathbf{KGL}$-valued characteristic class $C_X^{\mathbf{KGL}}(\mathbbold{1}_X)\in G_0(X)=K_0(X)$ can be written as
\begin{align}
C_X^{\mathbf{KGL}}(\mathbbold{1}_X)
=
\sum_{i=0}^{d}(-1)^i[\Lambda^iL_{X/k}^{\vee}].
\end{align}
On the other hand we have the $\mathbf{H\mathbb{Q}}$-valued characteristic class $C_X^{\mathbf{H\mathbb{Q}}}(\mathbbold{1}_X)\in CH_0(X)$ given by the top Chern class $c_d(L_{X/k})$. By \cite[Example 3.2.5]{Ful} we have
\begin{align}
\begin{split}
ch(C_X^{\mathbf{KGL}}(\mathbbold{1}_X))
&=
\sum_{i=0}^{d}(-1)^ich(\Lambda^iL_{X/k}^{\vee})
=
c_d(L_{X/k})\cdot Td(L_{X/k})^{-1}\\
&=
C_X^{\mathbf{H\mathbb{Q}}}(\mathbbold{1}_X)\cdot Td(L_{X/k})^{-1}.
\end{split}
\end{align}
In other words we have $ch(C_X^{\mathbf{KGL}}(\mathbbold{1}_X))\cdot Td(L_{X/k})=C_X^{\mathbf{H\mathbb{Q}}}(\mathbbold{1}_X)$, where the left hand side is nothing but $\tau_X(C_X^{\mathbf{KGL}}(\mathbbold{1}_X))$, and we recover a particular case of Corollary~\ref{cor:rr_trace}.

\end{example}

 \section{K\"unneth formulas over general bases}\label{sec:KFOGB}
 In this section, we study transversality conditions following the method of \cite{YZ18}, 
and generalize the K\"unneth formulas \eqref{eq:th_upper!} and \eqref{eq:th_hom} to a general base scheme $S$ under these conditions, which allows us to define the relative characteristic class. We consider $\mathbf{T}$ a motivic triangulated category 
which satisfies the condition~\ref{resol} in \ref{par:devissage}.
  
 \subsection{The transversality conditions}\label{subsec:tc}
 
In this Section~\ref{subsec:tc}, we introduce the transversality conditions and prove some elementary properties, which will be used in the formulation of K\"unneth formulas over a general base scheme in Section~\ref{subsec:rkf}. 
\subsubsection{}
Let $f\colon X\to S$ be a morphism of schemes. 
For two objects $A$ and $B$ of $\mathbf{T}(S)$, there is a canonical natural transformation
\begin{equation}
\label{eq:nat_upper*!_upper!}
f^*B\otimes f^!A\to f^!(B\otimes A)
\end{equation}
given by the composition
\begin{align}
f^*B\otimes f^!A
\xrightarrow{ad_{(f_!,f^!)}}
f^!f_!(f^*B\otimes f^!A)
\overset{\eqref{Ex_!tens}}{\simeq}
f^!(B\otimes f_!f^!A)
\xrightarrow{ad'_{(f_!,f^!)}}
f^!(B\otimes A).
\end{align}
In particular when $A=\mathbbold{1}_S$, the map~\eqref{eq:nat_upper*!_upper!} becomes
\begin{equation}
\label{eq:nat_upper*!_upper!_1}
f^*B\otimes f^!\mathbbold{1}_S\to f^!B.
\end{equation}

\subsubsection{}
\label{ex:example_transversal}
If $f\colon X\to S$ is a smooth morphism, or if $B\in\mathbf{T}(S)$ is dualizable, then the map~\eqref{eq:nat_upper*!_upper!} is an isomorphism: the first case follows from purity, and the second case is \cite[5.4]{FHM}.

\begin{definition}\label{def:Ftransversal}
Let $f\colon X\to S$ be a morphism of schemes. 
\begin{enumerate}
\item (\cite[Definition 8.5]{Sai}) Let $B$ be an object of $\mathbf{T}(S)$. We say that the morphism $f\colon X\to S$ is {\bf $B$-transversal} if the map~\eqref{eq:nat_upper*!_upper!_1} is an isomorphism.
\item \label{def:Ftransversal2}
Let $C$ be an object of $\mathbf{T}(X)$.
We say that the morphism $f\colon X\to S$ is {\bf $C$-transversal} if the graph morphism
$\Gamma_f\colon X\to X\times_k S$ of $f$ is $C\boxtimes_k D$-transversal 
for any object $D\in \mathbf{T}_c(S)$.
We say that $f$ is {\bf universally $C$-transversal} if this property holds after any base change (cf. Definition~\ref{def:loc_acyclic}).
\end{enumerate}
\end{definition}

\begin{remark}
\begin{enumerate}
\item It is easy to see that in Definition~\ref{def:Ftransversal}~\eqref{def:Ftransversal2}, $f$ is $C$-transversal if and only if $\Gamma_f$ is $C\boxtimes_k D$-transversal for any object $D\in \mathbf{T}(S)$.
\item
Let $\mathcal{T}_c$ be a constructible motivic derivator and let $f:X\to S$ be a morphism of schemes. Then there is a stable derivator $Fun^{ex}(\mathcal{T}_c(S),\mathcal{T}_c(X))$ given by triangulated functors from $\mathcal{T}_c(S)$ to $\mathcal{T}_c(X)$. We lift the natural transformation $f^*B\otimes f^!\mathbbold{1}_S\xrightarrow{\eqref{eq:nat_upper*!_upper!_1}} f^!B$ to a coherent morphism in $Fun^{ex}(\mathcal{T}_c(S),\mathcal{T}_c(X))(\underline{\mathbf{1}})$; the target of its cofiber (Definition~\ref{def:stable_der}) is a functor $f^\Delta$ in $Fun^{ex}(\mathcal{T}_c(S),\mathcal{T}_c(X))(\underline{\mathbf{0}})$, seen as a functor $\mathcal{T}_c(S)\to\mathcal{T}_c(X)$. In the underlying triangulated category $\mathcal{T}_c(X,\underline{\mathbf{0}})$ we have a canonical distinguished triangle
$$
f^*B\otimes f^!\mathbbold{1}_S
\to
f^!B
\to
f^\Delta B
\to
f^*B\otimes f^!\mathbbold{1}_S[1].
$$
By definition, $f$ is $B$-transversal if and only if $f^\Delta B=0$.
\end{enumerate}
\end{remark}

\subsubsection{}
\label{recall_pur_trans}
Recall that following \cite[4.3.7]{DJK}, for any lci morphism $f\colon X\to S$ with virtual tangent bundle $T_f$ and any object $B\in\mathbf{T}(S)$, there is a natural transformation called the \textbf{purity transformation}
\begin{equation}
\label{eq:pur_trans}
f^*B\otimes Th_X(L_f)\to f^!B,
\end{equation}
which is deduced from the transformation~\eqref{eq:nat_upper*!_upper!}, and the object $B$ is said \textbf{$f$-pure} if the map~\eqref{eq:pur_trans} is an isomorphism. 

If $f\colon X\to S$ is a lci morphism such that $\mathbbold{1}_Y$ is $f$-pure, then for any object $B\in\mathbf{T}(S)$, $f$ is $B$-transversal if and only if $B$ is $f$-pure.

By \cite[4.3.10]{DJK}, if there exists a scheme $S'$ such that $f$ is an $S'$-morphism between smooth $S'$-schemes, or if $f$ is a lci morphism between regular schemes over a field, then $\mathbbold{1}_S$ is $f$-pure for any motivic triangulated category $\mathbf{T}$. 

\begin{lemma}\label{lem:comptrans}
Let $f\colon X\to Y$ and $g\colon Y\to Z$ be morphisms of schemes such that $g$ is lci and $\mathbbold{1}_Z$ is $g$-pure. 
Let $F\in \mathbf{T}(Z)$ be such that $g$ is $F$-transversal. Then the following two conditions are equivalent:
\begin{enumerate}
\item $f$ is $g^\ast F$-transversal.
\item $g\circ f$ is $F$-transversal.
\end{enumerate}
\end{lemma}
\proof
We show that the first condition implies the second, the converse being similar. Since $g$ is $F$-transversal and $f$ is $g^\ast F$-transversal, the following maps are isomorphisms:
\begin{equation}
\label{eq:lem:comptrans-p1}
f^\ast g^\ast F\otimes f^!\mathbbold{1}_Y
\xrightarrow{\eqref{eq:nat_upper*!_upper!_1}}
f^!g^\ast F,
\end{equation}
\begin{equation}
\label{eq:lem:comptrans-p2}
g^\ast F\otimes g^!\mathbbold{1}_Z
\xrightarrow{\eqref{eq:nat_upper*!_upper!_1}}
g^!F.
\end{equation}
Since $g$ is lci and $\mathbbold{1}_Z$ is $g$-pure, it follows that $g^!\mathbbold{1}_Z$ is dualizable, and by~\ref{ex:example_transversal} the following canonical maps are isomorphisms:
\begin{equation}
\label{eq:lem:comptrans-p3}
f^\ast g^!\mathbbold{1}_Z\otimes f^!\mathbbold{1}_Y
\xrightarrow{\eqref{eq:nat_upper*!_upper!_1}}
f^!g^!\mathbbold{1}_Z,
\end{equation}
\begin{equation}
\label{eq:lem:comptrans-p4}
f^\ast g^!\mathbbold{1}_Z\otimes f^!g^\ast F
\xrightarrow{\eqref{eq:nat_upper*!_upper!}}
f^! (g^!\mathbbold{1}_Z\otimes g^\ast F).
\end{equation}
Therefore we have the following isomorphism
\begin{align}
\label{eq:lem:comptrans-p5}
\begin{split}
f^\ast g^\ast F\otimes f^!g^!\mathbbold{1}_Z
&\overset{{\eqref{eq:lem:comptrans-p3}}}{\simeq}
f^\ast g^\ast F\otimes f^!\mathbbold{1}_Y\otimes f^\ast g^!\mathbbold{1}_Z
\overset{{\eqref{eq:lem:comptrans-p1}}}{\simeq}
f^!g^\ast F\otimes f^\ast g^!\mathbbold{1}_Z\\
&\overset{{\eqref{eq:lem:comptrans-p4}}}{\simeq}
f^!(g^\ast F\otimes g^!\mathbbold{1}_Z)
\overset{{\eqref{eq:lem:comptrans-p2}}}{\simeq}
f^!g^! F.
\end{split}
\end{align}
It is straightforward to check that the map~\eqref{eq:lem:comptrans-p5} is induced by the map~\eqref{eq:nat_upper*!_upper!_1}, and therefore $g\circ f$ is $F$-transversal.
\endproof

\begin{lemma}[{\cite[Proposition 8.7]{Sai}}]\label{lem:dualtranequi}
Let $f\colon X\to Y$ be a $k$-morphism of schemes. 
Then the following statements hold:
\begin{enumerate}
\item
\label{num:dualtranequi_first}
If $f$ is lci and $\mathbbold{1}_Y$ is $f$-pure, then for any $G\in \mathbf{T}_c(Y)$,
$f$ is $G$-transversal if and only if $f$ is $\mathbb{D}(G)$-transversal.

\item
\label{num:dualtranequi_second}
If $X$ is smooth over $k$ and $f$ factors through an open subscheme $Y_0$ of $Y$ which is smooth over $k$, then for any $F\in \mathbf{T}_c(X)$, 
$f$ is $F$-transversal if and only if $f$ is $\mathbb{D}(F)$-transversal.
\end{enumerate}
\end{lemma}
\begin{proof}

\begin{enumerate}
\item
By~\ref{recall_pur_trans}, we need to show that $G$ is $f$-pure if and only if $\mathbb{D}(G)$ is $f$-pure. By duality, the map
\begin{equation}
f^\ast G\otimes Th_X(L_f)
\xrightarrow{\eqref{eq:pur_trans}}
f^! G
\end{equation}
is an isomorphism if and only if its dual $\mathbb{D}(f^! G)\to\mathbb{D}(f^\ast G\otimes Th_X(L_f))$ is an isomorphism. This is equivalent to say that the canonical map 
\begin{equation}
f^*\mathbb{D}(G)
\xrightarrow{\eqref{eq:pur_trans}}
f^!\mathbb{D}(G)\otimes Th_X(-L_f)
\end{equation}
is an isomorphism, i.e. $\mathbb{D}(G)$ is $f$-pure.
\item
We know that the graph $\Gamma_f\colon X\to X\times_k Y$ is lci and $\mathbbold{1}_{X\times_k Y}$ is $\Gamma_f$-pure. By~\eqref{num:dualtranequi_first} and duality, the following statements are equivalent:
\begin{enumerate}
\item For any $H\in \mathbf{T}_c(Y)$, $\Gamma_f$ is $F\boxtimes_k H$-transversal.
\item For any $H\in \mathbf{T}_c(Y)$, $\Gamma_f$ is $\mathbb D(F\boxtimes_k H)$-transversal.
\item For any $H\in \mathbf{T}_c(Y)$, $\Gamma_f$ is $\mathbb D(F\boxtimes_k \mathbb D(H))$-transversal.
\end{enumerate}
Denote by $p_1:X\times_k Y\to X$ and $p_2:X\times_k Y\to Y$ the projections. Then $p_2$ is a smooth morphism, and we have the following isomorphism:
\begin{align}
\begin{split}
\mathbb D(F\boxtimes_k \mathbb D(H))
&\overset{\eqref{eq:bidual_hom}}{\simeq}
\underline{Hom}(p_1^\ast F,p_2^\ast H)
\simeq
\underline{Hom}(p_1^\ast F,p_2^! H)\otimes Th_{X\times_k Y}(-L_{p_2})\\
&\overset{\eqref{dual_hom}}{\simeq}
(\mathbb D(F)\boxtimes H)\otimes Th_{X\times_k Y}(-L_{p_2}).
\end{split}
\end{align}
It follows that $\Gamma_f$ is $F\boxtimes_k H$-transversal for any $H\in \mathbf{T}_c(Y)$ if and only if $\Gamma_f\colon X\to X\times_k Y$ is $\mathbb D(F)\boxtimes_k H$-transversal for any $H\in \mathbf{T}_c(Y)$. This proves~\eqref{num:dualtranequi_second}.
\end{enumerate}
\end{proof}

\begin{lemma}[{\cite[Lemma 2.3.4]{YZ18}}]\label{lem:diagtrans}
Let $X$ be a smooth $k$-scheme and let $F_1$ and $F_2$ be two objects of $\mathbf{T}_c(X)$.
If the diagonal morphism $\delta\colon X\to X\times_k X$ is 
$\mathbb{D}(F_1)\boxtimes_k F_2$-transversal, then the following canonical map is an isomorphism:
\begin{align}
\underline{Hom}(F_1, \mathbbold{1})\otimes F_2\rightarrow \underline{Hom}(F_1,F_2).
\end{align}

\end{lemma}
\begin{proof}
Since $X$ is smooth over $k$, by purity we have $\mathcal K_X\simeq Th_X(L_{X/k})$ and $\delta^!\mathbbold{1}_{X\times_k X}\simeq Th(-L_{X/k})$.
For $i=1,2$, denote by $p_i\colon X\times_k X\to X$ the $i$-th projection. Since $\delta$ is $\mathbb{D}(F_1)\boxtimes_k F_2$-transversal, we have the following isomorphism:
\begin{align}
\label{eq:diagtrans}
\begin{split}
\underline{Hom}(F_1, \mathbbold{1})\otimes F_2
&\simeq
\underline{Hom}(F_1, \mathcal K_X)\otimes F_2\otimes \delta^!\mathbbold{1}_{X\times_k X}\\
&=
\delta^\ast(\underline{Hom}(F_1, \mathcal K_X)\boxtimes_k F_2)\otimes \delta^!\mathbbold{1}_{X\times_k X}\\
&
\overset{\eqref{eq:nat_upper*!_upper!_1}}{\simeq}
\delta^! (\underline{Hom}(F_1, \mathcal K_X)\boxtimes_k F_2)
\overset{\eqref{dual_hom}}{\simeq}
\delta^! \underline{Hom}(p_1^\ast F_1, p_2^! F_2)\\
&\simeq
\underline{Hom}(\delta^\ast p_1^\ast F_1, \delta^! p_2^! F_2)
=
\underline{Hom}(F_1,F_2).
\end{split}
\end{align}
One can check that the map~\eqref{eq:diagtrans} agrees with the canonical map, and the result follows.
\end{proof}

\begin{lemma}
\label{lem:Hom_upper*}
Let $f:X\to Y$ be a lci morphism and let $F$ and $G$ be two objects of $\mathbf{T}(Y)$. If both $G$ and $\underline{Hom}(F,G)$ are $f$-pure, then the following canonical map is an isomorphism:
\begin{align}
f^*\underline{Hom}(F,G)
\to
\underline{Hom}(f^*F,f^*G).
\end{align}
\end{lemma}

\proof
By hypothesis, we have the following isomorphism:
\begin{align}
\label{eq:Hom_upper*}
\begin{split}
&f^*\underline{Hom}(F,G)\otimes Th_X(L_f)
\overset{\eqref{eq:pur_trans}}{\simeq}
f^!\underline{Hom}(F,G)
\overset{\eqref{Ex^!Hom}}{\simeq}
\underline{Hom}(f^*F,f^!G)\\
\overset{\eqref{eq:pur_trans}}{\simeq}
&\underline{Hom}(f^*F,f^*G\otimes Th_X(L_f))
\simeq
\underline{Hom}(f^*F,f^*G)\otimes Th_X(L_f).
\end{split}
\end{align}
It is straightforward to check that~\eqref{eq:Hom_upper*} is induced by the canonical map, and the result follows.
\endproof

\begin{corollary}
\label{cor:Hom_upper*_1}
Let $f:X\to Y$ be a morphism between smooth $k$-schemes. Let $F$ be an object of $\mathbf{T}_c(Y)$ such that $f$ is $F$-transversal. Then the following canonical map is an isomorphism:
\begin{align}
\label{eq:Hom_upper*_1}
f^*\underline{Hom}(F,\mathbbold{1}_Y)
\to
\underline{Hom}(f^*F,\mathbbold{1}_X).
\end{align}
\end{corollary}

\proof
By Lemma~\ref{lem:dualtranequi}, $f$ is $\mathbb{D}(F)$-transversal. Since $\mathcal{K}_Y\simeq Th_Y(L_{Y/k})$ is $\otimes$-invertible, we know that $f$ is also $\underline{Hom}(F,\mathbbold{1}_Y)$-transversal. By hypothesis and~\ref{recall_pur_trans}, both $F$ and $\underline{Hom}(F,\mathbbold{1}_Y)$ are $f$-pure. We conclude by applying Lemma~\ref{lem:Hom_upper*}.
\endproof

\begin{lemma}\label{lem:smoothtranComp}
Let $p:X\to Y$ and $g:Y\to S$ be two morphisms of schemes and let $f=g\circ p:X\to S$ be their composition. Denote by $\Gamma_p$, $\Gamma_g$ and $\Gamma_f$ the graph morphisms of $p$, $g$ and $f$ respectively. Let $F$ be an object of $\mathbf{T}(X)$. Then the following statements hold:
\begin{enumerate}
\item If $g$ is smooth and $p$ is $F$-transversal, then $f$ is $F$-transversal.
\item \label{lem:propertranComp}
Assume that $p$ is proper, $f$ is $F$-transversal and the canonical map
\begin{align}
\label{eq:propertranComp}
p^\ast\Gamma_g^!\mathbbold{1}_{Y\times_kS}
\xrightarrow{\eqref{Ex*!}}
\Gamma_f^!\mathbbold{1}_{X\times_kS}
\end{align}
associated to the Cartesian square of schemes
\begin{align}
\begin{gathered}
\xymatrix{
X\ar[r]^-{p}\ar[d]_-{\Gamma_f}&Y\ar[d]^-{\Gamma_g}\\
X\times_k S\ar[r]^-{p\times id_S}&Y\times_kS.
}
\end{gathered}
\end{align}
is an isomorphism. Then $g$ is $p_\ast(F)$-transversal.
\end{enumerate}
\end{lemma}
\proof
\begin{enumerate}
\item
We have a  commutative diagram of schemes
\begin{align}
\begin{gathered}
\xymatrix{
X\ar[r]^-{\Gamma_p}\ar[rd]_-{\Gamma_f}&X\times_k Y\ar[d]^-{id_X\times_kg}\\
&X\times_kS.
}
\end{gathered}
\end{align}
Let $H$ be an object of $\mathbf{T}_c(S)$. Since $g$ is smooth,
we have canonical isomorphisms
\begin{align}
\label{eq:prop:smoothtranComp-p1}
\begin{split}
\Gamma_f^\ast(F\boxtimes_k H)\otimes\Gamma_f^!\mathbbold{1}_{X\times_kS}
&=
\Gamma_p^\ast(id_X\times_kg)^\ast(F\boxtimes_k H)\otimes\Gamma_p^!(id_X\times_kg)^!\mathbbold{1}_{X\times_kS}\\
&\simeq
\Gamma_p^\ast(F\boxtimes_k g^!H)\otimes\Gamma_p^!\mathbbold{1}_{X\times_kY},
\end{split}
\end{align}
\begin{align}
\label{eq:prop:smoothtranComp-p2}
\Gamma_f^!(F\boxtimes_k H)
=
\Gamma_p^!(id_X\times_kg)^!(F\boxtimes_k H)
\simeq
\Gamma_p^!(F\boxtimes_kg^!H).
\end{align}
Since $\Gamma_p$ is $F\boxtimes_kg^!H$-transversal, the following canonical map is an isomorphism:
\begin{align}
\label{eq:prop:smoothtranComp-p3}
\Gamma_p^\ast(F\boxtimes_kg^!H)\otimes\Gamma_p^!\mathbbold{1}_{X\times_kY}
\xrightarrow{\eqref{eq:nat_upper*!_upper!_1}}
\Gamma_p^!(F\boxtimes_kg^!H).
\end{align}
By \eqref{eq:prop:smoothtranComp-p1}, \eqref{eq:prop:smoothtranComp-p2} and \eqref{eq:prop:smoothtranComp-p3}, the following canonical map is an isomorphism:
\begin{align}
\Gamma_f^\ast(F\boxtimes_k H)\otimes\Gamma_f^!\mathbbold{1}_{X\times_kS}
\xrightarrow{\eqref{eq:nat_upper*!_upper!_1}}
\Gamma_f^!(F\boxtimes_k H).
\end{align}
In other words $\Gamma_f$ is $F\boxtimes_k H$-transversal. Since this is true for any $H$, by definition $f$ is $F$-transversal.

\item
Let $H$ be an object of $\mathbf{T}_c(S)$.  
Then $\Gamma_f$ is $F\boxtimes_k H$-transversal, and the following canonical map is an isomorphism:
\begin{align}\label{eq:prop:propertranComp-p1}
\Gamma_f^\ast(F\boxtimes_k H)\otimes \Gamma_f^!\mathbbold{1}_{X\times_kS}
\xrightarrow{\eqref{eq:nat_upper*!_upper!_1}}
\Gamma_f^!(F\boxtimes_k H).
\end{align}
Since $p$ is proper, we have canonical isomorphisms:
\begin{align}\label{eq:prop:propertranComp-p2}
\begin{split}
p_\ast(\Gamma_f^\ast(F\boxtimes_k H)\otimes \Gamma_f^!\mathbbold{1}_{X\times_kS})
&\overset{\eqref{eq:propertranComp}}{\simeq}
p_\ast(\Gamma_f^\ast(F\boxtimes_k H)\otimes p^\ast\Gamma_g^!\mathbbold{1}_{Y\times_kS})\\
&\overset{\eqref{Ex_*tens}}{\simeq}
p_\ast\Gamma_f^\ast(F\boxtimes_k H)\otimes \Gamma_g^!\mathbbold{1}_{Y\times_kS}\\
&\overset{\eqref{Ex**}}{\simeq}
\Gamma_g^\ast (p\times id_S)_\ast (F\boxtimes_k H)\otimes \Gamma_g^!\mathbbold{1}_{Y\times_kS}\\
&\overset{\eqref{Kunneth_*}}{\simeq}
\Gamma_g^\ast(p_\ast F\boxtimes_k H)\otimes \Gamma_g^!\mathbbold{1}_{Y\times_kS},
\end{split}
\end{align}
\begin{align}
\label{eq:prop:propertranComp-p3}
p_\ast(\Gamma_f^!(F\boxtimes_k H))
&\overset{\eqref{upper!lower*}}{\simeq}
\Gamma_g^! (p\times id)_\ast (F\boxtimes_k H)
\overset{\eqref{Kunneth_*}}{\simeq}
\Gamma_g^!(p_\ast F\boxtimes_k H).
\end{align}
By \eqref{eq:prop:propertranComp-p1}, \eqref{eq:prop:propertranComp-p2}
and \eqref{eq:prop:propertranComp-p3}, the following canonical map is an isomorphism:
\begin{align}
\Gamma_g^\ast(p_\ast F\boxtimes_k H)\otimes \Gamma_g^!\mathbbold{1}_{Y\times_kS}
\xrightarrow{\eqref{eq:nat_upper*!_upper!_1}}
\Gamma_g^!(p_\ast F\boxtimes_k H).
\end{align}
In other words $\Gamma_g$ is $p_\ast F\boxtimes_k H$-transversal. Since this is true for any $H$, by definition $g$ is $p_\ast F$-transversal.
\end{enumerate}
\endproof

\begin{remark}
Note that the map~\eqref{eq:propertranComp} is an isomorphism if $p$ is smooth, or if $g$ factors through an open subscheme of $S$ which is smooth over $k$.
\end{remark}

\begin{proposition}
\label{prop:kuntypeTrans}
Let $f_1:X_1\to Y_1$ and $f_2:X_2\to Y_2$ be two morphisms of schemes, and let $f:X_1\times_k X_2\to Y_1\times_k Y_2$ be their product.
For $i=1,2$,  let  $G_i\in\mathbf{T}(Y_i)$. 
Then we have the following K\"unneth type result for the transversality condition:
 if $f_i$ is $G_i$-transversal for $i=1,2$, then $f$ is $G_1\boxtimes G_2$-transversal.
\end{proposition}
\proof
By assumption and Proposition~\ref{Kunneth_thm}, we have the following isomorphism:
\begin{align}
\label{eq:kuntypeTrans}
\begin{split}
f^\ast(G_1\boxtimes_k G_2)\otimes f^!\mathbbold 1_{Y_1\times_k Y_2}
&\overset{{\eqref{Kunneth}}}{\simeq}
(f_1^\ast(G_1)\boxtimes_k f_2^\ast(G_2))\otimes (f_1^!\mathbbold 1_{Y_1}\boxtimes_k f_2^!\mathbbold 1_{Y_2})\\
&\simeq
(f_1^\ast G_1\otimes f_1^!\mathbbold 1_{Y_1})\boxtimes_k (f_2^\ast G_2\otimes f_2^!\mathbbold 1_{Y_2})\\
&\overset{{\eqref{eq:nat_upper*!_upper!_1}}}{\simeq} 
f_1^!G_1\boxtimes_S f_2^! G_2
\overset{\eqref{Kunneth}}{\simeq}
f^!(G_1\boxtimes_k G_2).
\end{split}
\end{align}
It is straightforward to check that the map~\eqref{eq:kuntypeTrans} agrees with~\eqref{eq:nat_upper*!_upper!_1}, and therefore $f$ is $G_1\boxtimes_k G_2$-transversal.
\endproof
\begin{conjecture}\label{conj:kuntypeTrans}
Let $f_1:X_1\to Y_1$ and $f_2:X_2\to Y_2$ be two morphisms of schemes, and let $f:X_1\times_k X_2\to Y_1\times_k Y_2$ be their product.
For $i=1,2$,  let $F_i\in\mathbf{T}(X_i)$. 
 If $f_i$ is $F_i$-transversal for $i=1,2$, then $f$ is $F_1\boxtimes F_2$-transversal.
\end{conjecture}

\begin{proposition}\label{prop:kuntypeTransconj}
Conjecture \ref{conj:kuntypeTrans} is true if $Y_2=\operatorname{Spec}(k)$: let $f_1\colon X_1\to Y_1$ be a morphism of schemes.
Let $F_1\in{\mathbf T}(X_1)$ such that $f_1$ is $F_1$-transversal. 
Then for any $k$-scheme $X_2$ and any object $F_2\in \mathbf T(X_2)$,
the composition morphism $f\colon X_1\times_k X_2\xrightarrow{p_1^{12}}X_1\xrightarrow{f_1}Y_1$ is
$F_1\boxtimes_k F_2$-transversal.
\end{proposition}

\proof
Denote by $p_1$, $p_2$, $p_3$ the projections of $X_1\times_k Y_1\times_k X_2$ to its components.
By definition, we need to show that for any $G\in\mathbf T(Y_1)$, the graph of $f$
\begin{align}
\Gamma_f
=
\Gamma_{f_1}\times_k id_{X_2}
\colon 
X_1\times_k X_2
\to
X_1\times_k Y_1\times_k X_2
\end{align}
is $p_1^\ast F_1\otimes p_2^\ast G \otimes p_3^\ast F_2$-transversal, 
namely the following canonical map is an isomorphism:
\begin{align}\label{eq:prop:kuntypeTransconj-p1}
\Gamma_f^\ast(p_1^\ast F_1\otimes p_2^\ast G\otimes p_3^\ast F_2)
\otimes \Gamma_f^!\mathbbold 1_{X_1\times_k Y_1\times_k X_2}
\xrightarrow{\eqref{eq:nat_upper*!_upper!_1}}
\Gamma_f^!(p_1^\ast F_1\otimes p_2^\ast G\otimes p_3^\ast F_2).
\end{align}

Since $\Gamma_{f_1}$ is $F_1\boxtimes_k G$-transversal, the following canonical map is an isomorphism:
\begin{align}\label{eq:prop:kuntypeTransconj-p2}
\Gamma_{f_1}^\ast(F_1\boxtimes_k G)\otimes\Gamma_{f_1}^!\mathbbold 1_{X_1\times_k Y_1}
\xrightarrow{\eqref{eq:nat_upper*!_upper!_1}}
\Gamma_{f_1}^!(F_1\boxtimes_k G).
\end{align}
By Proposition~\ref{Kunneth_thm}, we have the following isomorphism:
\begin{align}
\label{eq:prop:kuntypeTransconj-p3}
\begin{split}
&\Gamma_f^\ast(p_1^\ast F_1\otimes p_2^\ast G\otimes p_3^\ast F_2)
\otimes \Gamma_f^!\mathbbold 1_{X_1\times_k Y_1\times_k X_2}\\
\overset{\eqref{Ex*!}}{\simeq}
&(\Gamma_{f_1}^\ast(F_1\boxtimes_k G)\boxtimes_k  F_2)
\otimes p_1^{12*}\Gamma_{f_1}^!\mathbbold 1_{X_1\times_k Y_1}\\
\overset{\eqref{eq:prop:kuntypeTransconj-p2}}{\simeq}
&\Gamma_{f_1}^!(F_1\boxtimes_k G)\boxtimes_k F_2
\overset{\eqref{Kunneth}}{\simeq}
\Gamma_f^!(p_1^\ast F_1\otimes p_2^\ast G\otimes p_3^\ast F_2).
\end{split}
\end{align}
It is straightforward to check that the map~\eqref{eq:prop:kuntypeTransconj-p3} agrees with~\eqref{eq:prop:kuntypeTransconj-p1}, and therefore the latter is an isomorphism, which finishes the proof.
\endproof

\subsection{Relative K\"unneth formulas and the relative Verdier pairing}\label{subsec:rkf}
In this section we use the results in Section~\ref{subsec:tc} to extend the K\"unneth formulas to the relative setting, under some transversality assumptions. Using such results we define the relative Verdier pairing as in \cite{YZ18}.
\subsubsection{}\label{transSec4:notation} Let $S$ be a scheme, and 
let $\pi_1\colon X_1\to S$ and $ \pi_2\colon X_2\to S$ be two morphisms. 
For $i=1,2$, we denote by $p_i:X_1\times_SX_2\to X_i $ the projections.
\begin{lemma}\label{lem:boxtimesTrans} 
We use the notations of \ref{transSec4:notation}, and assume that the morphism $\pi_2$ is smooth. Let $F_1$ be an object of $ \mathbf{T}(X_1)$ such that the morphism $\pi_1\colon X_1\to S$ is universally $F_1$-transversal.
Then the canonical closed 
immersion  $\iota\colon X_1\times_SX_2 \to X_1\times_kX_2$ is 
 $F_1\boxtimes_k F_2$-transversal for any $F_2\in \mathbf{T}(X_2)$
\end{lemma}
\begin{proof}
Denote by $p_{13}:X_1\times_S X_2\times_k X_2\to X_1\times_k X_2$ the projection, and the graph of $p_2$
\begin{align}
\Gamma_{p_2}\colon X_1\times_SX_2\to X_1\times_S X_2\times_k X_2,
\end{align}
with the commutative diagram 
\begin{align}
\begin{gathered}
\xymatrix{
X_1\times_S X_2\ar[r]^-{\iota}\ar[rd]_<(.2){\Gamma_{p_2}}&X_1\times_k X_2\\
&X_1\times_S X_2\times_k X_2\ar[u]_-{p_{13}}.
}
\end{gathered}
\end{align}
Let $F_2$ be an object of $\mathbf{T}(X_2)$. Since $\pi_1\colon X_1\to S$ is universally $F_1$-transversal, the morphism $p_2\colon X_1\times_SX_2\to X_2$ is also universally $p_1^\ast F_1$-transversal.
Consequently $\Gamma_{p_2}$ is $p_1^\ast F_1\boxtimes_k F_2=p_{13}^\ast(F_1\boxtimes_k F_2)$-transversal. 
By assumption $p_{13}$ is smooth, and by~\ref{ex:example_transversal}, $p_{13}$ is universally $F_1\boxtimes_k F_2$-transversal.
By Lemma \ref{lem:comptrans}, the composition $\iota=p_{13}\circ\Gamma$ is  
$F_1\boxtimes_k F_2$-transversal, which finishes the proof.
\end{proof}

\begin{proposition}[{\cite[Proposition 3.1.3]{YZ18}}]\label{thm:relKuHom}
We use the notations of \ref{transSec4:notation}, and let $E_i$ and $F_i$ be 
objects of $ \mathbf{T}_c(X_i)$ for $i=1,2$. 
Assume that the following conditions are satisfied:
\begin{enumerate}
\item The morphisms $\pi_1$ and $\pi_2$ are smooth, and both $X_1$ and $X_2$ are smooth $k$-schemes.
\item \label{num:relKuHom_diag}
For $i=1,2$, the diagonal morphism $X_i\to X_i\times_k X_i$ is 
$\mathbb D(E_i)\boxtimes_k F_i$-transversal. 
\item For $i=1,2$, $\pi_i$ is universally $E_i$-transversal and universally $F_i$-transversal.
\end{enumerate}
Then the following canonical map is an isomorphism:
\begin{align}
\label{eq:relKuHom}
\underline{Hom}(E_1,F_1)\boxtimes_S \underline{Hom}(E_2,F_2)
\xrightarrow{\eqref{Kunneth_part}} 
\underline{Hom} (E_1\boxtimes_S E_2, F_1\boxtimes_S F_2).
\end{align}

\end{proposition}
\proof
We use the following notation: if $X$ is a scheme and $F\in\mathbf{T}(X)$, we denote $F^{\vee}\coloneq  \underline{Hom}(F, \mathbbold 1_{X})$.

By assumption the diagonal morphism 
$X_i\to X_i\times_k X_i$ is 
$\mathbb D(E_i)\boxtimes_k F_i$-transversal,
and by Lemma \ref{lem:diagtrans} the following canonical map is an isomorphism:
\begin{equation}
\label{eq:FE^vee}
 F_i\otimes  E_i^{\vee}=  F_i\otimes\underline{Hom}( E_i, \mathbbold 1_{X_i}) 
 \xrightarrow{\sim}  
 \underline{Hom}(E_i,  F_i).
\end{equation}
Hence we have the following isomorphism:
\begin{align}
\label{isom.lhs}
\begin{split}
 \underline{Hom}( E_1, F_1)\boxtimes_S \underline{Hom}( E_2, F_2)
&\overset{\eqref{eq:FE^vee}}{\simeq}
( F_1\otimes E_1^{\vee})\boxtimes_S( F_2 \otimes E_2^{\vee})\\
&\simeq 
( F_1\boxtimes_S F_2)\otimes( E_1^{\vee} \boxtimes_S E_2^{\vee}).
\end{split}
\end{align}

Denote by $\iota$ the canonical closed immersion $\iota\colon X_1\times_SX_2\to X_1\times_kX_2$. By Proposition~\ref{hom_product}, Corollary~\ref{cor:Hom_upper*_1} and Lemma \ref{lem:boxtimesTrans}, we have the following isomorphism:
\begin{align}
\label{isom.dual} 
\begin{split}
  E_1^{\vee} \boxtimes_S  E_2^{\vee}
=
\iota^*(E_1^{\vee} \boxtimes_k  E_2^{\vee})
\overset{\eqref{Kunneth_part}}{\simeq}
\iota^*(E_1\boxtimes_kE_2)^{\vee}
\overset{\eqref{eq:Hom_upper*_1}}{\simeq}
(\iota^*(E_1\boxtimes_kE_2))^{\vee}
=
(E_1\boxtimes_S E_2)^{\vee}.
 \end{split}
 \end{align}

By assumption~\eqref{num:relKuHom_diag}, Lemma \ref{lem:comptrans}, Lemma \ref{lem:dualtranequi} and Lemma \ref{lem:boxtimesTrans},
the diagonal morphism $X_i\to X_i\times_S X_i$ is 
$\mathbb D(E_i)\boxtimes_S F_i$-transversal for $i=1,2$. 
By Proposition \ref{prop:kuntypeTrans}, the diagonal morphism
$X_1\times_S X_2\to (X_1\times_S X_2)\times_k(X_1\times_S X_2)$ is $\mathbb D(E_1\boxtimes_SE_2)\boxtimes_k(F_1\boxtimes_S F_2)$-transversal.
Thus by Lemma \ref{lem:diagtrans}, the following canonical map is an isomorphism:
\begin{align}\label{eq:isomKunDual}
 (F_1\boxtimes_S F_2)\otimes (E_1\boxtimes_S E_2)^{\vee}
\to
 \underline{Hom}( E_1\boxtimes_S E_2,  F_1\boxtimes_SF_2).
\end{align}
We deduce from \eqref{isom.lhs}, \eqref{isom.dual} and \eqref{eq:isomKunDual} that the map~\eqref{eq:relKuHom} is an isomorphism, which finishes the proof.
\endproof

\begin{corollary}\label{cor:SSKunn}
We use the notations of \ref{transSec4:notation}, and let $F_i$ be 
an object of $ \mathbf{T}_c(X_i)$ for $i=1,2$. 
Assume that the following conditions are satisfied:
\begin{enumerate}
\item The morphisms $\pi_1$ and $\pi_2$ are smooth, and both $X_1$ and $X_2$ are smooth $k$-schemes.

\item For $i=1,2$, $\pi_i\colon X_i\to S$ is universally $F_i$-transversal.
\end{enumerate}
Then the map 
\begin{align}
\label{eq:SSKunn}
 F_1\boxtimes_S \underline{Hom}( F_2, \pi_2^!\mathbbold 1_S) 
 \to 
 \underline{Hom}(p_2^* F_2, p_1^! F_1)
\end{align}
given by the composition
\begin{align}
\begin{split}
 &F_1\boxtimes_S \underline{Hom}( F_2, \pi_2^!\mathbbold 1_S) 
 \to
 \underline{Hom}(p_2^* F_2, p_1^\ast  F_1\otimes p_2^*\pi_2^!\mathbbold 1_S)\\
 \xrightarrow{\eqref{Ex*!}}
 &\underline{Hom}(p_2^* F_2, p_1^\ast  F_1\otimes p_1^!\mathbbold 1_{X_1})
 \xrightarrow{\eqref{eq:nat_upper*!_upper!_1}}
 \underline{Hom}(p_2^* F_2, p_1^! F_1).
\end{split}
\end{align}
is an isomorphism.
\end{corollary}

\begin{proof}
Since $p_1$ is smooth, by Proposition~\ref{thm:relKuHom}, we have the following isomorphism
\begin{align}
\begin{split}
F_1\boxtimes_S \underline{Hom}( F_2, \pi_2^!\mathbbold 1_S) 
\overset{\eqref{eq:relKuHom}}{\simeq}  
\underline{Hom}(p_2^*  F_2, p_1^\ast  F_1\otimes p_2^*\pi_2^!\mathbbold 1_S)
\simeq
\underline{Hom}(p_2^* F_2, p_1^! F_1),
\end{split}
\end{align}
which shows that the map~\eqref{eq:SSKunn} is an isomorphism.
\end{proof}

\begin{proposition}[{\cite[Proposition 3.1.9]{YZ18}}]\label{thm:fUpperIsomKun}
For $i=1,2$, consider a commutative diagram of $S$-morphisms of the form
\begin{align}
\begin{gathered}
\xymatrix{ X_i \ar[rr]^-{f_i}\ar[dr]_-{\pi_i} && Y_i \ar[dl]^-{q_i}\\
&S,&
}
\end{gathered}
\end{align}
We denote $X\coloneq X_1\times_SX_2$, $Y\coloneq Y_1\times_SY_2$ and $f\coloneq f_1\times_Sf_2\colon X \to  Y$. Let $M_i$ be objects of $\mathbf{T}_c(Y_i)$ for $i=1,2$. 
Assume that the following conditions are satisfied:
\begin{enumerate}
\item The morphisms $\pi_i$ and $q_i$ are smooth, and both $X_i$, $Y_i$ are smooth $k$-schemes. 
\item For $i=1,2$, $q_i\colon Y_i\to S$ is universally $M_i$-transversal.
\end{enumerate}
Then the following canonical map is an isomorphism:
\begin{align}
f_1^! M_1\boxtimes_S f_2^!  M_2 
\xrightarrow{\eqref{Kunneth}}
 f^!(  M_1\boxtimes_S  M_2).
\end{align}
\end{proposition}
\begin{proof}
By Lemma~\ref{Ex!*_comp}, we may assume that $X_2=Y_2$ and $f_2=id_{X_2}$, i.e. it suffices to show that the following canonical map is an isomorphism:
\begin{equation}\label{eq:prop:fUpperIsomKun100}
f_1^! M_1\boxtimes_S  M_2 
\xrightarrow{}
(f_1\times id_{X_2})^!(  M_1\boxtimes_S   M_2).
\end{equation}
We assume that $q_2\colon Y_2\to S$ is universally $M_2$-transversal, the other case is very similar.
By assumption and duality, we have the following isomorphism:
\begin{align}
M_2
\underset{\sim}{\xrightarrow{\eqref{eq:biduality}}}
\underline{Hom}(\mathbb{D}(M_2),\mathcal{K}_{Y_2})
\simeq
\underline{Hom}(\mathbb{D}(M_2)\otimes Th(L_{q_2}-T_{Y_2}), q_2^!\mathbbold{1}_S).
\end{align}
In other words $M_2\simeq\underline{Hom}(L_2, q_2^!\mathbbold{1}_S)$ for some $ L_2\in \mathbf T(Y_2)$.
By assumption, $q_2^!\mathbbold{1}_S$ is a dualizing object in $\mathbf{T}(Y_2)$, and by Lemma \ref{lem:dualtranequi}, the morphism $q_2\colon Y_2\to S$ is universally $L_2$-transversal.
By Corollary \ref{cor:SSKunn}, we have the following isomorphisms:
\begin{align}\label{eq:prop:fUpperIsomKun101}
M_1\boxtimes_S \underline{Hom}( L_2, q_2^!\mathbbold{1}_S)
\overset{\eqref{eq:SSKunn}}{\simeq}
\underline{Hom}(p_2^\ast L_2, p_1^! M_1),
\end{align}
\begin{align}
\label{eq:prop:fUpperIsomKun102}
f_1^!  M_1\boxtimes_S \underline{Hom} (  L_2, q_2^!\mathbbold{1}_S)
\overset{\eqref{eq:SSKunn}}{\simeq}
\underline{Hom}((f_1\times id_{X_2})^*p_2^*  L_2, p_1^!f_1^!  M_1). 
\end{align}
We deduce from~\eqref{eq:prop:fUpperIsomKun101} and~\eqref{eq:prop:fUpperIsomKun102} the following isomorphism:
\begin{align}
\label{eq:prop:fUpperIsomKun103}
\begin{split}
&(f_1\times {\rm id})^!( M_1\boxtimes_S  M_2)
=
(f_1\times id_{X_2})^!( M_1\boxtimes_S   \underline{Hom}( L_2, q_2^!\mathbbold{1}_S))\\
\overset{\eqref{eq:prop:fUpperIsomKun101}}{\simeq} 
&(f_1\times id_{X_2})^!\underline{Hom}(p_2^\ast L_2, p_1^! M_1)
\overset{\eqref{Ex^!Hom}}{\simeq}  
\underline{Hom}((f_1\times id_{X_2})^*p_2^* L_2, (f_1\times id_{X_2})^!p_1^! M_1)\\
=
&\underline{Hom}((f_1\times id_{X_2})^*p_2^* L_2, p_1^!f_1^! M_1)
\overset{\eqref{eq:prop:fUpperIsomKun102}}{\simeq} 
f_1^! M_1\boxtimes_S   \underline{Hom} ( L_2, q_2^!\mathbbold{1}_S) 
=
f_1^! M_1\boxtimes_S  M_2.
\end{split}
\end{align}
One can check that the map~\eqref{eq:prop:fUpperIsomKun103} agrees with the map~\eqref{eq:prop:fUpperIsomKun100}, and the result follows.
\end{proof}

\subsubsection{}
We summarize the relative K\"unneth formulas we have obtained in Propositions~\ref{Kun_*},~\ref{thm:relKuHom} and~\ref{thm:fUpperIsomKun}, extending Theorem~\ref{th:kun}:
\begin{theorem}
\label{th:kun_rel}
Let $S$ be a scheme and let $f_1:X_1\to Y_1$, $f_2:X_2\to Y_2$ be two $S$-morphisms. Denote by $f:X_1\times_S X_2\to Y_1\times_S Y_2$ be their product. 
Let $\mathbf{T}$ be a motivic triangulated category. For $i=1,2$, consider objects $L_i\in\mathbf{T}(X_i)$ and $M_i,N_i\in\mathbf{T}_c(Y_i)$. Then the following maps in Theorem~\ref{th:kun}
\begin{align}
\label{eq:th_lower*_rel}
f_{1\ast}L_1\boxtimes_S f_{2\ast}L_2
\xrightarrow{\eqref{eq:th_lower*}}
f_\ast(L_1\boxtimes_S L_2)
\end{align}
\begin{align}
\label{eq:th_upper!_rel}
f^!_1M_1\boxtimes_S f^!_2M_2
\xrightarrow{\eqref{eq:th_upper!}}
f^!(M_1\boxtimes_S M_2)
\end{align}
\begin{align}
\label{eq:th_hom_rel}
\underline{Hom}(M_1, N_1)\boxtimes_S \underline{Hom}(M_2, N_2)
\xrightarrow{\eqref{eq:th_hom}}
\underline{Hom}(M_1\boxtimes_S M_2, N_1\boxtimes_S N_2)
\end{align}
are such that
\begin{enumerate}
\item If for $i=1,2$, $f_i$ is universally strongly locally acyclic relatively to $L_i$ , then then map~\eqref{eq:th_lower*_rel} is an isomorphism.
\item If for $i=1,2$, both $X_i$ and $Y_i$ are smooth over $S$ and smooth over $k$, and the structure morphism $Y_i\to S$ is universally $M_i$-transversal, then the map~\eqref{eq:th_upper!_rel} is an isomorphism.
\item For $i=1,2$, denote by $q_i:Y_i\to S$ the structure morphism. Then the map~\eqref{eq:th_hom_rel} is an isomorphism if the following conditions hold:
\begin{enumerate}
\item The morphisms $q_1$ and $q_2$ are smooth, and $Y_1$ and $Y_2$ are smooth $k$-schemes.
\item For $i=1,2$, the diagonal morphism $Y_i\to Y_i\times_k Y_i$ is 
$\mathbb D(M_i)\boxtimes_k N_i$-transversal. 
\item For $i=1,2$, $q_i$ is universally $M_i$-transversal and universally $N_i$-transversal.
\end{enumerate}
\end{enumerate}
\end{theorem}

\begin{remark}

By Proposition~\ref{prop:BGB2} below, the universal transversality conditions in Theorem~\ref{th:kun_rel} can be replaced by strong universal local acyclicity conditions.

\end{remark}

\subsubsection{}
Given Corollary~\ref{cor:SSKunn}, we are now ready to define the realtive Verdier pairing in the same way as we have done in Section~\ref{section_vpairing}. We fix a base scheme $S$, and for any morphism $h:X\to S$ we denote $\mathcal{K}_{X/S}=h^!\mathbbold{1}_S$. 

Let $X_1$ and $X_2$ be two smooth $S$-schemes which are also smooth over $k$. We denote by $X_{12}=X_1\times_S X_2$ 
and $p_i:X_{12}\to X_i$ the projections. Let $L_i\in\mathbf{T}_c(X_i)$ and
let $q_i:X_i\to S$ be the structure map for $i=1,2$. Let $c:C\to X_{12}$ and $d:D\to X_{12}$ be two morphisms, and let $E=C\times_{X_{12}}D$ with $e:E\to X_{12}$ the canonical morphism. For $i=1,2$ denote by $c_i=p_i\circ c:C\to X_i$ and $d_i=p_i\circ d:D\to X_i$. Assume that for $i=1,2$, $q_i$ is universally $L_i$-transversal. By Corollary~\ref{cor:SSKunn}, we produce the following map in the same way as the map~\eqref{e*KE}:
\begin{equation}
\label{e*KES}
c_*\underline{Hom}(c_1^*L_1, c_2^!L_2)\otimes d_*\underline{Hom}(d_2^*L_2, d_1^!L_1)
\to 
e_*\mathcal{K}_{E/S}
\end{equation}
\begin{definition}
\label{def:verdier_pairing_rel}
In the situation above, for two maps $u:c_1^*L_1\to c_2^!L_2$ and $v:d_2^*L_2\to d_1^!L_1$, we define the \textbf{relative Verdier pairing}
\begin{equation}
\label{verdier_pairing_rel}
\langle u,v\rangle:\mathbbold{1}_{E}\to\mathcal{K}_{E/S}
\end{equation}
obtained by adjunction from the composition
\begin{align}
\begin{split}
&\mathbbold{1}_{X_{12}}
\to
c_*\mathbbold{1}_C\otimes d_*\mathbbold{1}_D
\to 
c_*\underline{Hom}(c_1^*L_1,c_1^*L_1)\otimes d_*\underline{Hom}(d_2^*L_2,d_2^*L_2)\\
&\xrightarrow{u_*\otimes v_*}
c_*\underline{Hom}(c_1^*L_1,c_2^!L_2)\otimes d_*\underline{Hom}(d_2^*L_2,d_1^!L_1)
\xrightarrow{\eqref{e*KES}}
e_*\mathcal{K}_{E/S}.
\end{split}
\end{align}
\end{definition}

\subsubsection{}
The relative Verdier pairing satisfies a proper covariance similar to Proposition~\ref{proper_pf} (see \cite[Theorem 3.3.2]{YZ18}). It satisfies an additivity property along distinguished triangles similar to Theorem~\ref{th:add_trace}.

\subsubsection{}
We can define the relative characteristic class as in Definition~\ref{def:cc}:
\begin{definition}
\label{def:cc_rel}
Let $X$ be a smooth $S$-scheme which is also smooth over $k$. Let $M\in\mathbf{T}_c(X)$ be such that the structure morphism $X\to S$ is universally $M$-tranversal. The Verdier pairing in Definition~\ref{def:verdier_pairing_rel} in the particular case where $C=D=X_1=X_2=X$ and $L_1=L_2=M$ is a pairing
\begin{align}
\langle\ ,\ \rangle:Hom(M,M)\otimes Hom(M,M)\to H_0(X/S),
\end{align}
where $H_0(X/S)=Hom_{\mathbf{T}_c(X)}(\mathbbold{1}_X,\mathcal{K}_{X/S})$.
For any endomorphism $u\in Hom(M,M)$, the \textbf{relative characteristic class} of $u$ 
is defined as the element $C_{X/S}(M,u)\coloneqq\langle u,1_M \rangle\in H_0(X/S)$.
The relative characteristic class of $M$ is the characteristic class of the identity $C_{X/S}(M)\coloneqq C_{X/S}(M,1_M)$.
\end{definition}
The following result is similar to Proposition~\ref{prop:cc_endo_euler}:
\begin{proposition}
\label{smooth_zerocycle_rel}
Let $X$ be a smooth $S$-scheme which is also smooth over $k$, with tangent bundle $L_{X/S}$. Then we have $C_{X/S}(\mathbbold{1}_X(n))=(\langle-1\rangle_X)^ne(L_{X/S})$.
\end{proposition}

\subsubsection{}
We now establish a link between the relative characteristic class and the (absolute) characteristic class via specialization of cycles (\cite[4.5.1]{DJK}). Let $S$ be a smooth $k$-scheme and let $s:\operatorname{Spec}(k)\to S$ be a $k$-rational point. Let $f:X\to S$ be a smooth morphism, and form the Cartesian square
\begin{align}
\begin{gathered}
\xymatrix{ 
X_s \ar[d]_-{f_s} \ar[r]^-{s_X} \ar@{}[rd]|{\Delta} & X \ar[d]^-{f}\\
k \ar[r]^-{s} & S.
}
\end{gathered}
\end{align}
Then by \cite[2.2.7(1)]{DJK}, there is a canonical specialization map induced by the base change
\begin{align}
\label{eq:bc_spec}
\Delta^*:H_0(X/S)\to H_0(X_s/k).
\end{align}

\begin{proposition}
\label{prop:cc_rel_spec}
Let $M\in\mathbf{T}_c(X)$ be such that $f:X\to S$ is universally $M$-tranversal, and denote by $M_s\coloneq M_{|X_s}=s_X^*M\in\mathbf{T}_c(X_s)$. Let $u\in Hom(M,M)$ be an endomorphism of $M$, and denote by $u_s\in Hom(M_s,M_s)$ the induced endomorphism of $M_s$. Then via the specialization map~\eqref{eq:bc_spec}, the relative characteristic class $C_{X/S}(M,u)\in H_0(X/S)$ in Definition~\ref{def:cc_rel} and the characteristic class $C_{X_s}(M_s,u_s)\in H_0(X_s/k)$ in Definition~\ref{def:cc} satisfy
\begin{align}
\Delta^*C_{X/S}(M,u)=C_{X_s}(M_s,u_s).
\end{align}
\end{proposition}

\proof

By Corollary~\ref{cor:Hom_upper*_1} and Lemma~\ref{lem:BGB3} below, the following canonical map is an isomorphism:
\begin{align}
s_X^*\mathbb{D}_{X/S}(M)
\xrightarrow{\sim}
\mathbb{D}_{X_s}(M_s).
\end{align}
The result then follows from the following commutative diagram:
$$
\resizebox{\textwidth}{!}{
\xymatrix{ 
\mathbbold{1}_{X_s} \ar@{=}[d] \ar[r]^-{u_s} & \underline{Hom}(M_s,M_s) \ar[r]^-{} & \mathbb{D}_{X_s}(M_s)\otimes M_s \ar[r]^-{\sim} & M_s\otimes\mathbb{D}_{X_s}(M_s) \ar[r]^-{\epsilon_{M_s}} & \mathcal{K}_{X_s}\\
s_X^*\mathbbold{1}_{X} \ar[r]^-{s_X^*u} & s_X^*\underline{Hom}(M,M) \ar[r]^-{} \ar[u]^-{} & s_X^*\mathbb{D}_{X/S}(M)\otimes s_X^*M \ar[r]^-{\sim} \ar[u]^-{\wr} & s_X^*M\otimes s_X^*\mathbb{D}_{X/S}(M) \ar[r]^-{s_X^*\epsilon_{M}} \ar[u]^-{\wr} & s_X^*\mathcal{K}_{X/S} \ar[u]^-{\wr}.
}
}
$$
\endproof

\begin{example}
\begin{enumerate}
\item
Assume that $S$ is a smooth $k$-scheme of dimension $n$. For $\mathbf{T}_c=\mathbf{DM}_{cdh,c}$, we have $H_0(X/S)\simeq CH_n(X)$ is the Chow group of $n$-cycles over $X$ (up to $p$-torsion), and the specialization map~\eqref{eq:bc_spec} is Fulton's specialization map of algebraic cycles (\cite[Section 10.1]{Ful}). By Proposition~\ref{prop:cc_rel_spec}, the relative characteristic class of a motive can be seen as an $n$-cycle spanned by a family of $0$-cycles given by the characteristic classes of its fibers. It is conjectured that the relative characteristic class is related to the the relative characteristic cycle (see \cite[Conjecture 3.2.6]{YZ18}).

\item
If we work in $\mathbf{T}_c=\mathbf{SH}_{c}$ and apply the $\mathbb{A}^1$ regulator map with values in the Milnor-Witt spectrum (see~\ref{section:rr_eva}), then we get a quadratic refinement of the previous case, namely a family of Chow-Witt $0$-cycles (\cite[Example 4.5.5]{DJK}).
\end{enumerate}
\end{example}

\subsection{Purity, local acyclicity and transversality}
\label{section:purity_la_trans}
In this section we clarify the link between the notions of purity, local acyclicity and transversality conditions. We also study the relation with the Fulton-style specialization map in \cite{DJK}.
\subsubsection{}
Let $\mathbf{T}$ be a motivic triangulated category which satisfies the condition~\ref{resol} in \ref{par:devissage}. Let $f:X\to S$ be a morphism of schemes with $S$ smooth over $k$. Recall from~\ref{recall_pur_trans} that if $\mathbbold{1}_S$ is $f$-pure, then for any $B\in\mathbf{T}(X)$, $f$ is $B$-transversal if and only if for any $C\in\mathbf{T}_c(S)$, $B\boxtimes_kC$ is $\Gamma_f$-pure, where $\Gamma_f:X\to X\times_kS$ is the graph of $f$. This amounts to say that the following canonical map is an isomorphism:
\begin{equation}
\label{eq:pur_trans_dual}
B\otimes f^*C\otimes f^*Th_S(-L_{S/k})
\to
\mathbb{D}(\mathbb{D}(B)\otimes\mathbb{D}(f^!C)).
\end{equation}

\subsubsection{}
\label{rem:pur_trans_dual}
The map~\eqref{eq:pur_trans_dual} is always an isomorphism when $C$ is dualizable.

\begin{recall}
Consider a Cartesian square of schemes
\begin{equation}
\begin{gathered}
  \xymatrix{
    Y \ar[r]^-{q} \ar[d]_-{g} \ar@{}[rd]|{\Delta} & X \ar[d]^-{f}\\
    T \ar[r]_-{p} & S
  }
\end{gathered}
\end{equation}
with $p$ a lci morphism. For $K\in\mathbf{T}(X)$, there is a canonical map called \emph{refined purity transformation}
\begin{equation}
\label{eq:ref_pur_trans} 
q^*K\otimes g^*Th_T(L_p)\to q^!K
\end{equation}
(\cite[Definition 4.2.5]{DJK}). We say that $K$ is \textbf{$\Delta$-pure} if the map~\eqref{eq:ref_pur_trans} is an isomorphism.
\end{recall}

\begin{lemma}[{\cite[Lemma B.3]{BG}}]
\label{lem:BGB3}
Consider a Cartesian square of schemes
\begin{equation}
\begin{gathered}
  \xymatrix{
    Y \ar[r]^-{i'} \ar[d]_-{g} \ar@{}[rd]|{\Delta} & X \ar[d]^-{f}\\
    T \ar[r]_-{i} & S
  }
\end{gathered}
\end{equation}
with $i$ a regular immersion. Let $K\in\mathbf{T}(X)$.
\begin{enumerate}
\item If both $S$ and $T$ are smooth over $k$ and $f$ is $K$-transversal, then $K$ is $\Delta$-pure.
\item If $\mathbbold{1}_S$ is $i$-pure and $f$ is strongly locally acyclic relatively to $K$, then $K$ is $\Delta$-pure.
\end{enumerate}
\end{lemma}
\proof
\begin{enumerate}
\item We have an isomorphism
\begin{equation}
\label{eq:trans_pure}
K\otimes f^*i_*\mathbbold{1}_T\otimes f^*Th_S(-L_{S/k})
\xrightarrow{\sim}
i'_*(i'^!K\otimes g^*Th_T(-L_{T/k}))
\end{equation}
given by the composition
\begin{align*}
&K\otimes f^*i_*\mathbbold{1}_T\otimes f^*Th_S(-L_{S/k})
\simeq
\mathbb{D}(\mathbb{D}(K)\otimes\mathbb{D}(f^!i_*\mathbbold{1}_T))\\
\simeq 
&\mathbb{D}(\mathbb{D}(K)\otimes\mathbb{D}(i'_*g^!\mathbbold{1}_T))
\simeq
i'_*\mathbb{D}(\mathbb{D}(i'^!K)\otimes\mathbb{D}(g^!\mathbbold{1}_T))\\
\simeq
&i'_*(i'^!K\otimes g^*Th_T(-L_{T/k}))
\end{align*}
where the first isomorphism follows from the transversality condition, and the last isomorphism follows from~\ref{rem:pur_trans_dual}. The result follows by applying the functor $i'^*$ to the map~\eqref{eq:trans_pure}.

\item Without loss of generality we can assume that $i$ is a regular closed immersion.
We have a canonical map
\begin{equation}
\label{eq:la_pure_trans}
K\otimes f^*i_*i^!\mathbbold{1}_S
\xrightarrow{}
i'_*i'^!K
\end{equation}
given by the composition
\begin{align}
\begin{split}
K\otimes f^*i_*i^!\mathbbold{1}_S
&\simeq
K\otimes i'_*g^*i^!\mathbbold{1}_S
\simeq
K\otimes i'_*g^*Th_T(L_i)\\
&\simeq
i'_*(i'^*K\otimes g^*Th_T(L_i))
\xrightarrow{}
i'_*i'^!K
\end{split}
\end{align}
where the second isomorphism comes from purity, and the last map is the functor $i'_*$ applied to the map~\eqref{eq:ref_pur_trans}. It follows from the local acyclicity and the localization sequence that the map~\eqref{eq:la_pure_trans} is an isomorphism, which implies that $K$ is $\Delta$-pure.

\end{enumerate}
\endproof

\begin{proposition}[{\cite[Theorem B.2]{BG}}]
\label{prop:BGB2}
Assume that $k$ is a perfect field. Let $f:X\to S$ be a morphism of schemes which factors through an open subscheme $S_0$ of $S$ which is smooth over $k$. Let $K\in\mathbf{T}(X)$. We consider Cartesian squares of the form
\begin{equation}
\label{diag:Cart_BGB2}
\begin{gathered}
  \xymatrix{
    Y \ar[r]^-{q} \ar[d]_-{g} \ar@{}[rd]|{\Delta} & X \ar[d]^-{f}\\
    T \ar[r]_-{p} & S.
  }
\end{gathered}
\end{equation}
Then the following statements hold:
\begin{enumerate}
\item If for any Cartesian square~\eqref{diag:Cart_BGB2} with $p$ smooth,
$g$ is strongly locally acyclic relatively to $q^*K$, then $f$ is $K$-transversal.
\item If for any Cartesian square~\eqref{diag:Cart_BGB2} with $p$ smooth,
$g$ is $q^*K$-transversal, then $f$ is strongly locally acyclic relatively to $K$.
\end{enumerate}

\end{proposition}

\proof
By hypothesis, $f$ factors through a morphism $f_0:X\to S_0$. It is easy to see that $f$ is strongly locally acyclic relatively to $K$ (resp. $K$-transversal) if and only if $f_0$ is strongly locally acyclic relatively to $K$ (resp. $K$-transversal). Therefore by working with $f_0$ we can assume that $S=S_0$ is smooth over $k$.

We consider a Cartesian square of schemes
\begin{equation}
\begin{gathered}
  \xymatrix{
    X' \ar[r]^-{r} \ar[d]_-{f'} \ar@{}[rd]|{\Delta} & X \ar[d]^-{f}\\
    S' \ar[r]^-{s} & S
  }
\end{gathered}
\end{equation}
where $S'$ is smooth over $k$, and a fortiori $s$ is a lci morphism. 
Then we have the following diagram
\begin{equation}
\label{diag:BGB3}
\begin{gathered}
  \xymatrix{
    K\otimes f^*s_*\mathbbold{1}_{S'}\otimes f^*Th_S(-L_{S/k}) \ar[r]^-{(a)} \ar[d]_-{(b)} & \mathbb{D}(\mathbb{D}(K)\otimes\mathbb{D}(f^!s_*\mathbbold{1}_{S'})) \ar[d]^-{(c)} \\
    r_*r^*K\otimes f^*Th_S(-L_{S/k}) \ar[d]_-{(d)} & r_*\mathbb{D}(\mathbb{D}(r^!K)\otimes\mathbb{D}(f'^!\mathbbold{1}_{S'})) \\
   r_*\mathbb{D}(\mathbb{D}(r^*K\otimes f'^*Th_{S'}(-L_{S'/k}))\otimes\mathbb{D}(f_0^!\mathbbold{1}_{S'})) \ar[ru]_-{(e)} & 
  }
\end{gathered}
\end{equation}
where
\begin{itemize}
\item The map (a) is~\eqref{eq:pur_trans}.
\item The map (b) is~\eqref{ExDelta**}.
\item The map (c) is an isomorphism induced by base change and projection formula.
\item The map (d) is an isomorphism deduced from~\eqref{eq:pur_trans} by~\ref{rem:pur_trans_dual}.
\item The map (e) is~\eqref{eq:ref_pur_trans}.
\end{itemize}
One can check that the diagram is commutative. 

\begin{enumerate}
\item If $f$ is strongly locally acyclic relatively to $K$, then the map (b) above is an isomorphism. If the local acyclicity condition holds after any smooth base change, then by Lemma~\ref{lem:BGB3} $K$ is $\Delta$-pure and the map (e) above is an isomorphism, which implies that the map (a) above is an isomorphism. It follows from strong devissage that $f$ is $K$-transversal. 

\item If $f$ is $K$-transversal, then the map (a) above is an isomorphism. If the transversality condition holds after any smooth base change, then by Lemma~\ref{lem:BGB3} $K$ is $\Delta$-pure and the map (e) above is an isomorphism, which implies that the map (b) above is an isomorphism. It follows from strong devissage that $f$ is strongly locally acyclic relatively to $K$. 

\end{enumerate}
\endproof

\subsubsection{}
We now show that the strong local acyclicity is equivalent to the following property, similar to the one in \cite[Proposition 8.11]{Sai}:
\begin{definition}\label{def:la-CDY}

Let $f\colon X\to S$ be a morphism of schemes and let $K\in {\mathbf T}(X)$. We say that the morphism $f$ is \textbf{strongly fibrewise locally acyclic} relatively to $K$ if the following condition holds:

For any schemes $S^{\prime}$ and $S^{\prime\prime}$ smooth over a finite extension $k^{\prime}$ of $k$, and for any cartesian diagram of schemes
\begin{align}\label{eq:def:la}
\begin{gathered}
\xymatrix{
X^{\prime\prime}\ar[r]^-h\ar[d]_-g& X^{\prime}\ar[r]^-{p'}\ar[d]^{f'}&X\ar[d]^-f \\
S^{\prime\prime}\ar[r]^-i&S^{\prime}\ar[r]^-p&S
}
\end{gathered}
\end{align}
where $p$ is proper and generically finite and $i$ is a closed immersion, and for any $L\in\mathbf{T}(S')$ the following composition map is an isomorphism:
\begin{equation}
\label{eq:def:la1}
h^\ast(p'^\ast K)\otimes g^\ast i^!L
\xrightarrow{\eqref{Ex*!}}
h^\ast(p'^\ast K)\otimes h^!f'^*L
\xrightarrow{\eqref{eq:nat_upper*!_upper!}}
h^!(p'^\ast K\otimes f'^*L)
\end{equation}

We say that $f$ is \textbf{universally strongly fibrewise locally acyclic} relatively to $K$ if 
the condition above holds after any base change.
\end{definition}

The proof of the following statement is inspired by \cite[Proposition 7.2]{CD2}:
\begin{proposition}
\label{prop:la_fib}
Let $f\colon X\to S$ be a morphism of schemes and let $K\in {\mathbf T}(X)$. 
Then $f$ is universally strongly locally acyclic relatively to $K$ if and only if $f$ is universally strongly fibrewise locally acyclic relatively to $K$.

\end{proposition}

\proof

If $f$ is universally strongly locally acyclic relative to $F$, then by an argument similar to Lemma~\ref{lem:BGB3} we know that $f$ is universally strongly fibrewise locally acyclic relatively to $F$.

Now assume that $f$ is universally strongly fibrewise locally acyclic relatively to $F$. Then it follows that for any Cartesian diagram
\begin{equation}
\begin{gathered}
  \xymatrix{
    Y_U \ar[r]^-{j_Y} \ar[d]_-{f_U} & Y' \ar[r]^-{q_Y} \ar[d]^-{f_{T'}} & Y \ar[d]^-{f_{T}}\\
    U \ar[r]^-{j} & T' \ar[r]^-{q} &T
  }
\end{gathered}
\end{equation}
where $f_T:Y\to T$ is a base change of $f$, $q$ is proper and generically finite, $j$ is an open immersion with complement a strict normal crossing divisor and both $U$ and $T'$ are smooth over a finite extension $k'$ of $k$, and any $M\in\mathbf{T}(U)$, the following canonical map is an isomorphism:
\begin{equation}
\label{eq:fiberw_la_map}
K_{|Y}\otimes f_T^*q_*j_*M
\xrightarrow{\eqref{ExDelta**}}
q_{Y*}j_{Y*}(j_Y^*q_Y^*K_{|Y}\otimes f_U^*M).
\end{equation}

We need to prove that for any Cartesian square
\begin{equation}
\begin{gathered}
  \xymatrix{
    Y_V \ar[r]^-{r_Y} \ar[d]_-{f_V} & Y \ar[d]^-{f_{T}}\\
    V \ar[r]^-{r} & T
  }
\end{gathered}
\end{equation}
where $f_T:Y\to T$ is a base change of $f$, and any $N\in\mathbf{T}(V)$, the following canonical map is an isomorphism:
\begin{equation}
K_{|Y}\otimes f_T^*r_*N
\xrightarrow{\eqref{ExDelta**}}
r_{Y*}(r_Y^*K_{|Y}\otimes f_V^*N).
\end{equation}
We prove this claim by noetherian induction on $V$. By the existence of a compactification, we can factor the morphism $r:V\to T$ above as an open immersion with dense image $j_1:V\to\bar{V}$ followed by a proper morphism $p:\bar{V}\to T$.
\begin{equation}
\begin{gathered}
  \xymatrix{
    Y_V \ar[r]^-{j_{1Y}} \ar[d]_-{f_V} & Y_{\bar{V}} \ar[d]^-{f_{\bar{V}}}\\
    V \ar[r]^-{j_1} & \bar{V}
  }
\end{gathered}
\end{equation}
Since $p$ is proper, it suffices to prove that under the assumption~\ref{resol1} in \ref{par:devissage} (respectively under the assumption~\ref{resol2}, for every prime number $l$ different from the characteristic of $k$), there exists a non-empty open immersion $j_2:V'\to V$ such that the following canonical map is an isomorphism (resp. is an isomorphism with coefficients in $\mathbb{Z}_{(l)}$):
\begin{equation}
K_{|Y_{\bar{V}}}\otimes f_{\bar{V}}^*j_{1*}j_{2*}j_2^*N
\xrightarrow{\eqref{ExDelta**}}
j_{1Y*}(j_{1Y}^*K_{|Y_{\bar{V}}}\otimes f_V^*j_{2*}j_2^*N).
\end{equation}
We can assume that $\bar{V}$ is integral. By the assumption~\ref{resol1} (resp. by de Jong-Gabber alteration (\cite[X. Theorem 2.1]{ILO})), there exists a proper surjective morphism $h:\tilde{V}\to\bar{V}$ which is birational (respectively generically flat, generically finite with degree prime to $l$) such that $\tilde{V}$ is smooth over $k$ (resp. smooth over a finite extension $k'$ of $k$ of degree prime to $l$), and such that the inverse image of $\bar{V}\setminus V$ in $\tilde{V}$ is a strict normal crossing divisor. Let $j_2:V_0\to V$ be an open immersion such that the induced morphism $h_{V_0}:V_0':=h^{-1}(V_0)\to V_0$ is an isomorphism (resp. is finite lci with trivial virtual tangent bundle), and form the following Cartesian squares:
\begin{equation}
\begin{gathered}
  \xymatrix{
    V_0' \ar[r]^-{j_2'} \ar[d]_-{h_{V_0}} & V' \ar[r]^-{j_1'} \ar[d]^-{h_V} & \tilde{V} \ar[d]^-{h}\\
    V_0 \ar[r]^-{j_2} & V \ar[r]^-{j_1} & \bar{V}.
  }
\end{gathered}
\end{equation}
It follows that $j_{2*}j_2^*N$ is equal to $j_{2*}h_{V_0*}h_{V_0}^*j_2^*N=h_{V*}j'_{2*}j'^*_{2}h_{V}^*N$ (resp. is a direct summand of $h_{V*}j'_{2*}j'^*_{2}h_{V}^*N$ by \cite[Proposition 2.2.2]{EK}). By~\eqref{eq:fiberw_la_map} the canonical map
\begin{equation}
K_{|Y}\otimes f_{\bar{V}}^*j_{1*}h_{V*}h_V^*N
\xrightarrow{\eqref{ExDelta**}}
j_{1Y*}(j_{1Y}^*K_{|Y_{\bar{V}}}\otimes f_V^*h_{V*}h_V^*N)
\end{equation}
is an isomorphism. Given the localization distinguished triangle
\begin{equation}
h_{V*}i'_{2!}i'^!_{2}h_{V}^*N
\xrightarrow{}
h_{V*}h_V^*N
\xrightarrow{}
h_{V*}j'_{2*}j'^*_{2}h_{V}^*N
\xrightarrow{}
h_{V*}i'_{2!}i'^!_{2}h_{V}^*N[1]
\end{equation}
where $i'_2:Z'_0\to V'$ is the reduced closed complement of $j'_1$, and since $h_{V}i'_{2}$ factors through $h(Z'_0)$ which is a proper closed subscheme of $V$, we conclude using the induction hypothesis.
\endproof

\subsubsection{}
We now establish a link between strong local acyclicity and the Fulton-style specialization map defined in \cite{DJK}. Consider Cartesian squares of schemes
\begin{equation}
\label{diag:DJK_sp}
\begin{gathered}
  \xymatrix{
    X_Z \ar[r]^-{i_X} \ar[d]_-{f_Z} & X \ar[d]^-{f} & X_U \ar[l]_-{j_X} \ar[d]^-{f_U}\\
    Z \ar[r]^-{i} & S & U \ar[l]_-{j}
  }
\end{gathered}
\end{equation}
where $i$ is a regular closed immersion and $j$ the complementary open immersion.
Assume that the Euler class $e(-L_i)$, namely the Euler class of the normal bundle of $i$, is zero. Recall from \cite[4.5.6]{DJK} that for $A\in\mathbf{T}(X)$ there is a natural transformation of functors
\begin{equation}
\label{eq:DJK_sp}
i_{X*}(i_X^*A\otimes f_Z^*Th_Z(L_i))
\xrightarrow{}
j_{X!}j_X^!A.
\end{equation}
By Lemma~\ref{lem:BGB3}, if $\mathbbold{1}_S$ is $i$-pure and if $f$ is strongly locally acyclic relative to $A$, the following refined purity transformation~\eqref{eq:ref_pur_trans} is an isomorphism:
\begin{equation}
i_X^*A\otimes f_Z^*Th_Z(L_i)
\xrightarrow{}
i_X^!A.
\end{equation}
In particular from the construction of~\eqref{eq:DJK_sp} we deduce the following:
\begin{corollary}
\label{cor:sp_split}
We use the notations in~\eqref{diag:DJK_sp} and assume that
\begin{enumerate}
\item The Euler class $e(-L_i)$ is zero.
\item $\mathbbold{1}_S$ is $i$-pure.
\item $f$ is strongly locally acyclic relative to $A\in\mathbf{T}(X)$.
\end{enumerate}
Then there exists a canonical map
\begin{equation}
\label{eq:sp_ij}
i_{X*}i_X^!A
\xrightarrow{}
j_{X!}j_X^!A
\end{equation}
such that the canonical map $i_{X*}i_X^!A\xrightarrow{ad'_{(i_{X!},i_X^!)}}A$ agrees with the following composition
\begin{equation}
i_{X*}i_X^!A
\xrightarrow{\eqref{eq:sp_ij}}
j_{X!}j_X^!A
\xrightarrow{ad'_{(j_{X!},j_X^!)}}
A.
\end{equation}
Equivalently, there exists  a canonical map
\begin{equation}
\label{eq:sp_ji}
j_{X*}j_X^*A
\xrightarrow{}
i_{X*}i_X^*A
\end{equation}
such that the canonical map $A\xrightarrow{ad_{(i_{X}^*,i_{X*})}}i_{X*}i_X^*A$ agrees with the following composition
\begin{equation}
A
\xrightarrow{ad_{(j_{X}^*,j_{X*})}}
j_{X*}j_X^*A
\xrightarrow{\eqref{eq:sp_ji}}
i_{X*}i_X^*A.
\end{equation}
\end{corollary}
Note that by Lemma~\ref{lem:BGB3}, a similar result holds for the transversality condition when $Z$ and $S$ are smooth over a field.

\begin{remark}
Corollary~\ref{cor:sp_split} is a consequence of the strong local acyclicity and therefore gives a criterion to detect it. In the usual derived category of \'etale sheaves, the \emph{vanishing cycle formalism} (\cite[XIII]{SGA7}) gives an insightful interpretation of this phenomena: the failure of local acyclicity is precisely described by the stalks of the vanishing cycle complex. We do not know how to realize such a picture in the motivic world.
\end{remark}

\end{document}